\documentclass[10pt]{article}
\usepackage{}
\usepackage{amsfonts}
\usepackage{amssymb,amsmath,amsthm, amsfonts,mathrsfs}
\usepackage{color}

\usepackage {enumitem}
\usepackage{indentfirst}
\def\sqr#1#2{{\vcenter{\vbox{\hrule height.#2pt
              \hbox{\vrule width.#2pt height#1pt \kern#1pt \vrule width.#2pt}
          \hrule height.#2pt}}}}
\def\signed #1{{\unskip\nobreak\hfil\penalty50
          \hskip2em\hbox{}\nobreak\hfil#1
          \parfillskip=0pt \finalhyphendemerits=0 \par}}
\def\endpf{\signed {$\sqr69$}}

\def\esssup{\mathop{\rm esssup}}
\def\essinf{\mathop{\rm essinf}}
\def\max{\mathop{\rm max}}

\def\exp{\mathop{\rm exp}}
\def\inf{\mathop{\rm inf}}

\def\tr{\hbox{\rm tr$\,$}}
\font\tenbb=msbm10 \font\sevenbb=msbm7 \font\fivebb=msbm5
\newfam\bbfam
\scriptscriptfont\bbfam=\fivebb \textfont\bbfam=\tenbb
\scriptfont\bbfam=\sevenbb

\newtheorem{lemma}{Lemma}[section]
\newtheorem{remark}{Remark}[section]

\newtheorem{theorem}{Theorem}[section]

\newtheorem{definition}{Definition}[section]
\newtheorem{proposition}{Proposition}[section]

\makeatletter
   
   \@addtoreset{equation}{section}
\makeatother

\begin{document}

\title{ {Stochastic representation for solutions of a system of coupled HJB-Isaacs equations with integral-partial operators}
}
\author{Sheng Luo \\
{\small Department of Applied Mathematics, The Hong Kong Polytechnic University, Hung Hom, Kowloon,
Hong Kong.}\\
{\small \textit{E-mail: 22038664r@connect.polyu.hk}}\\
Wenqiang Li\thanks{%
W. Li  acknowledges the financial support from the NSF of P.R. China (No. 12101537, 12271304) and the Doctoral Scientific Research Fund of Yantai University (No. SX17B09).
}\\
{\small School of Mathematics and Information Sciences, Yantai University,
Yantai 264005, P. R. China.}\\
{\small \textit{E-mail: wenqiangli@ytu.edu.cn}}\\
Xun Li\thanks{%
 X. Li acknowledges the financial support from the Research Grants Council of Hong Kong under grants (No.~15216720, 15221621, 15226922), PolyU 1-ZVXA, and 4-ZZLT.
 }\\
{\small Department of Applied Mathematics, The Hong Kong Polytechnic University, Hung Hom, Kowloon,
Hong Kong.}\\
{\small \textit{E-mail:li.xun@polyu.edu.hk  }}\\
Qingmeng Wei{ \thanks{ 
Q. Wei  acknowledges the financial support from the NSF of Jilin Province for Outstanding Young Talents
(No. 20230101365JC) and the NSF of P.R. China (No. 11971099).
}\ \thanks{%
Corresponding author.}}\\
{\small School of Mathematics and Statistics, Northeast Normal University,
Changchun 130024, P. R. China.}\\
{\small \textit{E-mail: weiqm100@nenu.edu.cn}}}
\date{\today}
\maketitle

\begin{abstract} In this paper, we focus on the stochastic representation of a system of coupled Hamilton-Jacobi-Bellman-Isaacs (HJB-Isaacs (HJBI), for short) equations which is  in fact a system of coupled  Isaacs' type integral-partial differential equation.
For this, we introduce an associated  zero-sum stochastic differential game, where the state  process is  described by a classical stochastic differential equation (SDE, for short) with jumps, and the cost functional of recursive type is  defined by a new type of  backward stochastic differential equation (BSDE, for short)   with two Poisson random measures, whose wellposedness and  {a} prior estimate as well as the comparison theorem are investigated  for the first time. One of the {Poisson} random measures $\mu$ appearing in the  SDE  and the BSDE stems from the integral term of the HJBI equations; the other random measure in BSDE is introduced to link the coupling factor of the HJBI equations.
We show through an extension of the  dynamic programming principle  that the lower value function of this game problem is  the viscosity solution of the system of our coupled HJBI equations. The uniqueness of the viscosity solution is also obtained in a space of continuous functions satisfying  certain growth condition. In addition,  also the upper  value function of the game is shown to be  the solution of the associated   system of coupled Issacs' type of integral-partial differential equations.
As a byproduct, we obtain the existence of the value for the game problem under the well-known Isaacs' condition.

\end{abstract}


\newpage

\section{ Introduction}

Let  $T>0$ be an arbitrarily fixed finite time horizon, $U$ and $V$ be two compact metric spaces, $m\geq 2$ be an integer, and let $ \mathbf{M}=\{1,2,\cdots, m\}$.
In this paper, we focus on  the following  system of integral-partial differential equations :
\begin{equation}
\left\{
\begin{aligned}
&\! \frac{\partial W_{i}}{\partial t}(t,x)+\sup\limits_{u\in U}\underset{v\in V}{%
\mathop{\rm inf}}\Big\{b_{i} (t,x,u,v )DW_{i}(t,x)
  +\frac{1}{2}\tr\big(\sigma _{i}\sigma _{i}^{\ast
}(t,x,u,v)D^{2}W_{i}(t,x)\big)+B^i_{u,v}W_i(t,x) \\
 & \!    +f_{i}\big(t,x,\mathbf{W}(t,x),D
W_{i}(t,x)\sigma _{i}(t,x, u,v), C^i_{u,v}W_i(t,x),u,v\big)\Big\}=0,\ (t,x,i)\in [ 0,T)\times \mathbb{R}^n\times \mathbf{M},\\
&W_{i}(T,x)=\Phi_i (x),\quad (x,i)\in \mathbb{R}^n\times \mathbf{M},
\end{aligned}%
\right.   \label{1}
\end{equation}%
where $\mathbf{W}(t,x)=\big(W_1,\cdots,W_m\big)(t,x)$ and
\begin{equation}\label{ee01}
\begin{aligned}
&B^i_{u,v}W_i(t,x):= \int_{E}\Big[W_i\big(t,x+\gamma_i(t,x,u,v,e)\big)-W_i(t,x)-DW_i(t,x)\gamma_i(t,x,u,v,e)\Big]\nu(de),\\
&C^i_{u,v}W_i(t,x):= \int_{E} \Big[W_i\big(t,x+\gamma_i(t,x,u,v,e)\big)-W_i(t,x)\Big]\rho(x,e)\nu(de),\  (t,x,i)\in \lbrack 0,T]\times \mathbb{R}^n\times \mathbf{M}.
\end{aligned}
\end{equation}
The involved deterministic functions $b_i$, $\sigma_i$, $\gamma_i$, $f_i$, $\Phi_i $, $1\leq i\leq m$, and $\rho$   will be specified in Section 3.
Notice that system \eqref{1} consists of $m$ equations,   and these equations are coupled with each other through the vector $\mathbf{W}(t,x)$ in the functions $f_i$.
We call such an equation \eqref{1}   \emph{system of coupled  Hamilton-Jacobi-Bellman-Isaacs (HJBI) equations  with integral-partial operators}, its solution is the vector-valued function $\mathbf{W}(\cdot,\cdot)$.

Our aim is to study the stochastic representation for the viscosity solution of the system \eqref{1}. For this, we introduce
 an associated two-player  zero-sum stochastic differential game  defined on a Wiener-Poisson space. Precisely, the dynamics is described
  by the doubly controlled stochastic differential equations (SDEs)  with jumps, see  \eqref{fequ3.1}.
 The payoff functional is of recursive type and defined by   the first component  of the solution of  a backward {SDE (BSDE)}  with two Poisson random measures, see \eqref{equ3.1}.
The main result of this paper is to show that the related lower value functions   of the game problems  give  a probabilistic interpretation to   the solution of system \eqref{1}, that is,  the lower value functions solve    system \eqref{1} in  the sense of viscosity solution. To this end, we combine the theory of zero-sum stochastic differential games with the tool of forward-backward SDEs (FBSDEs) with jumps. But beyond the combination of tools, the proof of the main results and necessary auxiliary lemmas turn out to be rather tricky, subtle and technical.

The theory of FBSDEs with jumps has been turned out to be a powerful and useful tool to address a variety of problems including, among others,  optimal control problems \cite{LP, LW-SIAM, TL}, differential game problems \cite{BHL, LW-AMO}, investment portfolio management problems \cite{LM2021}, and also  the weak solution of integral-partial differential equations (PDEs).
 Barles, Buckdahn and Pardoux \cite{BBP}  associated a system of parabolic integral-PDEs with a type of FBSDE with jumps (without controls), where the backward equation was multi-dimensional.
 When  the backward equation was one-dimensional,  Buckdanhn, Hu and Li \cite{BHL} studied a class of Issacs' type integral-PDEs via associated zero-sum stochastic differential  games linked with FBSDE with jumps.   So it seems quite natural and interesting to consider how the stochastic  controlled system with jumps  link  with the Isaacs' type PDEs, when the backward equation in \cite{BHL} is multi-dimensional.
The existing literature  about  the  probabilistic interpretation  of the coupled system of parabolic PDEs including  Pardoux, Pardeilles and Rao \cite{PPR-1997} without control, Buckdahn, Hu \cite{BH-2010} with only one player and our recent work \cite{LLW-2020} with two players, has strongly inspired us for   the transformation technique.
Following the framework of Elliot-Kalton type ``strategy against control" setting for  differential games, we establish a finite set of lower value functions of our zero-sum stochastic differential games and prove that the vector-valued function, constructed via the elements in this set, is a viscosity solution of the system of coupled HJBI equations with integral-partial operators \eqref{1}.
Compared with the existing references, the contribution of our paper includes the following {three} aspects.

(i) Firstly we  introduce  a new type of   BSDEs with two Poisson random measures, which is due to the appearance of the integral term and the coupling factor in system \eqref{1}. For such backward equations,  we carry out the study of   the existence and the uniqueness of the solution, a useful estimate as well as the comparison theorem, which play an important role in our study.

  (ii) It is the first time to study \eqref{1} as the system of Isaacs' type HJB equations.
  We use a  transformation approach to link the coupled   system  \eqref{1} with a forward-backward {double} controlled system, where the backward part is in fact multi-dimensional.
The transformation brings us some new aspects. Firstly,  two Poisson random measures appear in our controlled system,
 one has the task to give a stochastic interpretation to the integral in the HJB-Isaacs system,
 the other one is to overcome the coupling in  system \eqref{1}. Secondly,  the dynamic programming principle (DPP) has  the strong form, i.e., the classical DPP involving stopping times. The proof of the  strong DPP demands a series of delicate  results, whose proof  are far from  routine methods.

 Our research generalizes  previous works in a great extent. Without  the integral and the two control  terms,    system \eqref{1} reduces to  the  case considered in  Pardoux, Pardeilles and Rao \cite{PPR-1997}.
  Buckdahn, Hu \cite{BH-2010} studied  system \eqref{1} without the integral term  and only involving one control term either $u$ or $v$. Compared with {the} recent work \cite{LLW-2020}, system \eqref{1} involves the integral operators, and this is novel.
 When $m=1$, our work is reduced that by  Buckdahn, Hu and Li \cite{BHL}  (seeing also Li, Peng \cite{LP}  involving only one control).

 (iii)  If another form $\inf\limits_{{v\in V}}\sup\limits_{u\in U}\{\cdots\}$ instead of $\sup\limits_{u\in U}\inf\limits_{{v\in V}}\{\cdots\}$ appears in system \eqref{1}, the upper value functions  of our game problem can also provide the probabilistic interpretation for this case in the sense of  the viscosity solution.
Additionally, when assuming the well-known Isaacs' condition, our game problem has a value.

Our paper is organized as follows. In Section 2, we introduce the underlying Wiener-Poisson space and provide some preliminaries for the BSDEs which involve two Poisson random measures.   The framework of the corresponding zero-sum stochastic differential games is introduced in Section 3. Moreover, also the properties of the upper and the lower value functions are   investigated in this section.
  In Section 4, we obtain the existence of the viscosity solution of the system \eqref{1}, as well as the uniqueness in a class of  continuous functions.
Finally, we provide  the proof of the strong DPP (Theorem \ref{SDPP}) in the Appendix.

\section{Preliminaries}
%

\subsection{Some notations}

Let  $\mathbf{L}=\mathbf{M}-\{m\}$ be an index set equipped with its $\sigma$-field $\mathcal{L}$ of all subsets of $\mathbf{L}$, where
$\mathbf{M}=\{1,2,\cdots
,m\}$, and  $m\geq 2$ is a given integer.
We introduce the underlying probability space $(\Omega,\mathcal{F},P)$, which is defined as the completed product space of the following three spaces:

\medskip

\noindent    (i) $(\Omega_1,{\mathcal{F}_1},P_1)$  is a classical Wiener space, that is,  $\Omega_1=C_0(\mathbb{R};\mathbb{R}^d)$ is the set of all continuous functions from $\mathbb{R}$ to $\mathbb{R}^d$ with value $0$ at time zero; $\mathcal{F}_1$ is the completed Borel $\sigma$-field on $\Omega_1$; $P_1$ is the Wiener measure such that the canonical processes $B_s(\omega)=\omega(s)$, $s\in\mathbb{R}_+$, $\omega\in\Omega_1$ and $B_{-s}(\omega)=\omega(-s)$, $s\in\mathbb{R}_+$, $\omega\in\Omega_1$, are two independent $d$-dimensional Brownian motions.
Denote the filtration on $(\Omega_1,{\mathcal{F}_1},P_1)$ by $\mathbb{F}^B=(\mathcal{F}_s^B)_{s\geq 0}$ with
   $$\mathcal{F}_s^B:=\sigma\big\{B_r, r\in(-\infty,s]\big\}\vee \mathcal{N}_{P_1},$$
where $\mathcal{N}_{P_1}$ is the collection of $P_1$-null sets.

\medskip

\noindent    (ii)  $(\Omega_2,{\mathcal{F}_2},P_2)$ is a Poisson space. Let $\Omega_2$ be the set of all $\mathbf{L}$-valued point functions  $p$  defined on  a countable subset  $D_{p}$ of the real line $\mathbb{R}$.  The counting measure  $N$ on $\mathbb{R}\times \mathbf{L}$ at $p\in\Omega_2$ is defined as follows
$$N\big(p,(s,t]\times \Delta\big):=\sharp\big\{r\in D_{p}\cap (s,t]\mid p(r)\in\Delta\big\},\ \Delta\in\mathcal{L},\ s,t\in\mathbb{R},\ s<t,$$
where $\sharp$ denotes the cardinal number of the set. We identify the point function $p$ with $N(p,\cdot)$ and denote by $\mathcal{F}_2$  the smallest $\sigma$-field on $\Omega_2$ such that the coordinate mapping $p\rightarrow N\big(p,(s,t]\times \Delta\big),$ $\Delta\in\mathcal{L},$ $ s,t\in\mathbb{R},$ $ s<t$, is measurable. For fixed $\lambda>0$, we  consider a probability measure $P_2$ on $(\Omega_2,\mathcal{F}_2)$ under which the  coordinate measure $N$ becomes a Poisson random measure with the following compensator  $$\hat{N}\big((s,t]\times \{l\}\big):=(t-s)\lambda\sum
\limits_{n=1}^{m-1}\delta_n(l)\footnote{Here, $\delta_n(\cdot)$ represents Dirac measure over $\mathbf{L}$, for $n=1,\cdots, m-1$.}=\lambda(t-s),\ l\in \mathbf{L}.$$
 Then, the Poisson martingale measure $\big\{\tilde{N}\big((s,t]\times \{l\}\big)\big\}_{
s\leq t}$ is given by
 $$\big\{\tilde{N}\big((s,t]\times \{l\}\big)\big\}_{
s\leq t}:=\big\{(N-\hat N)\big((s,t]\times \{l\}\big)\big\}_{s\leq t}=\big\{N\big((s,t]\times \{l\}\big)-\lambda (t-s)\big\}_{s\leq t},\ \text{for\ all}\ l\in \mathbf{%
L}.$$ Notice that the processes $\big\{\tilde{N}\big((s,t]\times \{l\}\big)\big\}_{s\leq t
},\ 1\leq l\leq m-1$, are independent.
The filtration $(\mathcal{F}_t^N)_{t\geq 0}$ generated by the coordinate measure $N$ is defined as follows
$$\mathcal{F}_t^N:=\Big(\bigcap\limits_{s>t}\dot{\mathcal{F}}_s^N\Big)\vee\mathcal{N}_{P_2},\ t\geq 0,$$
where $\dot{\mathcal{F}}_t^N:=\sigma\big\{N\big((s,r]\times \Delta\big),-\infty<s\leq r\leq t, \Delta\in \mathcal{L}\big\},\ t\geq 0.$

\medskip

\noindent    (iii)  $(\Omega_3,{\mathcal{F}_3},P_3)$ too is    a Poisson space. Let $\Omega_3$ be the set of all point functions $\bar{p}: D_{\bar{p}} \rightarrow E$, where $D_{\bar{p}}$ is a countable subset of the real line $\mathbb{R}$, and $E:=\mathbb{R}^{l}\backslash\{0\}$ is equipped with its Borel $\sigma$-field  $\mathcal{B}(E)$. Denote by $\mathcal{F}_3$ the smallest $\sigma$-field on $\Omega_3$ such that the coordinate mapping $\bar{p}\rightarrow \mu\big(\bar{p},(s,t]\times \Delta\big),$  $ \Delta\in\mathcal{B}(E),$ $s,t\in\mathbb{R},$ $ s<t$, defined analogously to the coordinate mapping $N$ in $(\Omega_2,\mathcal{F}_2,P_2)$,  is measurable.
%
%
   The probability $P_3$ is considered on the measurable space $(\Omega_3,\mathcal {F}_3)$ to ensure that the coordinate measure $\mu(\bar{p},dtde)$ is  a Poisson random measure with the compensator $\hat{\mu}(dtde):=dt\nu(de)$ and the process $\big\{\tilde{\mu}\big((s,t]\times A\big)\big\}_{s\leq t}:=\big\{(\mu-\hat{\mu})\big((s,t]\times A\big)\}_{s\leq t}$ is a martingale, for all  $A\in\mathcal {B}(E)$ satisfying $\nu(A)<\infty$. Here $\nu$ is supposed to be a $\sigma$-finite measure on $\big(E,\mathcal {B}(E)\big)$ with $\displaystyle \int_E\big(1\wedge|e|^2\big)\nu(de)<\infty.$
The filtration generated by $\mu$ on  $(\Omega_3,{\mathcal{F}_3},P_3)$ is denoted by $(\mathcal{F}_t^\mu)_{t\geq 0}$ with
   $\mathcal{F}_t^\mu:=\Big(\bigcap\limits_{s>t}\dot{\mathcal{F}}_s^\mu\Big)\vee\mathcal{N}_{P_3}$, where
$$\dot{\mathcal{F}}_t^{\mu}:=\sigma\Big\{\mu\big((s,r]\times \Delta\big),-\infty<s\leq r\leq t, \Delta\in \mathcal{B}(E)\Big\},\ t\geq 0.$$

\noindent Then, by putting  $\Omega:=\Omega_1\times \Omega_2\times \Omega_3$, $\mathcal {F}:=\mathcal {F}_1\otimes \mathcal {F}_2\otimes \mathcal {F}_3$, $P:=P_1\otimes P_2\otimes P_3$  with $\mathcal {F}$ being  completed with respect to $P$, we introduce the underlying probability space $(\Omega,\mathcal{F},P)$. We denote by  $\mathbb{F}=(\mathcal{F}_t)_{t\geq 0}$   the filtration on  $(\Omega,{\mathcal{F}},P)$, where  $\mathcal{F}_t:=   \mathcal{F}_t^B\otimes\mathcal{F}_t^N\otimes\mathcal{F}_t^\mu$, $t\geq0$, augmented by all $P$-null sets.




\medskip

For $0\leq t\leq s\leq T,i\in \mathbf{M}$ and $l\in \mathbf{L}$, we denote  $%
N_{s}:=  N\big((0,s]\times \mathbf{L}\big),$ $N_{s}(l):=  N\big((0,s]\times \{l\}\big), $ $\tilde{N}%
_{s}(l):=  N_{s}(l)-\lambda s,$
and we introduce an $\mathbf{M}$-valued Markov process
 $
N^{t,i}$ as
$$N_{s}^{t,i}:= \Big(i+\sum\limits_{l=1}^{m-1}lN\big((t,s]\times \{l\}\big)\Big)%
\mbox{mod}(m),$$
where $(j) \mbox{mod}%
(m)$ is identified with $j^{\prime }\in \mathbf{M}$ such that $%
j-j^{\prime }$ is a multiple of $m$, for  any given natural $j\geq 1$.

\medskip

Finally, we introduce the following spaces of the processes which will be used in our paper. Let $t\in [0,T]$. Denote by $\mathcal{P}_{t,T}$ the $\sigma $-field of $(\mathcal{%
F}_{t})_{t\geq 0}$-predictable subsets of $\Omega \times [t,T]. $
\noindent \begin{itemize}[leftmargin=*]

 \item ${\mathcal{S}}_\mathbb{F}^{2}(t,T;\mathbb{R}^n)\!\!:= \!\!  \Big\{\psi \!\mid \!\psi :\Omega\times [t,T]
\rightarrow \mathbb{R}^n$ is  an $(\mathcal{F}_{t})_{t\geq 0}$-adapted c\`{a}dl%
\`{a}g process  with $\mathbb{E}\Big[\sup\limits_{s\in [t,T]}|\psi_{s}|^{2}\Big]<\infty
\Big\}; \!\!\!\!\!\!\!\!$

 \item $\mathcal{M}_\mathbb{F}^{2}(t,T;\mathbb{R}^d)\!\!:= \!\!  \Big\{\varphi \!\mid \!\varphi :\Omega\times[t,T]  \rightarrow \mathbb{R}^d$ is   $(\mathcal{F}_{t})_{t\geq 0}$-%
progressively measurable with $\displaystyle \mathbb{E}  \int_{t}^{T} |\varphi _{t}|^{2}dt< \infty \Big\};\!\!\!\!\!\!\!\!$

\item $\big[L^{2}(\mathcal{P}_{t,T} \otimes \mathcal{L})\big]^{m-1}\!\!:= \!\! \Big\{  H\mid H=\big(H(1),\cdots,H(m-1)\big):\Omega \times \lbrack t,T]\times \mathbf{L}\rightarrow \mathbb{R}^{m-1}$ is $%
\mathcal{P}_{t,T}
\otimes \mathcal{L}$-measurable\\
\mbox{\hskip4cm }  with
  $\| H\|^2 _{\lbrack L^{2}(%
\mathcal{P}_{t,T} \otimes \mathcal{L})]^{m-1}}\!\!:= \!\!  \mathbb{E} \displaystyle\int_{t}^{T}\sum\limits_{l=1}^{m-1}H_{s}(l)^{2}ds  < \infty \Big\};$

\item $\big[L^{2}(\mathbf{L};\mathbb{R})\big]^{m-1}:= \Big\{ H(\cdot)\mid H(\cdot):\mathbf{L}\rightarrow \mathbb{R}$\Big\}.
Moreover, for a map $H(\cdot):\mathbf{L}\rightarrow\mathbb{R}$, we introduce the
norm  $\|
H\| _{[L^{2}(\mathbf{L};\mathbb{R})]^{m-1}}:=\Big[\sum\limits_{l=1}^{m-1}H(l)^{2}\Big]^{\frac{1}{2}};$

\item $\mathcal {K}_\nu^2(t,T;\mathbb{R})\!\!:= \!\!    \Big\{K \mid K:\Omega\times [t,T]\times E \rightarrow \mathbb{R}$ is $\mathcal {P}_{t,T}\otimes\mathcal{B}(E)$-measurable with \\
    $\mbox{\hskip5cm }\|
K\|^2_{\mathcal {K}_\nu^2(t,T;\mathbb{R})}:=\mathbb{E} \displaystyle\int_t^T\!\! \int_E|K_s(e)|^2\nu(de)ds  <\infty \Big\};$

\item  $L^2_\nu\big(E,\mathcal {B}(E);\mathbb{R}\big)\!\!:= \!\!    \Big\{K \!\mid\! K\! :\!  E \! \rightarrow\!  \mathbb{R}$ is $\mathcal{B}(E)$-measurable with $\|
K\|^2_{L_\nu^2(E,\mathcal {B}(E);\mathbb{R})}\! :=\!  \displaystyle \int_E|K(e)|^2\nu(de) < \infty \Big\}.\! \! \! \! \! $

\end{itemize}
For any $t\in[0,T]$, we put
$$\begin{aligned}
& \mathcal{S}^{2}[t,T]:= {\mathcal{S}}%
^{2}_{\mathbb{F}}(t,T;\mathbb{R}^n)\times \lbrack L^{2}(\mathcal{P}_{t,T}\otimes \mathcal{L}%
)]^{m-1}\times \mathcal{M}^{2}_{\mathbb{F}}(t,T;\mathbb{R}^{d})\times \mathcal {K}_\nu^2(t,T;\mathbb{R}),\\
& \mathcal{M}^{2}[t,T]:= {\mathcal{M}}%
^{2}_{\mathbb{F}}(t,T;\mathbb{R}^n)\times \lbrack L^{2}(\mathcal{P}_{t,T}\otimes \mathcal{L}%
)]^{m-1}\times \mathcal{M}^{2}_{\mathbb{F}}(t,T;\mathbb{R}^{d})\times \mathcal {K}_\nu^2(t,T;\mathbb{R}).\end{aligned}$$

\subsection{\protect\large BSDEs with  two Poisson random measures}
 In this subsection, we introduce and study a type of   BSDE with two Poisson random measures,
\begin{equation}
\left\{%
\begin{aligned}
\!dY_{s}=&-g\big(s,Y_{s},H_{s},Z_{s},K_s\big)ds+Z_{s}dB_{s}+\sum\limits_{l=1}^{m-1} H_{s}(l)d%
\tilde{N}_{s}(l)+\int_EK_s(e)\tilde\mu(ds,de),\ s\in [ 0,T],  \\
\!Y_T=& \ \xi,
\end{aligned}
\right.  \label{equ2.1}
\end{equation}%
where  the coefficient $g:\Omega
\times [ 0,T]\times \mathbb{R}\times [L^{2}(\mathbf{L};\mathbb{R})]^{m-1}\times \mathbb{R}^{d}\times L^2_\nu\big(E,\mathcal {B}(E);\mathbb{R}\big) \rightarrow \mathbb{R}$ is $\mathcal{P}_{0,T}$%
-measurable for every fixed $(y,h,z,k)\in \mathbb{R} \times [L^{2}(\mathbf{L};\mathbb{R})]^{m-1}\times \mathbb{R}^{d}\times L^2_\nu\big(E,\mathcal {B}(E);\mathbb{R}\big)$ and satisfies:
\begin{description}
\item[$(\mathbf{A1})$] (i) There exists some constant $L_g\geq 0$ such that, $P$%
-a.s., for all $t\in [ 0,T]$, $(y_{i},h_{i},z_{i},k_i)\in \mathbb{R}\times  [L^{2}(\mathbf{L};\mathbb{R})]^{m-1}\times
\mathbb{R}^{d}\times L^2_{\nu}\big(E,\mathcal {B}(E);\mathbb{R}\big)$, $i=1,2$,
\begin{equation*}\begin{aligned}
&\big|g(t,y_{1},h_{1},z_{1},k_1)-g(t,y_{2},h_{2},z_{2},k_2)\big|\\
 \leq&
L_g\big(|y_{1}-y_{2}|+\| h_{1}-h_{2}\|_{[L^{2}(\mathbf{L};\mathbb{R})]^{m-1}}+|z_{1}-z_{2}|+\| k_{1}-k_{2}\|_{L^2_{\nu}(E,\mathcal {B}(E);\mathbb{R})}\big);
\end{aligned}\end{equation*}%
(ii) $ \displaystyle \mathbb{E}\Big[ \int_{0}^{T}|g(s,0,0,0,0)|^{2}ds \Big]<+\infty .$
\end{description}

Similar to Lemma 2.3 in Tang, Li \cite{TL}, we get the following martingale representation theorem under our filtered  probability space $(\Omega, \mathcal {F},\mathbb{F},P)$.

\begin{lemma}\sl\label{MRT}
(Martingale Representation Theorem) Let $M(\cdot)$ be an $\mathbb{F}$-adapted $\mathbb{R}$-valued square integrable martingale. Then there exists a triple
$\big(p(\cdot,\cdot), q(\cdot), r(\cdot,\cdot)\big)\in  \big[L^2(\mathcal {P}_{0,T}\otimes \mathcal {L})\big]^{m-1} \times \mathcal {M}_\mathbb{F}^2(0,T;\mathbb{R}^d)\times \mathcal {K}_{\nu}^2(0,T;\mathbb{R})$, such that
$$M(t)=M(0)+\int_0^tq(s)dB_s+\int_0^t\sum_{l=1}^{m-1}p(s,l)d\tilde N_s(l)+\int_0^t\int_Er(s,e)\tilde\mu(ds,de).$$

\end{lemma}

Based on  Lemma \ref{MRT}, we have the following well-posedness of BSDE  with  jumps  \eqref{equ2.1}.

\begin{theorem}
\label{l1}\sl Under the condition $(\mathbf{A1})$, for any random variable $\xi
\in L^{2}(\Omega ,\mathcal{F}_{T},P;\mathbb{R})$, BSDE with jumps %
\eqref{equ2.1} admits a  unique adapted solution
$
\big(Y,H,Z,K\big)\in \mathcal{S}^{2}[0,T].
$
\end{theorem}

\begin{proof}
For any given $\big(\bar Y,\bar H,\bar Z,\bar K\big)\in \mathcal{M}^{2}[0,T]$,
we denote
$$\displaystyle M_t:= \mathbb{E}\Big[\xi+\int_0^Tg\big(s,\bar Y_{s},\bar H_{s},\bar Z_{s},\bar K_s\big)ds\mid \mathcal {F}_t \Big],\quad t\in[0,T].$$
Obviously,  $M$ is an $\mathbb{F}$-adapted square integrable martingale. From Lemma \ref{MRT},
there exists a triple of $ (H,Z,K )\in  \big[L^2(\mathcal {P}_{0,T}\otimes \mathcal {L})\big]^{m-1}\times \mathcal {M}_\mathbb{F}^2(0,T;\mathbb{R}^d)\times  \mathcal {K}_{\nu}^2(0,T;\mathbb{R})$ satisfying
$$M_t=M_0+\int_0^t Z_sdB_s+\int_0^t\sum_{l=1}^{m-1}H_s(l)d\tilde N_s(l)+\int_0^t\int_E K_s(e)\tilde \mu(ds,de),\quad t\in[0,T].$$
Further, we get
\begin{equation}\label{equ-M}M_t=M_T-\int_t^T Z_sdB_s-\int_t^T\sum_{l=1}^{m-1}H_s(l)d\tilde N_s(l)-\int_t^T\int_E K_s(e)\tilde \mu(ds,de),\quad t\in[0,T].\end{equation}

For some fixed  constant  $b>0$ (which will be specified later), we  consider the equivalent norm of   the Banach space  {$\mathcal{M}^{2}[0,T]$:}
$$\begin{aligned}&\|\big(Y,H,Z,K\big)\|_b:=   \Big[  \mathbb{E}  \int_{0}^{T}e^{bt} \Big(|Y_{t}|^{2}+\lambda  \sum\limits_{l=1}^{m-1}|H_{t}(l)|^{2}+|Z _{t}|^{2}  + \int_E |K_t(e)|^2\nu(de)\Big)dt\Big]^\frac12.
 \end{aligned}$$
 %
By setting  $\displaystyle Y_t:= \mathbb{E}\Big[\xi+\int_t^Tg\big(s,\bar Y_{s},\bar H_{s},\bar Z_{s},\bar K_s\big)ds \mid\mathcal {F}_t \Big]$, $t\in[0,T]$, and using \eqref{equ-M}, we have
\begin{equation}\begin{aligned}
&Y_t=\xi+\int_t^Tg\big(s,\bar Y_{s},\bar H_{s},\bar Z_{s},\bar K_s\big)ds-\int_t^T Z_sdB_s\\
&\qquad\qquad\qquad-\int_t^T\sum_{l=1}^{m-1}H_s(l)d\tilde N_s(l)-\int_t^T\int_E K_s(e)\tilde \mu(ds,de),\quad t\in[0,T],
\end{aligned}\end{equation}
which  defines a mapping $\Lambda:\mathcal{M}
^{2}[0,T]\rightarrow\mathcal{M}
^{2}[0,T]$  with  $\Lambda(\bar Y,\bar H,\bar Z,\bar K)=(Y,H,Z,K)$.

Next, we show that the mapping $\Lambda$ is contractive under the norm $\|\cdot\|_b$.
For $(\bar Y^i,\bar H^i,\bar Z^i,\bar K^i)\in \mathcal{M}
^{2}[0,T]$, $i=1,2,$ we denote
 $$\begin{aligned}
 & (Y^i,H^i,Z^i,K^i)=\Lambda(\bar Y^i,\bar H^i,\bar Z^i,\bar K^i),\ i=1,2,\\
 &(\hat Y,\hat H,\hat Z,\hat K)=(Y^1-Y^2,H^1-H^2,Z^1-Z^2,K^1-K^2),\\
 &(\tilde{Y},\tilde{H},\tilde{Z},\tilde{K})=(\bar Y^1-\bar Y^2,\bar H^1-\bar H^2,\bar Z^1-\bar Z^2,\bar K^1-\bar K^2).\end{aligned}$$
Applying It\^o's formula to $e^{bs}|\hat Y_s|^2$, for all $s\in[0,T]$, we get
%
$$\begin{aligned}
&\mathbb{E}\Big[ e^{bs}|\hat Y_s|^2+\int_s^Te^{br}\Big(b|\hat Y_r|^2+\lambda   \sum\limits_{l=1}^{m-1}|\hat H_r(l)|^2 + |\hat Z_r|^2   +\int_E|\hat K_r(e)|^2\nu(de)\Big)dr |\mathcal {F}_s\Big]\\
=&\mathbb{E}\Big[  2\int_s^Te^{br}\hat Y_r\big(g(r,\bar Y_r^1,\bar H_r^1,\bar Z_r^1,\bar K_r^1)-g(r,\bar Y_r^2,\bar H_r^2,\bar Z_r^2,\bar K_r^2)\big)dr|\mathcal {F}_s\Big] \\ %
\leq & \frac{8 L_g^2(1+\lambda)}{\lambda}\mathbb{E}\Big[   \int_s^Te^{br}|\hat Y_r|^2dr|\mathcal {F}_s\Big]\\
&+ \frac{\lambda}{2(1+\lambda)} \mathbb{E}\Big[   \int_s^T e^{br}\Big( |\tilde Y_r |^2+ \sum\limits_{l=1}^{m-1}|\tilde H_r(l) |^2+|\tilde Z_r|^2+\int_E|\tilde K_r(e) |^2\nu(de)\Big)dr|\mathcal {F}_s\Big] ,
\end{aligned}$$
where $L_g$ is the Lipschitz constant in $(\mathbf{A1})$-(i).
Taking $b:=1+ \frac{8 L_g^2(1+\lambda)}{\lambda}$, we get
$$\begin{aligned}
&\mathbb{E}\Big[  \int_0^T e^{br}\Big(|\hat Y_r|^2+ \lambda\sum\limits_{l=1}^{m-1}|\hat H_r(l)|^2 +|\hat Z_r|^2 + \int_E|\hat K_r(e)|^2\nu(de) \Big)dr\Big]\\
\leq & \frac 12 \mathbb{E}\Big[   \int_0^T e^{br}\Big( |\tilde Y_r |^2+\lambda\sum\limits_{l=1}^{m-1}|\tilde H_r(l) |^2+|\tilde Z_r|^2+\int_E|\tilde K_r(e) |^2\nu(de)\Big)dr\Big], %
\end{aligned}$$
which means that   $$\|(\hat Y,\hat H,\hat Z,\hat K)\|_b\leq \frac{1}{\sqrt{2}}\|(\tilde Y,\tilde  H,\tilde  Z,\tilde  K)\|_b.$$
Then there exists a  unique fixed point $(Y,H,Z,K)\in\mathcal{M}
^{2}[0,T] $ such that
$(Y,H,Z,K)=\Lambda(Y,H,Z,K)$. Therefore, the BSDE with jumps \eqref{equ2.1} has a unique solution  $(Y,H,Z,K)\in\mathcal{M}
^{2}[0,T]$.
Furthermore, it is easy to check $Y\in {\mathcal{S}}_\mathbb{F}
^{2}(0,T;\mathbb{R})$ by using the Burkholder-Davis-Gundy inequality.

\end{proof}

For  $g:\Omega \times \lbrack 0,T]\times \mathbb{R}\times [L^{2}(\mathbf{L};\mathbb{R})]^{m-1}\times \mathbb{R}^{d}\times L^2_\nu\big(E,\mathcal {B}(E);\mathbb{R}\big) \rightarrow
\mathbb{R}$ satisfying $(\mathbf{A1})$,   we suppose that the
driver  $g_{i}$, $i=1,2$, has the form
\begin{equation*}
g_{i}\big(s,y,h,z,k\big)=g\big(s,y,h,z,k\big)+%
\varphi_{i}(s),\ s\in [0,T],
\end{equation*}%
where $\varphi_i\in \mathcal{M}_\mathbb{F}^{2}(0,T;\mathbb{R})$.
Denote by $(Y^{1},H^{1},Z^{1},K^1)$ and $(Y^{2},H^{2},Z^{2},K^2)$ the solution of BSDE \eqref{equ2.1} with the data $(\xi _{1},g_{1})$ and $(\xi _{2},g_{2})$,
respectively. Then, we have the following  useful estimate.

\begin{lemma}
\label{l2} \sl Under $(\mathbf{A1})$, for all  $\beta >2L_g+4L_g^2+\frac2\lambda  L_g^2+1$, there exists  some constant  $C>0$ such that the difference of $(Y^{1},H^{1},Z^{1},K^1)$ and $%
(Y^{2},H^{2},Z^{2},K^2)$ satisfy
\begin{equation*}
\begin{array}{l}
\displaystyle |Y_{t}^{1}-Y_{t}^{2}|^{2}+ \mathbb{E}\Big[\int_{t}^{T}e^{\beta
(s-t)}\Big(|Y_{s}^{1}-Y_{s}^{2}|^{2}+\lambda\sum%
\limits_{l=1}^{m-1}|H_{s}^{1}(l)-H_{s}^{2}(l)|^{2} +|Z_{s}^{1}-Z_{s}^{2}|^{2}\\
\displaystyle \hskip7cm+\int_E|K_{s}^{1}(e)-K_{s}^{2}(e)|^{2}\nu(de)\Big)ds\mid\mathcal{F}_{t}\Big] \\
\displaystyle \leq C\mathbb{E}\Big[e^{\beta (T-t)}|\xi _{1}-\xi _{2}|^{2}+
\int_{t}^{T}e^{\beta (s-t)}|\varphi _{1}(s)-\varphi _{2}(s)|^{2}ds\mid\mathcal{F%
}_{t}\Big],\text{ \ P-}a.s.,\text{ }t\in \lbrack 0,T].
\end{array}%
\end{equation*}%

\end{lemma}

\begin{proof}
For convenience, we set $(\hat Y,\hat H,\hat Z,\hat K):=(Y^1-Y^2,H^1-H^2,Z^1-Z^2,K^1-K^2)$.
For some $\beta>0$, applying It\^o's formula to $e^{\beta(s-t)}|\hat Y_{s}|^{2}$, we get
$$\begin{aligned}
&  |\hat Y_t|^2+ \mathbb{E}\Big[\int_t^T e^{\beta (r-t)}\Big(\beta |\hat Y_r|^2+\lambda \sum\limits_{l=1}^{m-1}|\hat H_r(l)|^2 +|\hat Z_r|^2+\int_E|\hat K_r(e)|^2 \nu(de)\Big)dr\mid \mathcal {F}_t\Big]\\
=& \mathbb{E}\Big[e^{\beta (T-t)}|\xi _{1}-\xi _{2}|^{2} + 2\int_t^Te^{\beta (r-t)} \hat Y_r \Big(g(r,Y_r^1,H_r^1,Z_r^1,K_r^1)-g(r, Y_r^2, H_r^2, Z_r^2,K_r^2)\Big)dr\\
&\quad  +2\int_t^Te^{\beta (r-t)} \hat Y_r  \big(\varphi_1(r)-\varphi_2(r)\big) dr\mid \mathcal {F}_t\Big]\\
%
\leq &\mathbb{E}\Big[e^{\beta (T-t)}|\xi _{1}-\xi _{2}|^{2}\mid\mathcal{F}_{t}\Big]+\mathbb{E}\Big[\int_t^Te^{\beta (r-t)} |\varphi_1(r)-\varphi_2(r)|^2 dr\mid\mathcal{F}_{t}\Big]\\
&+(2L_g+2L_g^2+\frac 2\lambda L_g^2+1)\mathbb{E}\Big[\int_t^Te^{\beta (r-t)} |\hat Y_r|^2\mid \mathcal {F}_t \Big] \\
&+\frac12\mathbb{E}\Big[\int_t^Te^{\beta (r-t)} \Big(\lambda \sum\limits_{l=1}^{m-1}|\hat H_r(l)|^2+|\hat Z_r|^2+\int_E |\hat K_r(e)|^2\nu(de)\Big)dr\mid \mathcal {F}_t \Big].
\end{aligned}$$
Therefore,  for $\beta >2L_g+2L_g^2+\frac 2\lambda L_g^2+1$, we have
$$\begin{aligned}
&  |\hat Y_t|^2+ \mathbb{E}\Big[\int_t^T  e^{\beta (r-t)}\Big(|\hat Y_r|^2 + \lambda\sum\limits_{l=1}^{m-1}|\hat H_r(l)|^2+|\hat Z_r|^2 +\int_E|\hat K_r(e)|^2 \nu(de)\Big)dr\mid \mathcal {F}_t\Big]\\
\leq &C\mathbb{E}\Big[e^{\beta (T-t)}|\xi _{1}-\xi _{2}|^{2} +\int_t^Te^{\beta (r-t)}  |\varphi_1(r)-\varphi_2(r)|^2 dr\mid\mathcal{F}_{t}\Big].
\end{aligned}$$
\end{proof}

Next, we study the comparison theorem for BSDE with jumps \eqref{equ2.1}. For this,
let $$\Psi:\Omega\times [0,T]\times \mathbb{R}\times \mathbb{R}^{m-1}\times \mathbb{R}^d\times \mathbb{R}\rightarrow\mathbb{R}$$ be
$\mathcal {P}_{0,T}\otimes\mathcal {B}(\mathbb{R})\otimes\mathcal {B}(\mathbb{R}^{m-1})\otimes\mathcal {B}(\mathbb{R}^d)\otimes\mathcal {B}(\mathbb{R})$-measurable, and satisfy
\begin{description}
\item[$(\mathbf{A2})$] (i) For all $t\in[0,T]$, $(y_i,h_i,z_i,k_i)\in \mathbb{R}\times \mathbb{R}^{m-1}\times \mathbb{R}^d\times \mathbb{R}$, $i=1,2,$ for some constant $L_\Psi\geq 0$, $P$-a.s.,
    $$|\Psi(t,y_1,h_1,z_1,k_1)-\Psi(t,y_2,h_2,z_2,k_2)|\leq L_\Psi\big(|y_1-y_2|+|h_1-h_2|+|z_1-z_2|+|k_1-k_2|\big);$$

 {\rm(ii)} $\Psi(\cdot,0,0,0,0)\in \mathcal {M}_\mathbb{F}^2(0,T;\mathbb{R})$;

 {\rm(iii)} There exist  two nonnegative constants $\alpha$ and $\beta$  such that,
  for any $j\in\mathbf{L}$,  $t\in[0,T]$, $(y,z)\in\mathbb{R}\times \mathbb{R}^{d}$,    $h=(h_1,\cdots,h_{m-1}),$ $h'=(h'_1,\cdots,h'_{m-1})\in \mathbb{R}^{m-1}$ with $h_p=h'_p$, $\forall\ p\neq j$, $p\in \mathbf{L}$,  $h_j\geq h'_j$, and $k,\ k'\in\mathbb{R}$ with $k\geq k'$,
    $$\Psi(t,y,h,z,k )-\Psi(t,y,h',z,k')\geq \alpha(h_j-h'_j)+ \beta(k-k').$$
\end{description}
Furthermore, let $\bar\rho:\Omega\times [0,T]\times E\rightarrow\mathbb{R}$ be $\mathcal {P}_{0,T}\otimes \mathcal {B}(E)$-measurable and satisfy $0\leq \bar\rho(t,e)\leq C(1\wedge|e|)$, $t\in[0,T]$, $e\in E$.
For $(t, y,h,z,k)\in  [0,T]\times \mathbb{R}\times[L^{2}(\mathbf{L};\mathbb{R})]^{m-1}\times\mathbb{R}^d\times L_\nu^2\big(E,\mathcal {B}(E);\mathbb{R}\big)$,  we define
\begin{equation}\label{equ-g}g(t, y,h,z,k):= \Psi\Big(t, y,  h(1),\cdots,h(m-1)\ ,z,\int_Ek(e)\bar\rho(t,e)\nu(de)\Big).\end{equation}
Clearly, the mapping $g$ satisfies the assumption $(\mathbf{A1})$ if $\Psi$ satisfies $(\mathbf{A2})$-(i) and $(\mathbf{A2})$-(ii).

\begin{theorem}\label{Com-Th}\sl (Comparison Theorem) Suppose that  $(\mathbf{A2})$ holds and the mapping $g$ is given by \eqref{equ-g}.
Let $\xi$, $\xi'\in L^2(\Omega, \mathcal {F}_T,P;\mathbb{R})$ and the mapping $g'$  satisfy  $(\mathbf{A1})$.
We denote by $(Y,H,Z,K)$ $($resp., $(Y',H',Z',$ $K')$$)$ the unique solution of equation \eqref{equ2.1} associated with $(\xi,g)$ $($resp., $(\xi',g')$$)$. If $\xi\geq \xi'$, $P$-a.s., and for all $(t,y,h,z,k)\in [0,T]\times \mathbb{R}\times [L^{2}(\mathbf{L};\mathbb{R})]^{m-1}\times \mathbb{R}^d\times L^2_\nu(E,\mathcal {B}(E);\mathbb{R})$, it holds
%
$$g(t,y,h,z,k)\geq g'(t,y,h,z,k),\ P\text{-a.s.}$$
Then we have $$Y_t\geq Y'_t,\ P\text{-a.s.},\ \text{for\ all}\ t\in[0,T].$$ And if, in addition, we also assume that
$P(\xi>\xi') > 0$, then $P\{Y_t>Y'_t\} > 0$, $0 \leq t\leq T$. In particular, $Y_0> Y'_0$.

\end{theorem}

\begin{proof}  Setting $(\hat Y,\hat H,\hat Z,\hat K):= (Y-Y',H-H',Z-Z',K-K')$, we know $(\hat Y,\hat H,\hat Z,\hat K)$ satisfies the following linear backward equation with jumps:
\begin{equation}\label{eemm}
\left\{\begin{aligned}
&\! d\hat Y_s=-\Big[g_s^1\hat Y_s+\lambda\sum_{l=1}^{m-1}g_s^{2,l}\hat H_s(l)+g_s^3\hat Z_s+g_s^4\int_E\hat K_s(e)\bar\rho(s,e)\nu(de)+\hat g_s\Big]ds\\
&\!\hskip 1.3cm+\sum\limits_{l=1}^{m-1}\hat H_s(l)d\tilde N_s(l)+\hat Z_sdB_s+\int_E\hat K_s(e)\tilde\mu(ds,de),\ \ s\in[0,T],\\
 &\!\hat Y_T=\xi-\xi',
\end{aligned}\right.\end{equation}
where
$$\begin{aligned}
&g_s^1:= \left\{\begin{aligned}
&\! \frac{g(s,Y_s,H_s,Z_s,K_s)-g(s,Y'_s,H_s,Z_s,K_s)}{\hat Y_s }, \hskip1cm   \mbox{if  } \hat Y_s\neq 0,\\
&\!  0,  \hskip7cm   \mbox{ otherwise},
 \end{aligned}\right.\\
 & g_s^{2,l}:=  \left\{\begin{aligned}
& \!\frac{g(s,Y'_s,\bar{H}^l_s,Z_s,K_s)-g(s,Y'_s,\tilde{H}^l_s,Z_s,K_s)}{\lambda\hat H_s(l)}, \hskip0.7cm  \mbox{ if } \hat H_s(l)\neq 0,\\
&\!  0,   \hskip6.8cm\ \mbox{otherwise},
   \end{aligned}\right.\\
&g_s^3:=  \left\{\begin{aligned}
& \!\frac{g(s,Y'_s,H'_s,Z_s,K_s)-g(s,Y'_s,H'_s,Z'_s,K_s)}{|\hat Z_s|^2}\hat Z_s, \hskip0.5cm  \mbox{ if } \hat Z_s\neq 0,\\
 & \! 0, \hskip7.1cm  \mbox{otherwise},
   \end{aligned}\right.\\
& g_s^4:=  \left\{\begin{aligned}
 &\! \frac{g(s,Y'_s,H'_s ,Z'_s,K_s)-g(s,Y'_s,H'_s,Z'_s,K'_s)}{\int_E\hat K_s(e)\bar\rho(s,e)\nu(de) },  \hskip0.9cm  \mbox{ if } \int_E\hat K_s(e)\bar\rho(s,e)\nu(de)\neq0,\\
& \! 0,   \hskip7.2cm   \mbox{otherwise},
 \end{aligned}\right.\\
 &\hat g_s:=  g(s,Y'_s,H'_s,Z'_s,K'_s)-g'(s,Y'_s,H'_s,Z'_s,K'_s),
 \end{aligned}$$
 with  $l=1,2,\cdots, m-1,$ and
 $$\bar{H}^l_s:=\big( H'_s(1),\cdots,H'_s(l-1),H_s(l),H_s(l+1),\cdots,H_s(m-1)\big),$$
 $$\tilde{H}^l_s:=\big( H'_s(1),\cdots,H'_s(l-1),H'_s(l),H_s(l+1),\cdots,H_s(m-1)\big).$$
Now we introduce the  adjoint equation of BSDE \eqref{eemm} as follows,
\begin{equation}\label{eead}
\left\{\begin{aligned}
&\! dP_s=g_s^1P_sds+ \sum_{l=1}^{m-1}g_s^{2,l}P_{s-}d\tilde{N}_s(l)+g_s^3P_sdB_s+\int_Eg_s^4P_{s-}\bar\rho(s,e) \tilde\mu(ds,de),\ \ s\in[0,T],\\
 &\! P_0=1.
\end{aligned}\right.\end{equation}
Its solution $P(\cdot)$  can be expressed as (refer to Theorem 122 in \cite{Situ})
\begin{equation}\label{2020111401}
P_t=\exp\Big\{\int_0^tg_s^1ds+\int_0^tg_s^3 dB_s-\frac 12 \int_0^t|g_s^3|^2ds  \Big\} \prod\limits_{0< s\leq t}(1+ Q_s^1)\exp\{- Q_t^2\},\quad t\in[0,T],
\end{equation}
with
$$\begin{aligned}
& Q_s^1:=\sum_{l=1}^{m-1}g_s^{2,l}\Delta{N}_s(l)+ \int_Eg_s^4 \bar\rho(s,e) \mu(\{s\},de),\quad s\in[0,T],\\
&Q_s^2:=\int_0^s\lambda\sum_{l=1}^{m-1}g_s^{2,l}ds+\int_0^s \int_Eg_s^4 \bar\rho(s,e) \nu(de)ds,\quad s\in[0,T].\end{aligned}$$
By  $(\mathbf{A2})$-(iii) and $\bar\rho(\cdot,\cdot)\geq 0$, we have
$ Q_s^1\geq 0$, a.s., for $s\in[0,T]$. Combined with \eqref{2020111401}, we conclude that
 $P_t> 0$, $t\in[0,T]$.
By applying It\^o's formula to $\hat Y_sP_s$, we get
\begin{equation}\label{2020111402}
\hat Y_t P_t=\mathbb{E}\Big[(\xi-\xi')P_T+\int_t^T\hat g_sP_sds|\mathcal {F}_t\Big],\quad t\in[0,T].
\end{equation}
 Combined   with  $P_t> 0$, $t\in[0,T]$, we get $Y_t\geq Y'_t$, for all $t\in[0,T]$.
Moreover, it follows from \eqref{2020111402}  that if $P(\xi>\xi') > 0$, then $P\{Y_t>Y'_t\} > 0$, $0 \leq t\leq T$.

\end{proof}

\section{Stochastic Differential Games with Jumps}

 In this section, we formulate the problem of  two-player zero-sum stochastic differential  game  with jumps.
Let   $U$ and $V$ be two  compact metric
spaces, the admissible control sets for the two players are given as
 $$\begin{aligned}
 &\mathcal{U}=\Big\{u:[0,T]\times \Omega \rightarrow U\mid u\mbox{ is\ an }(\mathcal{F}%
_{t})_{t\geq0}{\rm-progressively\ measurable\ process }\Big\}  \mbox{ (for Player I)}, \\
 &\mathcal{V}=\Big\{v:[0,T]\times \Omega \rightarrow V\mid   v\  {\rm is\ an \ } (\mathcal{F}
_{t})_{t\geq0} {\rm-progressively\  measurable\  process}\Big\}  \mbox{ (for Player II)}.  \end{aligned}$$
For every $i\in \mathbf{M}=\{1,2,\cdots,m\}$, let
\begin{equation*}
\begin{aligned}
 b_{i}:\mathbb{[}0,T\mathbb{]}\times \mathbb{R}^n \times U\times V\rightarrow \mathbb{R}^n,\ \sigma _{i}:
[0,T]\times \mathbb{R}^n \times U\times
V\rightarrow \mathbb{R}^{n\times d},\
\gamma_i:[0,T]\times \mathbb{R}^n \times U\times V\times E\rightarrow \mathbb{R}^n
\end{aligned}%
\end{equation*}%
satisfy
\begin{description}
\item[$(\mathbf{H1})$]
(i) For every fixed $x\in \mathbb{R}^n$, $e\in  E$, $b_i(\cdot,x,\cdot,\cdot),\ \sigma_i(\cdot,x,\cdot,\cdot),\ \gamma_i(\cdot,x,\cdot,\cdot,e)$ are continuous in $(t,u,v)$.

(ii) There exists some constant $L>0$ and a mapping $\kappa:E\rightarrow \mathbb{R}^+$ satisfying $\displaystyle \int_E\kappa^2(e)\nu(de)<+\infty$  such that, for all $t\in[0,T]$,  $x_1$, $x_2\in \mathbb{R}^n$, $u\in U$, $v\in V$, $i\in \mathbf{M}$, $e\in E$,
\begin{equation}\nonumber
\begin{aligned}
&|b_i(t,x_1,u,v)-b_i(t,x_2,u,v)|+|\sigma_i(t,x_1,u,v)-\sigma_i(t,x_2,u,v)|\leq L|x_1-x_2|,\\
%
&|\gamma_i(t,x_1,u,v,e)-\gamma_i(t,x_2,u,v,e)|\leq \kappa(e)|x_1-x_2|,\\
& |\gamma_i(t,0,u,v,e)|\leq \kappa(e).
\end{aligned}
\end{equation}
\end{description}

For any given admissible control  $(u,v)\in \mathcal {U}\times \mathcal {V}$, $i\in \mathbf{M}$, the initial stopping time $\tau$ with $0\leq \tau\leq T$, $P$-a.s., and the initial state $\eta\in L^2(\Omega, \mathcal {F}_\tau,P;\mathbb{R}^n)$, we consider the following  forward stochastic system with jumps:
\begin{equation}
\left\{
\begin{aligned}
&\!dX_{s}^{\tau,\eta ,i;u,v}  =     b_{N_{s}^{\tau,i}}(s,X_{s}^{\tau,\eta
,i;u,v},u_{s},v_{s})ds   +\sigma _{N_{s}^{\tau,i}}(s,X_{s}^{\tau,\eta ,i;u,v},u_{s},v_{s})dB_{s}\\
&\!\qquad\qquad\quad   +\int_E\gamma_{N_{s}^{\tau,i}}(s,X_{s-}^{\tau,\eta ,i;u,v},u_{s},v_{s},e)\tilde{\mu}(ds,de),\text{ \ \ }s\in \lbrack \tau,T],    \\
&\! X_{\tau}^{\tau,\eta ,i;u,v}   =  \  \eta .    \end{aligned}%
\right.  \label{fequ3.1}
\end{equation}
Obviously,  under   assumption $(\mathbf{H1})$, the SDE with jumps  \eqref{fequ3.1} has a unique solution $X^{\tau,\eta ,i;u,v}\in {\mathcal{S}}_\mathbb{F}^{2}(\tau,T;\mathbb{R}^n)$. Moreover,  there exists some constant $C>0$ such that, for any $\tau\in[0,T]$, $i\in \mathbf{M}$, $(u,v)\in \mathcal {U}\times \mathcal {V}$, $\eta,$ $\eta'\in L^2(\Omega, \mathcal {F}_\tau,P;\mathbb{R}^n)$,  we have, $P$-a.s.,
  \begin{equation}\label{ee00}
  \begin{aligned}
 \mathbb{E}\Big[\sup\limits_{s\in[\tau,T]}|X_s^{\tau,\eta,i;u,v}-X_s^{\tau,\eta',i;u,v}|^2\mid \mathcal {F}_\tau\Big]\leq C|\eta-\eta'|^2.
 \end{aligned} \end{equation}
 For any $p\geq2$, we have the existence of  some constant $C_p>0$ such that, for any stopping time $\tau$ with $0\leq \tau\leq s \leq T$, $P$-a.s., $i\in \mathbf{M}$, $(u,v)\in \mathcal {U}\times\mathcal {V}$, $\eta\in L^p(\Omega, \mathcal {F}_\tau,P;\mathbb{R}^n)$,
   \begin{equation}\label{ee222}\begin{aligned}
   & \text{(i)} \ \mathbb{E}\Big[\sup\limits_{r\in[\tau,T]}|X_r^{\tau,\eta,i;u,v}|^p\mid \mathcal {F}_\tau\Big]\leq C_p(1+|\eta|^p),\ P\text{-a.s.},\\
 & \text{(ii)} \ \mathbb{E}\Big[\sup\limits_{r\in[\tau,s]}|X_r^{\tau,\eta,i;u,v}-\eta|^p\mid \mathcal {F}_\tau\Big]\leq C_p(s-\tau) (1+|\eta|^p),\ P\text{-a.s.}
 \end{aligned}\end{equation}
 %
For the details of the above estimates, the readers can refer to \cite{BBP, BHL, LW-SPA, YZ}.

\medskip

For every $i\in \mathbf{M}$, let the following mappings
\begin{equation*}
\begin{array}{lll}
&  f_{i}:\mathbb{[}0,T\mathbb{]}\times \mathbb{R}^n\times \mathbb{R}^{m}\times
\mathbb{R}^{d}\times\mathbb{R}\times U\times V\rightarrow \mathbb{R},\ \  \Phi_i :%
\mathbb{R}^n\rightarrow \mathbb{R}, &
\end{array}%
\end{equation*}%
satisfy
\begin{description}
\item[$(\mathbf{H2})$]
(i) For every $(x,a,z,k)\in \mathbb{R}^n\times \mathbb{R}^m\times \mathbb{R}^d\times \mathbb{R}$, ${f}_i(\cdot, x,a,z,k,\cdot,\cdot)$ is continuous in $(t,u,v)$;

(ii) ${f}_i$ and   $\Phi_i $ are Lipschitz continuous in $(x,a,z,k)$ and $x$, respectively, uniformly with  respect to  $(t,u,v)\in [0,T]\times U\times V$;

(iii) For the mapping $\rho:\mathbb{R}^n\times E\rightarrow \mathbb{R}$, there exists a constant $C>0$ such that,
$$\begin{array}{lll}
0\leq\rho(x,e)\leq C(1\wedge|e|),\quad x\in \mathbb{R}^n, \ e\in E,\\
|\rho(x_1,e)-\rho(x_2,e)|\leq {C(1\wedge|e|)|x_1-x_2|},\quad x_1,x_2\in \mathbb{R}^n, \ e\in E.\end{array}$$
\end{description}
\begin{description}
\item[$(\mathbf{H3})$] For fixed $\lambda>0$, there exists two constants $\alpha$ and $\beta$ satisfying  $\alpha\geq 0$ and $ \beta\geq0 $  such that, for any $i,j\in \mathbf{M}$, $i\neq j$, $(t,x,z,u,v)\in \lbrack 0,T]\times \mathbb{R}^n\times \mathbb{R}^{d}\times U\times V$,  and $a,a'\in \mathbb{R}^m$ with $a_p=a'_p$, $\forall p\neq j$,  $a_j\geq a'_j$,  and  $k,k'\in\mathbb{R}$ with $k\geq k'$,
    $$f_i(t,x,a,z,k,u,v)-f_i(t,x,a',z,k',u,v)\geq \alpha(a_j-a'_j)+\beta(k-k').$$
\end{description}


  Similar to \cite{BH-2010} and \cite{LLW-2020}, for every $i\in \mathbf{M},$
   we introduce a mapping $\tilde{f}_{i}:\mathbb{[}0,T\mathbb{]}\times \mathbb{R}^n%
\times \mathbb{R}\times \mathbb{R}^{m-1}\times \mathbb{R}^{d}\times \mathbb{R}\times U\times
V\rightarrow \mathbb{R}$ associated with the function $f_i$ as follows:
\begin{equation*}
\tilde{f}_{i}(t,x,y,h,z,k,u,v):=  f_{i}(t,x,a,z,k,u,v),
\end{equation*}%
where  $y$ and  $h=\big(h(1),\ldots,h(m-1)\big)$ are related with  $a=(a_{1},\cdots ,a_{m})$ as
\begin{equation}\label{y-a-h}
\left\{
\begin{aligned}
&\!a_{j}  =  y+h(m-i+j),& j<i,    \\
&\!a_{j}  =  y, & j=i,    \\
&\!a_{j}  =   y+h(j-i),& j>i.
\end{aligned}%
\right.
\end{equation}%
Clearly, for all $(t,x)\in [ 0,T]\times \mathbb{R}^n,\ z\in \mathbb{R}^{d},\ k\in \mathbb{R},\ u\in U,\ v\in V,$
\begin{equation*}
\left\{
\begin{aligned}
&\!f_{1}(t,x,a,z,k,u,v)   =   \tilde{f}_{1}(t,x,a_{1},h^{1},z,k,u,v),    \\
&\!f_{2}(t,x,a,z,k,u,v)   =  \tilde{f}_{2}(t,x,a_{2},h^{2},z,k,u,v),    \\
&\!\ \ \ \ \ \ \ \ \ \ \ \ \ \ldots \ldots &  &  &  \\
&\!f_{m}(t,x,a,z,k,u,v)  = \tilde{f}_{m}(t,x,a_{m},h^{m},z,k,u,v),
\end{aligned}
\right.
\end{equation*}%
where $h^{i}=\big(h^{i}(1),\cdots ,h^{i}(m-1)\big)$ with
\begin{equation*}
h^{i}(j)=a_{(i+j)\mbox{mod}( m)}-a_{i}=\left\{
\begin{aligned}
&\!a_{i+j}-a_{i},\ \  \ \ \ \ 1\leq j\leq m-i, &  &  &  \\
&\! a_{i+j-m}-a_{i},\ \ \ m-i+1\leq j\leq m-1.&  &  &
\end{aligned}
\right.
\end{equation*}%
 %
It is easy to check that, for each $i\in\mathbf{M}$,
the function    $\tilde{f}_i$  has   the following properties  similar to
   $(\mathbf{H2})$ and $(\mathbf{H3})$:

\begin{description}
 \item$(\mathbf{H2})'$  For every $(x,y,h,z,k)\in \mathbb{R}^n\times \mathbb{R}\times \mathbb{R}^{m-1}\times \mathbb{R}^d\times \mathbb{R}$, $\tilde{f}_i(\cdot, x,y,h,z,k,\cdot,\cdot)$ is continuous in $(t,u,v)$;

 $\ \ \tilde{f}_i$ is  Lipschitz in  $(x,y,h,z,k)$, uniformly with respect to $(t,u,v)\in [0,T]\times U\times V$.

\item$(\mathbf{H3})'$ For  $j\in\mathbf{L}$,  $(t,x,y,z,u,v)\in [0,T]\times  \mathbb{R}^n\times \mathbb{R}\times \mathbb{R}^{d}\times U\times V$,
 $h,$ $h'\in \mathbb{R}^{m-1}$ with $h_p=h'_p$, $\forall\ p\neq j$, $  h_j\geq h'_j$, and $k$, $k'\in\mathbb{R}$ with $k\geq k'$,
    $$\tilde f_i(t,x,y,h,z,k, u,v)-\tilde f_i(t,x,y,h',z,k',u,v)\geq {\alpha(h_j-h'_j)+\beta(k-k')}.$$
    \end{description}


For any given admissible controls $(u,v)\in \mathcal{U}\times
\mathcal{V}$, $ i\in \mathbf{M}$ and the initial data $(\tau,\eta )\in \lbrack
0,T]\times L^{2}(\Omega ,\mathcal{F}_{\tau},P;\mathbb{R}^n)$, we consider the following backward system with two Poisson random measures
\begin{equation}
\left\{
\begin{aligned}
&\!dY_{s}^{\tau,\eta ,i;u,v}   =    -\tilde{f}_{N_{s}^{\tau,i}}\Big(s,X_{s}^{\tau,\eta
,i;u,v},Y_{s}^{\tau,\eta ,i;u,v},H_{s}^{\tau,\eta ,i;u,v},Z_{s}^{\tau,\eta
,i;u,v}, \\
& \! \qquad\qquad\qquad\qquad\qquad\int_EK_{s}^{\tau,\eta ,i;u,v}(e)\rho(X_{s}^{\tau,\eta ,i;u,v},e)\nu(de),u_{s},v_{s}\Big)ds +\lambda \sum\limits_{l=1}^{m-1}H_{s}^{\tau,\eta
,i;u,v}(l)ds\\
& \!\qquad\qquad\qquad+Z_{s}^{\tau,\eta
,i;u,v}dB_{s}+\sum\limits_{l=1}^{m-1}H_{s}^{\tau,\eta ,i;u,v}(l)d\tilde{N}%
_{s}(l)+\int_E  K_{s}^{\tau,\eta ,i;u,v}(e)\tilde{\mu}(ds,de),    \\
&\! Y_{T}^{\tau,\eta ,i;u,v}   =  \Phi_{N_{T}^{\tau,i}}(X_{T}^{\tau,\eta ,i;u,v}),
\end{aligned}%
\right.  \label{equ3.1}
\end{equation}
where $X^{\tau,\eta ,i;u,v}$ is the solution of equation \eqref{fequ3.1}. Obviously, due to Theorem \ref{l1}, \eqref{equ3.1} admits a unique solution
 $(Y^{\tau,\eta ,i;u,v},H^{\tau,\eta ,i;u,v},$ $Z^{\tau,\eta,i;u,v},K^{\tau,\eta ,i;u,v})\in \mathcal{S}^{2}[\tau,T]$.
Similar to  Proposition A.1    in \cite{BHL},   there exists some constant $C>0$ such that, for all
$\tau\in[0,T]$, $i\in \mathbf{M}$, $(u,v)\in \mathcal {U}\times \mathcal {V}$, $\eta,\eta'\in L^2(\Omega, \mathcal {F}_\tau,P;\mathbb{R}^n)$, P-a.s.,
 \begin{equation}\label{equ3.01}\begin{aligned}
 {\rm (i)} &\  \mathbb{E}\Big[\sup\limits_{s\in[\tau,T]}|Y_s^{\tau,\eta,i;u,v}|^2+\int_\tau^T\Big(\lambda\sum\limits_{l=1}^{m-1}|H_s^{\tau,\eta,i;u,v}(l)|^2 + |Z_s^{\tau,\eta,i;u,v}|^2  +\int_E|K_s^{\tau,\eta,i;u,v}(e)|^2\nu(de)\Big)ds\mid \mathcal {F}_\tau\Big]\\
 &\leq C(1+|\eta|^2),\\
 {\rm(ii)}&\ \mathbb{E}\Big[\sup\limits_{s\in[\tau,T]}|Y_s^{\tau,\eta,i;u,v}-Y_s^{\tau,\eta',i;u,v}|^2+\int_\tau^T\lambda\sum\limits_{l=1}^{m-1}|H_s^{\tau,\eta,i;u,v}(l)-H_s^{\tau,\eta',i;u,v}(l)|^2ds\\
 & \ \ \  +\int_\tau^T|Z_s^{\tau,\eta,i;u,v}-Z_s^{\tau,\eta',i;u,v}|^2 ds +\int_\tau^T\int_E|K_s^{\tau,\eta,i;u,v}(e)-K_s^{\tau,\eta',i;u,v}(e)|^2\nu(de)ds\mid  \mathcal {F}_\tau\Big]\\
&\ \leq C|\eta-\eta'|^2. \end{aligned}\end{equation}

For the given  control processes $(u,v)\in\mathcal {U}\times\mathcal {V}$  and the initial data $(t,x,i)\in [0,T]\times \mathbb{R}^n\times\mathbf{M}$,
the  gain  functional $J(t,x,i;u,v)$ is defined  as follows
\begin{equation}
J(t,x,i;u,v):=   Y_{s}^{t,x,i;u,v}|_{s=t}, \label{equ3.05}
\end{equation}%
where $Y^{t,x,i;u,v}$ is the first component of the solution of the BSDE with jumps \eqref{equ3.1} with initial data $(\tau,\eta)=(t,x)$.
%
The classical result in stochastic  control theory (referring to \cite{BL, P-1997}) implies that, for any $i\in \mathbf{M}$, $(t,\eta)\in [0,T]\times
L^{2}(\Omega ,\mathcal{F}_{t},P;\mathbb{R}^n),$ we also have
\begin{equation}
J(t,\eta ,i;u,v)=J(t,x,i;u,v)|_{x=\eta }=Y_{t}^{t,\eta ,i;u,v},\ \ %
\mbox{P-a.s.}  \label{equ3.07}
\end{equation}%

To formulate our game problem consistent with  the Elliott-Kalton ``strategy against control" setting in \cite{EK}, we introduce the following  definitions of   admissible controls and  nonanticipative strategies on \textcolor{red}{a} stochastic time interval $[[\tau_1,\tau_2]]$\footnote{$[[\tau_1,\tau_2]]:=   \big\{(t,\omega)\in
[0,T]\times\Omega,\tau_1(\omega)\leq t\leq \tau_2(\omega)\big\}.$}, where
$\tau_1$ and $\tau_2$ are two stopping times with $t\leq\tau_1< \tau_2\leq T$, $P$-a.s.

\begin{definition}\label{Def-1}\sl
\  Let $u^0\in {U}$. A process $u=\big\{u_r(\omega), (r,\omega)\in [[\tau_1,\tau_2]]\big\}$
is an admissible control for Player  I on $[[\tau_1,\tau_2]]$, if $u\textbf{1}_{[[\tau_1,\tau_2]]}+u^0\textbf{1}_{[[0,T]]\backslash[[\tau_1,\tau_2]]}$ is $(\mathcal {F}_r)$-progressively measurable and with values in ${U}$.
The set of all admissible controls for Player  I on $%
[[\tau_1,\tau_2]]$ is denoted by $\mathcal{U}_{\tau_1,\tau_2}$. We identify two processes $u$ and $\bar{u}$
in $\mathcal{U}_{\tau_1,\tau_2}$ and write $u\equiv \bar{u}$ on $%
[[\tau_1,\tau_2]]$, if $ P\{u=\bar{u}\ a.e.\ in \
[[\tau_1,\tau_2]]\}=1.$

Similarly, we can define the admissible control process $v=\big\{v_r(\omega), (r,\omega)\in [[\tau_1,\tau_2]]\big\}$
 for Player II on $[[\tau_1,\tau_2]]$. The set of all admissible controls for Player  II on $%
[[\tau_1,\tau_2]]$ is denoted by $\mathcal{V}_{\tau_1,\tau_2}$.
 $v=\bar{v}$ on $%
[[\tau_1,\tau_2]]$ in $\mathcal{V}_{\tau_1,\tau_2}$ can also be interpreted similarly.
\end{definition}

\begin{definition}\label{Def-2}\sl
\ A nonanticipative strategy for Player  I on $[[\tau_1,\tau_2]]$ is a mapping $%
\alpha:\mathcal{V}_{\tau_1,\tau_2}\rightarrow \mathcal{U}_{\tau_1,\tau_2}$
such that, for any $\mathbb{F}$-stopping time $S:\Omega \rightarrow[{{%
[\tau_1,\tau_2]}}]$ and any $v_1,v_2\in \mathcal{V}_{\tau_1,\tau_2},$ with $%
v_1\equiv v_2$ on $[[\tau_1,S]]$, it holds that $\alpha(v_1)\equiv
\alpha(v_2)$ on $[[\tau_1,S]]$. Nonanticipative strategies $\beta: \mathcal{U}_{\tau_1,\tau_2}\rightarrow\mathcal{V}_{\tau_1,\tau_2}$ for Player  II on $%
[[\tau_1,\tau_2]]$,  are
defined similarly. The set of all nonanticipative strategies $\alpha:
\mathcal{V}_{\tau_1,\tau_2}\rightarrow \mathcal{U}_{\tau_1,\tau_2}$ for
Player  I on $[[\tau_1,\tau_2]]$ is denoted by $\mathcal{A}_{\tau_1,\tau_2}$, while the
 set of all nonanticipative strategies $\beta:  \mathcal{U}%
_{\tau_1,\tau_2}\rightarrow \mathcal{V}_{\tau_1,\tau_2}$ for Player  II on $%
[[\tau_1,\tau_2]]$ is denoted by $\mathcal{B}_{\tau_1,\tau_2}.$
\end{definition}

\medskip

The Elliott-Kalton ``strategy against control" setting  means that, Player II (resp. Player I) will choose an
admissible strategy  $\beta$ (resp. $\alpha$) responding to the admissible control $u $ (resp. $v$) chosen by
Player I (resp. Player II)  to minimize (resp. maximize) his/her loss (resp. gain) $J(t,x,i;u,v)$.
%
According to this mechanism,    the associated lower and upper value functions are defined respectively as follows: for any $(t,x,i)\in[0,T]\times\mathbb{R}^n\times\mathbf{M}$,
\begin{equation}\label{equ3.08}
\begin{aligned}
 &(\text{lower})\ W_{i}(t,x)= \mathop{\rm essinf}\limits_{\beta \in
\mathcal{B}_{t,T}}\mathop{\rm esssup}\limits _{u\in \mathcal{U}%
_{t,T}}J\big(t,x,i;u,\beta (u)\big),\\
 &(\text{upper})\ U_{i}(t,x)= \mathop{\rm esssup}\limits _{\alpha \in
\mathcal{A}_{t,T}}\mathop{\rm essinf}\limits_{v\in \mathcal{V}%
_{t,T}}J\big(t,x,i;\alpha (v),v\big).
\end{aligned}
\end{equation}
%
%
%
For fixed $(t,x,i)\in [0,T]\times \mathbb{R}^n\times\mathbf{M}$, since $W_{i}(t,x)$ is the essential infimum and the essential supremum over a family of $\mathcal {F}_t$-measurable criterion
functional $J\big(t,x,i;u,\beta(u)\big)$, it  is  originally an $\mathcal{F}_{t}$-measurable
random variable. However, it turns out  to be deterministic as shown in the following Proposition \ref{Pro---3.1}. From now on, we focus on the study of the lower value functions $W_{i}(t,x)$, $i\in\mathbf{M}$, the upper value functions $U_{i}(t,x)$, $i\in \mathbf{M}$, can be investigated in a same manner.

Using the same arguments as Proposition 3.1 in \cite{BHL}, we   get the following result.
\begin{proposition}\label{Pro---3.1}\sl
For any $(t,x,i)\in \lbrack 0,T]\times \mathbb{R}^n\times
\mathbf{M}$,
the lower value function $W_{i}(t,x)$ is   deterministic in the sense that $%
\mathbb{E}[W_{i}(t,x)]=W_{i}(t,x),$  $P$-a.s.
\end{proposition}

Combining this with \eqref{equ3.01} and \eqref{equ3.05},  we also have the following properties for the lower value functions.

\begin{lemma}
\label{l2.1}\sl  There exists a constant $C>0$ such that, for all $i\in \mathbf{M},$ $x$, $x_1,$ $x_2 \in \mathbb{R}^n,$ $ t\in [ 0,T],$
\begin{equation}  \label{3.7}
\begin{aligned}
  |W_{i}(t,x_1)-W_{i}(t,x_2)|\leq C|x_1-x_2|,\
   |W_{i}(t,x)|\leq C(1+|x|).
\end{aligned}
\end{equation}
\end{lemma}




\subsection{Dynamic Programming Principle}
In order to  establish the relationship between the lower  value function $\mathbf{W}(t,x)=(W_{1}(t,x),\cdots,$ $W_{m}(t,x))$, $(t,x)\in [0,T]\times \mathbb{R}^n$ and %
the system   \eqref{1} of coupled HJBI equations with integral-partial operators, we need the related dynamic programming principle.
For this, we firstly adapt the  stochastic backward semigroup introduced by Peng
\cite{P-1997}    to our framework.
\begin{definition}\label{Def-bs}\sl
 Given $(t,x,i)\in [0,T]\times \mathbb{R
}^n\times \mathbf{M}$, $(u,v)\in \mathcal{U}\times \mathcal{V}$, for any stopping
times $\varsigma_1$ and $\varsigma_2 $ such that $t\leq \varsigma_1 \leq \varsigma_2 \leq T$, $P$-a.s., and  $\xi\in L^{2}(\Omega ,\mathcal{F}_{\varsigma_2},P;\mathbb{R}^n)$,  the stochastic backward semigroup
$G_{\varsigma_1, \varsigma_2}^{t,x,i;u,v}[\cdot]$ is defined as
\begin{equation*}
G_{\varsigma_1,\varsigma_2}^{t,x,i;u,v}[ \xi ]:=   %
\widetilde{Y}_{\varsigma_1}^{t,x,i;u,v},
\end{equation*}%
where $(\widetilde{Y}_{s}^{t,x,i;u,v},\widetilde{H}%
_{s}^{t,x,i;u,v},\widetilde{Z}_{s}^{t,x,i;u,v},\widetilde{K}_{s}^{t,x,i;u,v})_{s\in \lbrack t,\varsigma_2]}$ is the
solution of the following BSDE with the random time horizon $\varsigma_2$:
{\small \begin{equation}\label{equ_SBG}
\left\{
\begin{aligned}
&\!d\widetilde{Y}_{s}^{t,x,i;u,v}  =   -\tilde{f}_{N_{s}^{t,i}}\Big(s, {X}%
_{s}^{t,x,i;u,v},\widetilde{Y}_{s}^{t,x,i;u,v},\widetilde{H}_{s}^{t,x,i;u,v},\widetilde{Z}%
_{s}^{t,x,i;u,v},\int_E\widetilde K_{s}^{t,x,i;u,v}(e)\rho(X_{s}^{t,x,i;u,v},e)\nu(de),u_{s},v_{s}\Big)ds  \\
 &\!\hskip1.7cm +\lambda \sum\limits_{l=1}^{m-1}\widetilde{H}_{s}^{t,x,i;u,v}(l)ds+\widetilde{Z}%
_{s}^{t,x,i;u,v}dB_{s}+\sum\limits_{l=1}^{m-1}\widetilde{H}_{s}^{t,x,i;u,v}(l)d%
\tilde{N}_{s}(l)+\int_E\widetilde K_{s}^{t,x,i;u,v}(e)\tilde{\mu}(ds,de),  \\
&\!\widetilde{Y}_{\varsigma_2 }^{t,x,i;u,v} =   \xi,\ \ s\in[t,\varsigma_2],
\end{aligned}
\right.
\end{equation}%
}
where $X^{t,x,i;u,v}$ is the solution of equation \eqref{fequ3.1} with $(\tau,\eta)=(t,x)\in [0,T]\times\mathbb{R}^n$.
\end{definition}

By using the backward semigroup $G_{\varsigma_1,\varsigma_2 }^{t,x,i;u,v}[\cdot]$,  we get the following relation for the solution component  $Y^{t,x,i;u,v}$
of BSDE with jumps \eqref{equ3.1} with $(\tau,\eta) =(t,x)$:
\begin{equation}
G_{t,\varsigma_2 }^{t,x,i;u,v}\big[Y_{\varsigma_2
}^{t,x,i;u,v}\big]=G_{t,T}^{t,x,i;u,v}\big[\Phi_{N_{T}^{t,i}}(X_{T}^{t,x,i;u,v})\big]=Y_{t}^{t,x,i;u,v}=J(t,x,i;u,v).
\end{equation}

Now we present the following   two dynamic programming principles in this paper. The first one (Theorem \ref{WDPP}) is classical and known as a weak DPP, while the second one (Theorem \ref{SDPP}) is called a strong DPP in the sense that it generalizes the weaker version to  stopping times.

\begin{theorem}
\label{WDPP}\sl Suppose that   $(\mathbf{H1})$-$(\mathbf{H3})$ hold. For any $t\in[0,T]$, $\delta\in[0,T-t]$,
$x\in \mathbb{R}^n,\ i\in \mathbf{M},$ we have
\begin{equation}\label{equ_DPP}
W_{i}(t,x)=\mathop{\rm essinf}\limits_{\beta \in \mathcal{B}_{t,t+\delta }}%
\mathop{\rm esssup}_{u\in \mathcal{U}_{t,t+\delta }}G_{t,t+\delta
}^{t,x,i;u,\beta (u)}\Big[W_{N_{t+\delta }^{t,i}}\Big(t+\delta , {X}_{t+\delta
}^{t,x,i;u,\beta (u)}\Big)\Big],\ P\text{-a.s.}
\end{equation}%

\end{theorem}
\begin{proposition}
\label{pro3.1} \sl Under the assumptions $(\mathbf{H1})$-$(\mathbf{H3})$, for each  $i\in \mathbf{M}$, $W_{i}(t,x)$ is $\frac{1}{2}$-H\"{o}lder continuous in $%
t $.  That is, there exists some
constant $C$ such that, for each $i\in \mathbf{M},$ for every $x\in \mathbb{R}^n,$
$t,\ t^{\prime }\in \lbrack 0,T],$
\begin{equation*}
|W_{i}(t,x)-W_{i}(t^{\prime },x)|\leq C(1+|x|)|t-t^{\prime }|^{\frac{1}{2}}.
\end{equation*}%

\end{proposition}
The proofs of Theorem \ref{WDPP} and Proposition \ref{pro3.1} are similar to Theorem 3.1 and Theorem 3.2   in  \cite{BHL}, respectively, and thus are omitted. We now turn our attention to the following strong DPP.

\begin{theorem}
\label{SDPP}\sl   Suppose that  $(\mathbf{H1})$-$(\mathbf{H3})$ hold. For any $t\in [0,T]$, any stopping time $\tau$ such that $ t\leq \tau\leq T$, $P$-a.s., $x\in \mathbb{R}^n,\ i\in \mathbf{M},$ we have
\begin{equation}\label{equ-SDPP}
W_i(t,x)=\mathop{\rm essinf}\limits_{\beta\in \mathcal{B}_{t,\tau}}%
\mathop{\rm esssup}\limits_{u\in\mathcal{U}_{t,\tau}}G_{t,\tau}^{t,x,i;u,%
\beta(u)}\Big[W_{N_{\tau}^{t,i}}\Big(\tau, {X}_{\tau}^{t,x,i;u,\beta(u)}\Big)\Big],\ P\text{-a.s.}
\end{equation}

\end{theorem}

In order to show Theorem \ref{SDPP}, we first generalize the definition of lower value function $W_i(\cdot,\cdot),$  $i\in \mathbf{M}$ to the case
involving  stopping times $\tau$, random variables $\eta\in  L^2(\Omega,\mathcal {F}_\tau,P; \mathbb{R}^n)$ and the Markov process $N_\tau^{t,i}$, which plays an important role when proving strong DPP.

\begin{proposition}\label{prop3.3}\sl
For any $t\in[0,T]$, $i\in \mathbf{M}$,  stopping times $\tau$ with values in $[t,T]$ and $\eta\in
L^2(\Omega,\mathcal{F}_\tau,P;\mathbb{R}^n)$, it holds
\begin{equation*}
W_{N_\tau^{t,i}}(\tau,\eta)=\mathop{\rm essinf}\limits_{\beta\in\mathcal{B}%
_{\tau,T}} \mathop{\rm esssup}\limits_{u\in\mathcal{U}_{\tau,T}}Y_\tau^{%
\tau,\eta,N_\tau^{t,i};u,\beta(u)},\  P\text{-a.s.}
\end{equation*}
\end{proposition}
We introduce the following three auxiliary lemmas in order to prove Proposition \ref{prop3.3}.
\begin{lemma}\label{ll3.4}\sl For each $i\in  \mathbf{M}$, $x\in\mathbb{R}^n$, and  any stopping time $\tau$ such that $t\leq \tau\leq T$, $P$-a.s.,  we have
\begin{equation}\label{1--1}
\essinf\limits_{\beta\in \mathcal {B}_{t,T}}\esssup\limits_{u\in \mathcal {U}_{t,T}}Y_{\tau}^{\tau,x,i;u,\beta(u)}=\essinf\limits_{\beta\in \mathcal {B}_{\tau,T}}\esssup\limits_{u\in \mathcal {U}_{\tau,T}}Y_{\tau}^{\tau,x,i;u,\beta(u)},\ \mbox{P-a.s.}
\end{equation}
\end{lemma}

 For any $t\in[0,T]$, let $\tau $ be a stopping time  with values in $[t, T]$. For any positive integer $N$,  we introduce a partition $\{t=t_0^N<t_1^N\cdots<t_N^N=T\}$ of time interval $[t,T]$ with the  points $t_j^N:=  t+\frac{(T-t)j}{N}$, $0\leq j\leq N$, and then define
 \begin{equation}\label{stopping}
 \tau_{N}:=t \textbf{1}_{\{\tau=t\}}+ \sum\limits_{j=1}^N t_j^N \textbf{1}_{\{t_{j-1}^N<\tau\leq t_j^N\}}.\end{equation}
It is easy to check that the stopping time   $\tau_N$ is  $\sigma\{\tau\}$-measurable, and $\tau_N\downarrow\tau$, as $N\rightarrow\infty$.

\begin{lemma} \label{le 031301}\sl
For the stopping time $\tau_N$ given in \eqref{stopping} and $\eta\in L^2(\Omega,\mathcal{F}_\tau,P;\mathbb{R}^n)$, $i\in \mathbf{M}$, we have
 $$W_i(\tau_N,\eta)=\mathop{\rm essinf}\limits_{\beta \in \mathcal{B}_{t,T}}%
\mathop{\rm esssup}_{u\in \mathcal{U}_{t,T}}\mathbb{E}\big[Y_{\tau_N}^{\tau_N,\eta,i;u,\beta(u)}\mid\mathcal {F}_\tau\big],\ \mbox{P-a.s.}$$
\end{lemma}
\begin{lemma} \label{ll3.3}\sl For all stopping times $\tau$ $(t\leq\tau\leq T)$, $\eta  \in L^{2}(\Omega ,%
\mathcal{F}_{\tau},P;\mathbb{R}^n),\ i\in \mathbf{M},$ $(u,v)\in \mathcal{U}_{\tau,T}\times \mathcal{V}_{\tau,T}$, we have the following estimate

$$\big|\mathbb{E}\big[Y_{\tau_N}^{\tau_N,\eta ,i;u,v}\mid\mathcal {F}_\tau\big]-Y_{\tau }^{\tau,\eta
 ,i;u,v}\big|^2\leq  C(\tau_N-\tau )(1+|\eta|^2),\ \mbox{P-a.s.},\ N\geq 1.
 $$
\end{lemma}
\smallskip
We only give the proof of  Lemma \ref{ll3.3} since Lemmas \ref{ll3.4}  and  \ref{le 031301}  can be obtained in a same manner to the ones in our recent study \cite{LLW-2020}.

\begin{proof}
%
According to Lemma \ref{l2},  by setting
{\small  $$\begin{aligned}
&
\xi_1: =\Phi _{N_T^{\tau,i}}\big(X_T^{{\tau} ,\eta ,i ;u,v }\big),\ \  \xi_2 := \Phi_{N_T^{\tau_N ,i}}\big(X_T^{\tau_N,\eta ,i;u,v}\big),\ \
 \varphi_1(s): =0,\\
  &\varphi_2(s): =\tilde f_{N_s^{\tau_N,i}}\Big(s,X_s^{\tau_N,\eta ,i;u,v},Y_s^{\tau_N,\eta ,i;u,v},H_s^{\tau_N,\eta ,i;u,v},Z_s^{\tau_N,\eta ,i;u,v},\int_EK_s^{\tau_N,\eta ,i;u,v}(e)\rho(X_s^{\tau_N,\eta ,i;u,v},e)\nu(de),u_s,v_s\Big)\\
 &\hskip1.2cm-\tilde f_{N_s^{\tau,i}}\Big(s,X_s^{\tau_N,\eta ,i;u,v},Y_s^{\tau_N,\eta ,i;u,v},H_s^{\tau_N,\eta ,i;u,v},Z_s^{\tau_N,\eta ,i;u,v},\int_EK_s^{\tau_N,\eta ,i;u,v}(e)\rho(X_s^{\tau_N,\eta ,i;u,v},e)\nu(de),u_s,v_s\Big),
 \end{aligned}$$ }
we get from the fact $N_s^{\tau,i}=N_s^{\tau_N,N_{\tau_N}^{\tau,i}},\ s\geq \tau_N$,  that
\begin{equation}\label{ee33}
 \begin{aligned}
&  \mathbb{E}\Big[\big|Y_{{\tau_N}}^{{\tau_N},\eta ,i ;u,v}-Y_{{\tau_N}}^{\tau,\eta
 ,i  ;u,v }\big|^{2}\mid\mathcal{F}_\tau\Big]\\
 \leq&    C\mathbb{E}\Big[\big|\Phi_{N_T^{\tau,i}}(X_T^{{\tau} ,\eta ,i  ;u,v })-\Phi_{N_T^{{\tau_N},i}}(X_T^{{\tau_N},\eta ,i  ;u,v })\big|^2\mid\mathcal {F}_{\tau}\Big] +C  \mathbb{E}\Big[\int_{\tau_N}^{T}\big|\varphi_2(s)\big|^2ds\mid\mathcal {F}_{\tau}\Big]\\
   \leq  &  C\mathbb{E}\Big[\big|X_T^{\tau,\eta ,i;u,v}-X_T^{\tau_N,\eta ,i;u,v}\big|^2\mid \mathcal {F}_{\tau}\Big]+C \mathbb{E}\Big[\textbf{1}_{\{N_{\tau_N}^{\tau,i}\neq i\}}\mathbb{E}\big[ 1+|X_T^{{\tau_N} ,\eta ,i ;u,v }|^2 \mid\mathcal {F}_{\tau_N}\big]\mid\mathcal {F}_{\tau}\Big]\\
   & +C\mathbb{E}\Big[\textbf{1}_{\{N_{\tau_N}^{\tau,i}\neq i\}} \mathbb{E}\Big[\int_{\tau_N}^{T}\Big(1+|X_s^{\tau_N,\eta ,i;u,v}|^2+|Y_s^{\tau_N,\eta ,i;u,v}|^2+\sum_{l=1}^{m-1}|H_s^{\tau_N,\eta ,i;u,v}(l)|^2\\
  &\hskip4cm+|Z_s^{\tau_N,\eta ,i;u,v}|^2+\int_E|K_s^{\tau_N,\eta ,i;u,v}(e)|^2\nu (de)\Big)ds|\mathcal {F}_{\tau_N}\Big]\mid\mathcal {F}_{\tau}\Big] \\
 \leq&    C\mathbb{E}\Big[\big|X_T^{\tau,\eta ,i;u,v}-X_T^{\tau_N,\eta ,i;u,v}\big|^2\mid \mathcal {F}_{\tau}\Big]+C(1+|\eta|^2)\mathbb{E}\Big[\textbf{1}_{\{N_{\tau_N}^{\tau,i}\neq i\}}\mid\mathcal {F}_{\tau}\Big],
 \end{aligned}
 \end{equation}
where the last inequality is obtained from \eqref{ee222}-(i) and \eqref{equ3.01}-(i).
Notice that, as
\begin{equation}\label{2020122101}
 \begin{aligned}
& P(N_{\tau_N}^{\tau,i}\neq i\mid  {\mathcal{F}_\tau})\leq 1-P\big\{N\big((\tau,\tau_N]\times\{l\}\big)=0,\ 1\leq l\leq m-1\mid {\mathcal{F}_\tau}\big\}\\
& =1-e^{-\lambda(m-1)(\tau_N-\tau)}\leq C|\tau_N-\tau|,
\end{aligned}\end{equation}
  we have
\begin{equation}\label{2020122102}
  \mathbb{E}\Big[\big|Y_{{\tau_N}}^{{\tau_N},\eta ,i ;u,v}-Y_{{\tau_N}}^{\tau,\eta
 ,i  ;u,v }\big|^{2}\mid\mathcal{F}_\tau\Big]
 \leq    C\mathbb{E}\Big[\big|X_T^{\tau,\eta ,i;u,v}-X_T^{\tau_N,\eta ,i;u,v}\big|^2\mid \mathcal {F}_{\tau}\Big]+C(1+|\eta|^2)|\tau_N-\tau|.
 \end{equation}

Now we focus on $\mathbb{E}\Big[\big|X_T^{\tau,\eta ,i;u,v}-X_T^{\tau_N,\eta ,i;u,v}\big|^2\mid \mathcal {F}_{\tau}\Big]$.
 %
 For $s\in [\tau_N,T]$, applying It\^o's formula to $|X_s^{\tau,\eta ,i;u,v}-X_s^{\tau_N,\eta ,i;u,v}|^2 $, we get from \eqref{ee222} and \eqref{2020122101} that
 $$
 \begin{aligned}
&\mathbb{E} \Big[\big|X_s^{\tau,\eta ,i;u,v}-X_s^{\tau_N,\eta ,i;u,v}\big|^2\mid\mathcal{F}_\tau\Big]\\
 =& \mathbb{E} \Big[ \big|X_{\tau_N}^{\tau,\eta ,i;u,v}-\eta\big|^2\mid\mathcal{F}_\tau\Big]\\
&
 +\mathbb{E} \Big[\int_{\tau_N}^s\Big(2(X_r^{\tau,\eta ,i;u,v}-X_r^{\tau_N,\eta ,i;u,v})\big(b_{N_r^{\tau,i}}(r,X_r^{\tau,\eta ,i;u,v},u_r,v_r )-b_{N_r^{\tau_N,i}}(r,X_r^{\tau_N,\eta ,i;u,v},u_r,v_r )\big) \\
 & \hskip1.8cm+\big|\sigma_{N_r^{\tau,i}}(r,X_r^{\tau,\eta ,i;u,v},u_r,v_r )-\sigma_{N_r^{\tau_N,i}}(r,X_r^{\tau_N,\eta ,i;u,v},u_r,v_r )\big|^2\\
 & \hskip1.8cm+\int_E\big|\gamma_{N_r^{\tau,i}}(r,X_{r-}^{\tau,\eta ,i;u,v},u_r,v_r,e )-\gamma_{N_r^{\tau_N,i}}(r,X_{r-}^{\tau_N,\eta ,i;u,v},u_r,v_r,e )\big|^2\nu(de)
 \Big) dr\mid\mathcal{F}_\tau\Big]\\
 \leq& C(1+|\eta|^2)|{\tau_N}-\tau|+C\mathbb{E} \Big[\int_{\tau_N}^s\big|X_r^{\tau,\eta ,i;u,v}-X_r^{\tau_N,\eta ,i;u,v}\big|^2dr\mid\mathcal{F}_\tau\Big]\\
& +C\mathbb{E} \Big[\int_{\tau_N}^s\Big( \big|b_{N_r^{\tau,i}}(r,X_r^{\tau_N,\eta ,i;u,v} ,u_r,v_r )-b_{N_r^{\tau_N,i}}(r,X_r^{\tau_N,\eta ,i;u,v},u_r,v_r)\big|^2 \\
 & \hskip2 cm+\big|\sigma_{N_r^{\tau,i}}(r,X_r^{\tau_N,\eta ,i;u,v} ,u_r,v_r )-\sigma_{N_r^{\tau_N,i}}(r,X_r^{\tau_N,\eta ,i;u,v} ,u_r,v_r )\big|^2\\
 & \hskip2 cm+\int_E\big|\gamma_{N_r^{\tau,i}}(r,X_{r-}^{\tau_N,\eta ,i;u,v},u_r,v_r,e )-\gamma_{N_r^{\tau_N,i}}(r,X_{r-}^{\tau_N,\eta ,i;u,v},u_r,v_r,e )\big|^2\nu(de)
 \Big) dr\mid\mathcal{F}_\tau\Big]\\
 \leq& C(1+|\eta|^2)|{\tau_N}-\tau|
 +C\mathbb{E} \Big[\int_{\tau_N}^s\big|X_r^{\tau,\eta ,i;u,v}-X_r^{\tau_N,\eta ,i;u,v}\big|^2dr\mid\mathcal{F}_\tau\Big] \\
  & +C\mathbb{E} \Big[\textbf{1}_{\{N_{\tau_N}^{{\tau},i}\neq  i\}}\mathbb{E}\big[\int_{\tau_N}^{T}\big(1+|X_r^{\tau_N,\eta ,i;u,v}|^2\big)dr|\mathcal {F}_{\tau_N}\big]\mid\mathcal{F}_\tau\Big]\\
  %
   %
 \leq& C(1+|\eta|^2)|{\tau_N}-\tau|
 +C\mathbb{E} \Big[\int_{\tau_N}^s\big|X_r^{\tau,\eta ,i;u,v}-X_r^{\tau_N,\eta ,i;u,v}\big|^2dr\mid\mathcal{F}_\tau\Big].
 \end{aligned}
 $$
 Then   Gronwall's inequality implies  that
\begin{equation} \label{ee1-1}
\mathbb{E} \Big[\big|X_T^{\tau,\eta ,i;u,v}-X_T^{\tau_N,\eta ,i;u,v}\big|^2\mid\mathcal{F}_\tau\Big]\leq C(1+|\eta|^2)|{\tau_N}-\tau|.
\end{equation}
Substituting \eqref{ee1-1} into \eqref{2020122102}, we get
\begin{equation}\label{ee001}
  \mathbb{E}\Big[\big|Y_{{\tau_N}}^{{\tau_N},\eta ,i ;u,v}-Y_{{\tau_N}}^{\tau,\eta
 ,i  ;u,v }\big|^{2}\mid\mathcal{F}_\tau\Big]
 \leq   C(1+|\eta|^2)|{\tau_N}-\tau|.
 \end{equation}
%
%

On the other hand, since  $\tau_N(\geq \tau)$ is $\sigma(\tau)$-measurable, from     \eqref{ee222}-(i) and \eqref{equ3.01}-(i), we get
 \begin{equation}\label{eeee33.1}
\begin{aligned}
& \Big|\mathbb{E}\big[Y_{\tau_N}^{\tau,\eta,i;u,v}\mid\mathcal {F}_\tau\big]-Y_\tau^{\tau,\eta,i;u,v}\Big|\\
=&\Big|\mathbb{E}\Big[\int_\tau^{\tau_N}\tilde f_{N_s^{\tau,i}}\Big(s,X_s^{\tau,\eta,i;u,v},Y_s^{\tau,\eta,i;u,v},H_s^{\tau,\eta,i;u,v},Z_s^{\tau,\eta,i;u,v},\int_EK_s^{{\tau},\eta ,i;u,v  }(e)\rho(X_s^{\tau,\eta,i;u,v},e)\nu(de),u_s,v_s\Big)ds\\
& \hskip0.6cm
-\int_\tau^{\tau_N}\lambda\sum_{l=1}^{m-1}H_s^{\tau,\eta,i;u,v}(l)ds\mid \mathcal{F}_\tau\Big]\Big|\\
 &\leq C|\tau_N-\tau|^\frac12\Big(\mathbb{E}\Big[\int_\tau^{\tau_N}\Big(1+ |X_s^{\tau,\eta,i;u,v}|^2+|Y_s^{\tau,\eta,i;u,v}|^2+\sum_{l=1}^{m-1}|H_s^{\tau,\eta,i;u,v}(l)|^2+|Z_s^{\tau,\eta,i;u,v}|^2\\
&\hskip4cm+\int_E|K_s^{\tau,\eta,i;u,v}(e)|^2\nu(de)\Big)ds\mid\mathcal {F}_\tau\Big]\Big)^\frac12\\
&\leq C|\tau_N-\tau|^\frac12(1+|\eta|),\  \mbox{P-a.s.}
\end{aligned}\end{equation}
Therefore, by   \eqref{ee001} and \eqref{eeee33.1}, we get

$$
\big|\mathbb{E}\big[Y_{\tau_N}^{\tau_N,\eta ,i;u,v}\mid\mathcal {F}_\tau\big]-Y_\tau^{\tau,\eta,i;u,v}\big|^2
\leq C(\tau_N-\tau) (1+|\eta|^2),\  \mbox{P-a.s.}, \ N\geq 1.
$$
\end{proof}

Based on the above results,  we complete the proof of Proposition \ref{prop3.3} as follows.
 \medskip

\noindent \textbf{Proof of Proposition \ref{prop3.3}:}
\noindent Firstly, from Lemma \ref{ll3.3}, we have
$$\begin{array}{lll}
\displaystyle \Big|\essinf\limits_{\beta\in\mathcal{B}%
_{t,T}} \esssup\limits_{u\in\mathcal{U}_{t,T}}\mathbb{E}\big[Y_{\tau_N}^{\tau_N,\eta,i;u,\beta(u)}\mid\mathcal {F}_\tau\big]-\essinf\limits_{\beta\in\mathcal{B}%
_{t,T}} \esssup\limits_{u\in\mathcal{U}_{t,T}} Y_{\tau }^{\tau,\eta,i;u,\beta(u)}  \Big|^2\\
\displaystyle \leq C(\tau_N-\tau)(1+|\eta|^2)

\displaystyle \leq C\frac{T-t}{N}(1+|\eta|^2)\rightarrow0,\ \mbox{as } N\rightarrow \infty,\ \mbox{P-a.s.}
\end{array}  $$
Using Proposition \ref{pro3.1} and     Lemma \ref{le 031301}, we get
that
$$W_i(\tau,\eta)=\essinf\limits_{\beta\in\mathcal{B}%
_{t,T}} \esssup\limits_{u\in\mathcal{U}_{t,T}} Y_{\tau }^{\tau,\eta,i;u,\beta(u)},\ \mbox{P-a.s.}, \ i\in \mathbf{M}.$$
Finally, since $\mathbf{M}$ is finite, and $N_\tau^{t,i}$ is $\mathcal {F}_\tau$-measurable, it follows from
Lemma \ref{ll3.4} that
$$W_{N_\tau^{t,i}}(\tau,\eta)=\essinf\limits_{\beta\in\mathcal{B}%
_{t,T}} \esssup\limits_{u\in\mathcal{U}_{t,T}} Y_{\tau }^{\tau,\eta,N_\tau^{t,i};u,\beta(u)}=\essinf\limits_{\beta\in\mathcal{B}%
_{\tau,T}} \esssup\limits_{u\in\mathcal{U}_{\tau,T}} Y_{\tau }^{\tau,\eta,N_\tau^{t,i};u,\beta(u)},\ \mbox{P-a.s.}, \ i\in \mathbf{M}.$$
\endpf

Based on the above preparations, we  can complete the proof of   Theorem \ref{SDPP} (strong DPP), which is postponed in Appendix.

\section{Viscosity solutions of coupled Isaacs' type integral-partial differential equations}

\subsection{Existence Theorem }

For each $t\in [0,T)$, $i\in\mathbf{M},$ $\varphi\in C^2\big([0,T]\times\mathbb{R}^n;\mathbb{R}\big),$ by introducing
the following integral-differential operator
\begin{equation*}
\begin{aligned}
\label{equ4.1}
L_{u,v}^i\varphi(t,x)  :=   &\frac{1}{2}\tr\big(\sigma_i%
\sigma_i^*(t,x,u,v)D^2\varphi(t,x)\big)  +b_i(t,x,u,v)D\varphi(t,x)\\
& + \int_E\Big[\varphi\big(t,x+\gamma_i(t,x,u,v,e)\big)-\varphi(t,x)-D\varphi(t,x)\gamma_i(t,x,u,v,e)\Big]\nu(de),
\end{aligned}
\end{equation*}
we can  rewrite the system of coupled Isaacs' type  integral-partial differential equations \eqref{1} as follows
\begin{equation}  \label{equ4.2}
\left \{
\begin{aligned}
&\! \frac{\partial W_i}{\partial t}(t,x)+ \sup\limits_{u\in U}\mathop{\rm inf}%
\limits_{v\in V}\Big\{L_{u,v}^iW_i(t,x)
+f_i\big(t,x,\mathbf{W}(t,x), D
W_i(t,x)\sigma_i(t,x,u,v),C^i_{u,v}W_i(t,x),u,v\big)\Big\}=0,    \\
&\!W_i(T,x)=\Phi_i (x),\ \ \ (t,x,i)\in[0,T)\times \mathbb{R}^n\times \mathbf{M}.
\end{aligned}
\right.
\end{equation}

 We intend to prove
that the lower value function $\mathbf{W}(t,x)=\big(W_1(t,x),W_2(t,x),\cdots,W_m(t,x)\big),$ $(t,x)\in [0,T]\times
\mathbb{R}^n$, defined through  \eqref{equ3.08}, is a viscosity solution of the system \eqref{equ4.2}. Note that we only get the continuity  rather than the regularity  property of $W_i(\cdot,\cdot)$, $i\in \mathbf{M}$, in Lemma \ref{l2.1} and Proposition \ref{pro3.1},  and so the notion of viscosity solutions is considered here.
To this end, we firstly adapt the definition of viscosity solution introduced in \cite{CIL} to our system
\eqref{equ4.2}.

\begin{definition}\label{Def111}\sl
 The function $\mathbf{W}=(W_1,W_2,\cdots,W_m)\in C\big([0,T]\times\mathbb{R}^n;%
\mathbb{R}^m\big)$   is said to be\newline
{\rm(i)} a viscosity subsolution of  system \eqref{equ4.2}, if  for all $i\in%
\mathbf{M}$,  $W_i(T,x)\leq
\Phi_i (x),\  x\in\mathbb{R}^n, $ and if for all $i\in%
\mathbf{M}$,  for all functions $\phi\in C_{l,b}^3\big([0,T]\times
\mathbb{R}^n; \mathbb{R}\big)$ and  $(t,x)\in [0,T)\times \mathbb{R}^n$
such that $(t,x)$ is a local maximum point of $W_i-\phi$, we
have, for any $\delta>0$,
\begin{equation*}
\begin{aligned}
&\frac{\partial \phi}{\partial t}(t,x)+\sup_{u\in U}\mathop{\rm inf}_{v\in
V}\Big\{A_{u,v}^i\phi(t,x)+B^i_{\delta,u,v}(W_i,\phi)(t,x)\\
&\hskip3cm+f_i\Big(t,x,\mathbf{W}(t,x), D
\phi(t,x)\sigma_i(t,x,u,v),C^i_{\delta,u,v}(W_i,\phi)(t,x),u,v\Big)\Big\}\geq 0,
\end{aligned}\end{equation*}
where
$$
\begin{aligned}
&A_{u,v}^i\phi(t,x):= \frac{1}{2}\tr\big(\sigma_i\sigma_i^*(t,x,u,v)D^2\phi(t,x)\big)  +b_i(t,x,u,v)D\phi(t,x),\\
&B^i_{\delta,u,v}(W_i,\phi)(t,x):= \int_{E_\delta}\Big[\phi \big(t,x+\gamma_i(t,x,u,v,e)\big)-\phi (t,x)-D\phi (t,x)\gamma_i(t,x,u,v,e)\Big]\nu(de)\\
&\hskip3.3cm+\int_{E_\delta^c}\Big[W_i\big(t,x+\gamma_i(t,x,u,v,e)\big)-W_i(t,x)-D\phi (t,x)\gamma_i(t,x,u,v,e)\Big]\nu(de),\\
&C^i_{\delta,u,v}(W_i,\phi)(t,x):= \int_{E_\delta}\Big[\phi \big(t,x+\gamma_i(t,x,u,v,e)\big)-\phi (t,x) \Big] \rho(x,e) \nu(de)\\
&\hskip3.3cm+\int_{E_\delta^c}\Big[W_i\big(t,x+\gamma_i(t,x,u,v,e)\big)-W_i(t,x)\Big] \rho(x,e) \nu(de),
\end{aligned}
$$
with $E_\delta :=\big\{e\in E \mid |e|<\delta\big\}$;

\noindent {\rm(ii)} a viscosity  supersolution of  system \eqref{equ4.2}, if for
all $i\in\mathbf{M}$,  $%
W_i(T,x)\geq \Phi_i (x),\   x\in\mathbb{R}^n $, and if for all $i\in%
\mathbf{M}$, for all  functions $\phi\in
C_{l,b}^3\big([0,T]\times \mathbb{R}^n; \mathbb{R}\big)$ and $(t,x)\in[0,T)\times\mathbb{R}^n$ such that $(t,x)$ is a local minimum
point of $W_i-\phi$, we have,
for any $\delta>0$,
\begin{equation*}\begin{aligned}
&
\frac{\partial \phi}{\partial t}(t,x)+\sup_{u\in U}\mathop{\rm inf}_{v\in
V}\Big\{A_{u,v}^i\phi(t,x)+B^i_{\delta,u,v}(W_i,\phi)(t,x)\\
&
\hskip3cm +f_i\Big(t,x,\mathbf{W}(t,x), D
\phi(t,x)\sigma_i(t,x,u,v),C^i_{\delta,u,v}(W_i,\phi)(t,x),u,v\Big)\Big\}\leq 0;
\end{aligned}\end{equation*}

\noindent {\rm(iii)}  a viscosity solution of  system \eqref{equ4.2} if $\mathbf{W}$ is
both a viscosity subsolution and a viscosity supersolution of \eqref{equ4.2}.
\end{definition}
Here $C_{l,b}^3([0,T]\times \mathbb{R}^n; \mathbb{R})$ denotes the set of the real-valued functions that are continuously differentiable up to the third order and whose derivatives of order 1 to 3 are bounded.

\begin{remark}\sl
 {\rm(i)} From now on we impose the following additional condition on $\kappa(\cdot)$ in $(\mathbf{H1})$:
 $$\kappa(e)\leq C(1\wedge|e|), \mbox{ for all } e\in E.$$
 Notice that the lower value function $\mathbf{W}$ is   of linear growth (referring to Lemma \ref{l2.1}). Thus, this additional condition  ensures that the local maximum and minimum in Definition \ref{Def111} can be replaced by the global one. See Remark 4.3 in \cite{BHL} for more details.

 {\rm(ii)} For any $(t,x)\in [0,T]\times \mathbb{R}^n$, we can replace $B^i_{\delta,u,v}(W_i,\phi)(t,x)$ and $C^i_{\delta,u,v}(W_i,\phi)(t,x)$ by  $B^i_{u,v}\phi(t,x)$, $C^i_{u,v}\phi(t,x)$, respectively,   where $B^i_{u,v}\phi(t,x)$, $C^i_{u,v}\phi(t,x)$ are the ones in \eqref{ee01} with $\phi$ instead of  $W_i$.
For more details, please refer to Lemma 4.1 in \cite{BHL} or Lemma 3.3 in \cite{BBP}.
\end{remark}

We now begin to study the existence of the viscosity solution of   system \eqref{equ4.1}.

\begin{theorem}
\label{th4.1}\sl
Assume that $(\mathbf{H1})$-$(\mathbf{H3})$ hold, and  $\lambda\in \Big(0,\frac1{(m-1)T} \Big)$. Then the
function $\mathbf{W}(t,x)=\big(W_1(t,x),W_2(t,x),\cdots,W_m(t,x)\bigl),\  (t,x)\in [0,T]\times
\mathbb{R}^n$,  defined through \eqref{equ3.08}, is a viscosity solution of the system  \eqref{equ4.2} of coupled HJBI equations.
\end{theorem}

We first make some preparations in order to prove Theorem \ref{th4.1}.

Let $(t,x,i)\in[0,T)\times\mathbb{R}^n\times \mathbf{M}$. For any given $\delta>0$ with $t+\delta\leq T$,   defining  a stopping time  %
$
\displaystyle \tau^\delta:= (t+\delta)\wedge \mathop{\rm inf}\big\{s\geq t:\sum_{l=1}^{m-1}lN\big((t,s]\times\{l\}\big)\neq 0\big\},
$
we consider the following BSDE with two Poisson random measures on the stochastic interval $[ t,\tau ^{\delta }]$:
\begin{equation}
\left\{
\begin{aligned}
&\!dY_{s}^{u,v}   = -\tilde{f}%
_{i}\big(s,X_{s}^{t,x,i;u,v},Y_{s}^{u,v},H_{s}^{u,v},Z_{s}^{u,v},\int_EK_s^{u,v}(e)\rho(X_{s}^{t,x,i;u,v},e)\nu(de),u_{s},v_{s}\big)ds  \\
&\qquad\quad +
\lambda \sum\limits_{l=1}^{m-1}H_{s}^{u,v}(l)ds   +Z_{s}^{u,v}dB_{s}+\sum\limits_{l=1}^{m-1}H_{s}^{u,v}(l)d\tilde{N}%
_{s}(l)+\int_EK_s^{u,v}(e)\tilde{\mu}(ds,de), \  s\in [ t,\tau ^{\delta }] , \\
&\!Y_{\tau ^{\delta }}^{u,v}   =  \phi \big(\tau ^{\delta },X_{\tau ^{\delta
}}^{t,x,i;u,v}\big)\textbf{1}_{\{N_{\tau ^{\delta }}^{t,i}=i\}}+\sum\limits_{l=1}^{m-1}W_{(l+i)\mbox{mod}(m)}\big(\tau ^{\delta },X_{\tau ^{\delta }}^{t,x,i;u,v}\bigl)\textbf{1}_{\{N_{\tau
^{\delta }}^{t,i}=(l+i) \mbox{mod} (m)\}},
\end{aligned}
\right.  \label{au-equ}
\end{equation}
where $\phi \in C_{l,b}^3\big([0,T]\times \mathbb{R}^n;\mathbb{R}\big)$, $X^{t,x,i;u,v}$ is the solution of SDE   \eqref{fequ3.1} with $(\tau,\eta)=(t,x)\in [0,T)\times\mathbb{R}^n$.

\smallskip

From Theorem \ref{l1} and using  standard arguments similar to  Theorem 3.4 in \cite{LW-SPA}, we get the following results.
 \begin{lemma}\label{l00} \sl Under the assumptions $(\mathbf{H1})$ and $(\mathbf{H2})$, \eqref{au-equ} has a unique solution $(Y^{u,v},H^{u,v},Z^{u,v},K^{u,v})\in\mathcal{S}^{2}[t,\tau^\delta]$. Moreover,
for any $p\geq2$, there exists a sufficient small $\tilde\delta>0$ such that for $\delta\in [0,\tilde\delta]$, we have
 \begin{equation} \begin{aligned}
{\rm (i)} &\  \mathbb{E}\Big[\sup\limits_{s\in[t,\tau^\delta]}|Y_s^{u,v}|^p
+\Big(\int_t^{\tau^\delta}\lambda\sum\limits_{l=1}^{m-1}|H_s^{u,v}(l)|^2ds\Big)^\frac{p}{2}
+\Big(\int_t^{\tau^\delta}|Z_s^{u,v}|^2ds\Big)^\frac{p}{2}\\
&\hskip0.3cm
+ \Big(\int_t^{\tau^\delta}\int_E|K_s^{u,v}(e)|^2\nu(de)ds\Big)^\frac{p}{2} \mid\mathcal {F}_t\Big]\leq C_p(1+|x|^p),\ P\text{-a.s.,}\\
{\rm (ii)} &\  \mathbb{E}\Big[ \Big(\lambda\int_t^{{\tau^\delta}}\sum\limits_{l=1}^{m-1}|H_s^{u,v}(l)|^2ds\Big)^\frac p2+\Big(\int_t^{\tau^\delta}|Z_s^{u,v}|^2 ds\Big)^\frac p2 +\Big(\int_t^{\tau^\delta}\int_E|K_s^{u,v}(e)|^2\nu(de)ds\Big)^\frac p2\mid\mathcal {F}_t\Big]\\
&\hskip0.3cm\leq C_p\delta^\frac p2(1+|x|^p),\ P\text{-a.s}.  \end{aligned}\end{equation}

\end{lemma}
In the sequel, we focus on the stopping time $\tau^\delta$ with $\delta$ restricting to time interval $[0,\tilde\delta]$.
For each $i\in \mathbf{M}$ and  any $(s,x,y,h,z,k,u,v)\in[0,T]\times \mathbb{R}^n\times \mathbb{R}\times \big[L^{2}(\mathbf{L};\mathbb{R})\big]^{m-1}\times\mathbb{R}^d\times L^2_\nu\big(E,\mathcal {B}(E);\mathbb{R}\big)\times U\times V $, we define
\begin{equation*}
\begin{aligned}
&F_i(s, x,y,h,z,k,u,v) := \frac{\partial \phi }{\partial s}(s,  x)+ A^i_{u,v}\phi(s,x)+ B^i_{u,v}\phi(s,x) +\tilde f_i\Big(s,x,y+\phi(s,x),\\
&\quad h+{\mathbf{W}}_{(i+\cdot) \mbox{mod}(m)}(s, x)-\phi (s,  x) \textbf{1},z+D\phi (s,  x)\sigma _{i}(s, x, u,v),
 \int_Ek(e)\rho(x,e)\nu(de)+C^i_{u,v}\phi(s,x),u,v\Big),\\
\end{aligned}\end{equation*}
where
$
\displaystyle
h+{\mathbf{W}}_{(i+\cdot)\mbox{mod}(m)}(s, x)-\phi (s,  x)\textbf{1}=\big(h_{j}+{W}%
_{(i+j)\mbox{mod}(m)}(s, x)-\phi (s,  x)\big)_{1\leq j\leq m-1}.
$

\begin{remark}\label{Re000}\sl
It is easy to check that
there exists some constant $C>0$, such that for any $(x,y,h,z,k)$,
$(x',y',h',z',k')\in \mathbb{R}^n\times \mathbb{R}\times \big[L^{2}(\mathbf{L};\mathbb{R})\big]^{m-1}\times \mathbb{R}^d\times L^2_\nu(E,\mathcal {B}(E);\mathbb{R})$, $(s,u,v)\in [0,T]\times U\times V$,   $i\in \mathbf{M}$,
$$\begin{aligned}
{\rm(i)} \ & \big|F_i(s,x,y,h,z,k,u,v)-F_i(s,x',y',h',z',k',u,v)\big|\\
& \leq C\Big(1+|x|+|x'|+|x|^2  + \int_E|k(e)|(1\wedge|e|)\nu(de) \Big)|x-x'|   \\
&\quad+C\Big(|y-y'|+\sum_{l=1}^{m-1}| h(l)-h'(l)|+|z-z'|+\int_E|k(e)-k'(e)| (1\wedge|e|)\nu(de)\Big);\\
{\rm(ii)} \ &\big |F_i(s,x,y,h,z,k,u,v)\big|\leq C\Big(1+|x|+|x|^2+|y|+\sum_{l=1}^{m-1}| h(l)|+|z|+\int_E|k(e)| (1\wedge|e|)\nu(de)\Big).
\end{aligned}$$

\end{remark}

%

For each  $i\in \mathbf{M}$, we consider
\begin{equation}
\left\{
\begin{aligned}
&\! dY_{s}^{1,i,u,v}  =
-F_i\Big(s,X_{s}^{t,x,i;u,v}, {Y_{s}^{1,i,u,v}},H_{s}^{1,i,u,v},Z_{s}^{1,i,u,v},K_{s}^{1,i,u,v},u_{s},v_{s}\Big)ds+\lambda \sum\limits_{l=1}^{m-1}H_{s}^{1,i,u,v}(l)ds
   \\
  &\!\qquad\qquad\quad +Z_{s}^{1,i,u,v}dB_{s}+\sum\limits_{l=1}^{m-1}H_{s}^{1,i,u,v}(l)d\tilde{%
N}_{s}(l)+\int_EK_{s}^{1,i,u,v}(e)\tilde{\mu}(ds,de),\ \ s\in [ t,\tau ^{\delta }],    \\
& \!  Y_{\tau ^{\delta }}^{1,i,u,v}  =  0,
\end{aligned}
\right.  \label{equ4.8}
\end{equation}%
    and we  define
 {
\begin{equation}
\displaystyle \tilde Y_{s}^{1,i,u,v} :=  \phi
(s,X_{s}^{t,x,i;u,v})+\sum\limits_{l=1}^{m-1}\int_{t}^{s}\Big(W_{(l+i)\mbox{mod}%
(m)}\big(r,X_{r}^{t,x,i;u,v}\big)-\phi \big(r,X_{r}^{t,x,i;u,v}\big)\Big)dN_{r}(l),\ s\in [ t,\tau ^{\delta }].
\label{ee4.7}
\end{equation}
}
From Remark \ref{Re000} and Theorem \ref{l1}, it is easy to check that \eqref{equ4.8} admits a unique solution $(Y^{1,i,u,v},H^{1,i,u,v},$ $Z^{1,i,u,v},$ $K^{1,i,u,v})\in \mathcal{S}^{2}[t,\tau^\delta]$.
\begin{lemma}
\label{l4.1}  \sl Let  $(\mathbf{H1})$, $(\mathbf{H2})$ hold.
For $(t,x,i)\in [0,T)\times\mathbb{R}^n\times\mathbf{M},\ (u,v)\in\mathcal{U}_{t,\tau_\delta}\times\mathcal{V}_{t,\tau_\delta}$, we have
$$Y_{s\wedge\tau^\delta}^{1,i,u,v}=G_{s\wedge\tau^\delta, \tau^\delta}^{t,x,i;u,v}[Y_{\tau^\delta}^{u,v}]-\tilde Y_{s\wedge\tau^\delta}^{1,i,u,v},\ s\in[t,\tau^\delta].   $$

\end{lemma}

\begin{proof}
When $s=\tau^\delta$, it follows from \eqref{ee4.7} and \eqref{au-equ} that
\begin{equation}\label{equ111}
\begin{array}{lll}
\tilde Y_{\tau^\delta}^{1,i,u,v}  = \phi
(\tau^\delta,X_{\tau^\delta}^{t,x,i;u,v})+\sum\limits_{l=1}^{m-1} \big(W_{(l+i)\mbox{mod}%
(m)}(\tau^\delta,X_{\tau^\delta}^{t,x,i;u,v})-\phi (\tau^\delta,X_{\tau^\delta}^{t,x,i;u,v})\big)   \Delta N_{\tau^\delta}(l)\\
=\phi(\tau^\delta,X_{\tau^\delta}^{t,x,i;u,v})\textbf{1}_{\{N_{\tau^\delta}^{t,i}=i\}}+\sum\limits_{l=1}^{m-1}  W_{(l+i)\mbox{mod}%
(m)}(\tau^\delta,X_{\tau^\delta}^{t,x,i;u,v}) \textbf{1}_{\{N_{\tau^\delta}^{t,i}=(l+i)\mbox{mod}(m)\}}  =Y_{\tau^\delta}^{u,v},
\end{array}%
\end{equation}%
which implies that $Y_{\tau^\delta}^{u,v}-\tilde Y_{\tau^\delta}^{1,i,u,v} =0=Y_{\tau ^{\delta }}^{1,i,u,v}.$

By  applying It\^{o}'s formula to $\tilde Y_{s}^{1,i,u,v}-Y_s^{u,v}$,  $s\in[t,\tau^\delta)$, and noting that $N_s^{t,i}=i$, we get
\begin{equation*}
\begin{aligned}
&d\big(\tilde Y_{s}^{1,i,u,v} -Y_s^{u,v}\big)  =   \Big[\frac{\partial \phi }{\partial s}(s,X_{s}^{t,x,i;u,v}) +  b_{i}(s,X_{s}^{t,x,i;u,v},u_{s},v_{s})D\phi(s,X_{s}^{t,x,i;u,v})\\
&\quad
+\frac{1}{2} \tr\big({\sigma}_{i}{\sigma}_{i}^{\ast }(s,X_{s}^{t,x,i;u,v},u_{s},v_{s})D^2\phi (s,X_{s}^{t,x,i;u,v})\big)+B^i_{u_s,v_s}\phi(s,X_{s}^{t,x,i;u,v})  \\
 %
&\quad  +\tilde{f}%
_{i}\Big(s,X_{s}^{t,x,i;u,v},Y_{s}^{u,v},H_{s}^{u,v},Z_{s}^{u,v},\int_EK_{s}^{u,v}(e)\rho(X_{s}^{t,x,i;u,v},e)\nu(de),u_{s},v_{s}\Big)\Big]ds \\
 & \quad    -\lambda \sum\limits_{l=1}^{m-1}\Big[H_{s}^{u,v}(l)-W_{(l+i)\mbox{mod}(m)}(s,X_{s}^{t,x,i;u,v})+\phi (s,X_{s}^{t,x,i;u,v})\Big]ds \\
&  \quad     -\Big[Z_{s}^{u,v}-D\phi (s,X_{s}^{t,x,i;u,v}){\sigma}%
_{i}(s,X_{s}^{t,x,i;u,v},u_{s},v_{s})\Big]dB_{s} \\
 &  \quad -\sum\limits_{l=1}^{m-1}\Big[H_{s}^{u,v}(l)-W_{(l+i)\mbox{mod}(m)}(s,X_{s}^{t,x,i;u,v})+\phi (s,X_{s}^{t,x,i;u,v})\Big]d\tilde{N}_{s}(l)\\
 &\quad  -\int_E\Big[K_s^{u,v}(e)-\phi \big(s,X_{s}^{t,x,i;u,v}+\gamma_i(s,X_{s-}^{t,x,i;u,v},u_s,v_s,e)\big)+\phi(s,X_{s}^{t,x,i;u,v})\Big]\tilde{\mu}(ds,de).%
\end{aligned}%
\end{equation*}%
Compared with \eqref{equ4.8}, the uniqueness of the solution of \eqref{equ4.8} implies that, for any $s\in [t,\tau ^{\delta })$, P-a.s.,
\begin{equation}  \label{eezh}
\begin{aligned}
&Y_{s}^{1,i,u,v}=Y_{s}^{u,v}- \tilde Y_{s}^{1,i,u,v} , \\
&Z_{s}^{1,i,u,v}=Z_{s}^{u,v}-D\phi (s,X_{s}^{t,x,i;u,v}){\sigma}%
_{i}(s,X_{s}^{t,x,i;u,v},u_{s},v_{s}), \\
&H_{s}^{1,i,u,v}(l)=H_{s}^{u,v}(l)-W_{(l+i)\mbox{mod}(m)}(s,X_{s}^{t,x,i;u,v})+\phi (s,X_{s}^{t,x,i;u,v}),\ 1\leq l\leq m-1,\\
&K_{s}^{1,i,u,v}(e)=K_{s}^{u,v}(e)-\phi\big(s,X_{s}^{t,x,i;u,v}+\gamma_i(s,X_{s-}^{t,x,i;u,v},u_s,v_s,e)\big)+\phi(s,X_{s}^{t,x,i;u,v}),\  e\in E.
\end{aligned}%
\end{equation}%
Furthermore,  the definition of the backward semigroup (Definition \ref{Def-bs}) implies
\begin{equation}\label{2021110301}
Y_{s\wedge\tau^\delta}^{1,i,u,v}=Y_{s\wedge\tau^\delta}^{u,v}- \tilde Y_{s\wedge\tau^\delta}^{1,i,u,v}=G_{s\wedge\tau^\delta, \tau^\delta}^{t,x,i;u,v}[Y_{\tau^\delta}^{u,v}]-\tilde Y_{s\wedge\tau^\delta}^{1,i,u,v},\ s\in[t,\tau^\delta] .
\end{equation}

\end{proof}

\begin{lemma}\sl
 For each $i\in \mathbf{M}$ and 
  for all $u\in \mathcal{U}_{t,\tau ^{\delta }},\ v \in
\mathcal{V}_{t,\tau ^{\delta }}$, we have
\begin{equation}\label{ee444}
 \mathbb{E}\Big[\int_t^{\tau^\delta}\Big(|Y_s^{1,i,u,v}|+\lambda\sum_{l=1}^{m-1}|H_s^{1,i,u,v}(l)|+|Z_s^{1,i,u,v}|
 +\int_E|K_s^{1,i,u,v}(e)|(1\wedge|e|)\nu(de)\Big)ds\mid \mathcal {F}_t\Big]\leq C\delta^\frac 54,\  \mbox{P-a.s.},
\end{equation}
for some constant $C>0$.
\end{lemma}

\begin{proof}
From \eqref{2021110301} and \eqref{ee4.7} we know
\begin{equation*}
\begin{aligned}
& Y_{s\wedge\tau^\delta}^{1,i,u,v}   =  Y_{s\wedge\tau^\delta}^{u,v} -\tilde Y_{s\wedge\tau^\delta}^{1,i,u,v}= \displaystyle - \mathbb{E}\big[Y_{ \tau^\delta}^{ u,v} -Y_{s\wedge\tau^\delta}^{u,v}\mid\mathcal {F}_{s\wedge\tau^\delta}\big]+\mathbb{E}\big[\tilde Y_{ \tau^\delta}^{1,i,u,v}-\tilde Y_{s\wedge\tau^\delta}^{1,i,u,v}\mid\mathcal {F}_{s\wedge\tau^\delta}\big]\\
&= \displaystyle  \mathbb{E}\Big[ \int_{s\wedge\tau^\delta}^{\tau ^{\delta
}}\tilde{f}%
_{i}\Big(r,X_{r}^{t,x,i;u,v},Y_{r}^{u,v},H_{r}^{u,v},Z_{r}^{u,v},\int_EK_{r}^{u,v}(e)\rho(X_{r}^{t,x,i;u,v},e)\nu(de),u_{r},v_{r}\Big)dr\\
&\hskip0.8cm-
\int_{s\wedge\tau^\delta}^{\tau ^{\delta}}\lambda \sum\limits_{l=1}^{m-1}H_{r}^{u,v}(l)dr\mid
\mathcal{F}_{s\wedge\tau^\delta}\Big]   +\mathbb{E}\Big[\phi \big(\tau ^{\delta },X_{\tau ^{\delta }}^{t,x,i;u,v}\big) -\phi\big(s\wedge\tau^\delta,X_{s\wedge\tau^\delta}^{t,x,i;u,v}\big)\mid\mathcal{F}_{s\wedge\tau^\delta}\Big]&  \\
&  \ \ \  \displaystyle+\mathbb{E}\Big[\sum\limits_{l=1}^{m-1}\int_{s\wedge\tau^\delta}^{\tau^\delta}\Big(W_{(l+i) %
\mbox{mod}(m)}(r,X_{r}^{t,x,i;u,v})-\phi (r,X_{r}^{t,x,i;u,v})\Big)dN_{r}(l)\mid
\mathcal{F}_{s\wedge\tau^\delta}\Big],\ \ s\in [t,t+\delta]. &
\end{aligned}
\end{equation*}
Therefore, from \eqref{ee222} and Lemma \ref{l00}, it holds
\begin{equation}\label{ee4.4}
\begin{aligned}
  &|Y_{s\wedge\tau^\delta}^{1,i,u,v}| \\
  %
  \leq & \displaystyle C\mathbb{E}\Big[\int_{s\wedge\tau^\delta}^{\tau ^{\delta}}\Big(1+|X_{r}^{t,x,i;u,v}|+|Y_r^{u,v}|+\sum%
\limits_{l=1}^{m-1}|H_r^{u,v}(l)|+|Z_r^{u,v}|+\int_E|K_{r}^{u,v}(e)|(1\wedge|e|)\nu(de)\Big)dr\mid \mathcal{F}_{s\wedge\tau^\delta}\Big] &  \\
  & \displaystyle+C\delta+C\mathbb{E}\big[|X_{\tau ^{\delta}}^{t,x,i;u,v}-X_{s\wedge\tau^\delta}^{t,x,i;u,v}|\mid \mathcal{F}_{s\wedge\tau^\delta}\big]\\
   & \displaystyle+\lambda \mathbb{E}\Big[\sum\limits_{l=1}^{m-1}\int_{{s\wedge\tau^\delta}}^{\tau^\delta}\Big(W_{(l+i) \mbox{mod}(m)}(r,X_{r}^{t,x,i;u,v})-\phi (r,X_{r}^{t,x,i;u,v})\Big)dr\mid \mathcal{F}_{s\wedge\tau^\delta}\Big] &  \\
    \leq &\displaystyle C\delta^{\frac{1}{2}}\Big(\mathbb{E}\Big[\int_{{s\wedge\tau^\delta}}^{
\tau^\delta}\big(1+|Y_r^{u,v}|^2+\sum%
\limits_{l=1}^{m-1}|H_r^{u,v}(l)|^2+|Z_r^{u,v}|^2+\int_E|K_{r}^{u,v}(e)|^2\nu(de)\big)dr\mid \mathcal{F}_{s\wedge\tau^\delta}\Big]\Big)^{%
\frac{1}{2}} &  \\
&    \displaystyle +C\delta+C\delta^{\frac{1}{2}}\big(1+|X_{s\wedge\tau^\delta}^{t,x,i;u,v}|\big)   &  \\
  \leq &\displaystyle  C\delta^{\frac{1}{2}}\big(1+|X_{s\wedge\tau^\delta}^{t,x,i;u,v}|\big),\ \ \mbox{P-a.s.},\ s\in
[t,\tau^\delta]. &
\end{aligned}%
\end{equation}
From \eqref{eezh}, we get, for any $s\in%
[t,\tau^\delta]$,  $P$-a.s.,
\begin{equation}  \label{ee4.5}
\begin{aligned}
&    |Z_s^{1,i,u,v}|\leq
C\big(1+|X_{s}^{t,x,i;u,v}|+|Z_s^{u,v}|\big), &  \\
&  |H_s^{1,i,u,v}(l)|\leq
C\big(1+|X_{s}^{t,x,i;u,v}|+|H_s^{u,v}(l)|\big),\ 1\leq l\leq m-1,\\
&  |K_s^{1,i,u,v}(e)|\leq
C\big(1+|X_{s}^{t,x,i;u,v}|+|K_s^{u,v}(e)|\big),\ e\in E.&
\end{aligned}%
\end{equation}
Applying It\^o's formula to $|Y_s^{1,i,u,v}|^2$,  we obtain
\begin{equation}  \label{ee4.6.1}
\begin{array}{llll}
\displaystyle  |Y_t^{1,i,u,v}|^2+\mathbb{E}\Big[ \int_t^{\tau ^{\delta}}\Big(
 \lambda\sum\limits_{l=1}^{m-1}|H_r^{1,i,u,v}(l)|^2 +|Z_r^{1,i,u,v}|^2+
\int_E |K_r^{1,i,u,v}(e)|^2\nu(de)\Big) dr\mid
\mathcal{F}_t\Big]   \\
 =  \displaystyle2\mathbb{E}\Big[\int_t^{\tau^\delta}Y_r^{1,i,u,v}
 F_i\big(r,X_{r}^{t,x,i;u,v}, {Y_r^{1,i,u,v}},H_r^{1,i,u,v},Z_r^{1,i,u,v},K_r^{1,i,u,v},u_r,v_r\big)dr\mid
\mathcal{F}_t\Big]\\
 \ \ \ \displaystyle-2\lambda\mathbb{E}\Big[\int_t^{\tau ^{\delta}} Y_r^{1,i,u,v}\sum\limits_{l=1}^{m-1}H_r^{1,i,u,v}(l)dr\mid
\mathcal{F}_t\Big]    \\
\leq  \displaystyle 2\mathbb{E}\Big[\int_t^{\tau^\delta}Y_r^{1,i,u,v}F_i\big(r,X_{r}^{t,x,i;u,v},
 {Y_r^{1,i,u,v}},H_r^{1,i,u,v},Z_r^{1,i,u,v},K_r^{1,i,u,v},u_r,v_r\big)dr\mid
\mathcal{F}_t\Big]\\
 \ \ \ \displaystyle+C\mathbb{E}\Big[\int_t^{\tau ^{\delta}} |Y_r^{1,i,u,v}|^2dr\mid
\mathcal{F}_t\Big]+\frac{\lambda}{2}\mathbb{E}\Big[\int_t^{\tau ^{\delta}}\sum\limits_{l=1}^{m-1}|H_r^{1,i,u,v}(l)|^2dr\mid
\mathcal{F}_t\Big].
\end{array}%
\end{equation}
Thus, combing this estimate with Remark \ref{Re000}, \eqref{ee4.4} and \eqref{ee4.5} as well as Lemma \ref{l00} yields
\begin{equation}\label{ee4.13}
\begin{array}{llll}
 \displaystyle \mathbb{E}\Big[ \int_t^{\tau ^{\delta}}\Big(\lambda
 \sum\limits_{l=1}^{m-1}|H_r^{1,i,u,v}(l)|^2 + |Z_r^{1,i,u,v}|^2 + \int_E|K_r^{1,i,u,v}(e)|^2\nu(de)\Big) dr\mid
\mathcal{F}_t\Big]   \\
  \leq \displaystyle C\mathbb{E}\Big[\int_t^{\tau ^{\delta}}Y_r^{1,i,u,v}F_i\big(r,X_{r}^{t,x,i;u,v},
 {Y_r^{1,i,u,v}},H_r^{1,i,u,v},Z_r^{1,i,u,v},K_r^{1,i,u,v},u_r,v_r\big)dr\mid
\mathcal{F}_t\Big] \\
\quad \displaystyle +C\mathbb{E}\Big[\int_t^{\tau ^{\delta}} |Y_r^{1,i,u,v}|^2dr\mid
\mathcal{F}_t\Big]\\
 \leq \displaystyle C\mathbb{E}\Big[\int_t^{\tau ^{\delta}}|Y_r^{1,i,u,v}|\Big(1+|X_{r}^{t,x,i;u,v}|+|X_{r}^{t,x,i;u,v}|^2+{|Y_r^{1,i,u,v}|
} +
\sum\limits_{l=1}^{m-1}|H_r^{1,i,u,v}(l)| \\
\ \displaystyle\qquad\ \ +|Z_r^{1,i,u,v}|+\int_E|K_r^{1,i,u,v}(e)|(1\wedge|e|)\nu(de)\Big)dr\mid \mathcal{F}_t\Big] +C\mathbb{E}\Big[\int_t^{\tau ^{\delta}} |Y_r^{1,i,u,v}|^2dr\mid
\mathcal{F}_t\Big]\\
  \leq  \displaystyle C\delta^{\frac{1}{2}}\mathbb{E}\Big[\int_t^{\tau ^{\delta}}\big(1+|X_{r}^{t,x,i;u,v}|+|X_{r}^{t,x,i;u,v}|^2+|X_{r}^{t,x,i;u,v}|^3\big)dr\mid \mathcal{F}_t\Big] \\
 \ \ \ \displaystyle { +C\delta^{\frac{1%
}{2}}\mathbb{E}\Big[ \int_t^{\tau ^{\delta}}\big(1+|X_{r}^{t,x,i;u,v}|\big)\Big(\sum
\limits_{l=1}^{m-1}|H_r^{u,v}(l)|+|Z_r^{u,v}| + \int_E|K_r^{u,v}(e)|(1\wedge|e|)\nu(de)\Big)dr\mid \mathcal{F}_t\Big] } \\
\ \ \ \displaystyle {+C\delta \mathbb{E}\Big[\int_t^{\tau ^{\delta}}\big(1+|X_{r}^{t,x,i;u,v}|^2\big)dr \mid \mathcal{F}_t\Big]  \leq  C\delta^\frac{3}{2},}
\end{array}%
\end{equation}
{where we use the estimate \eqref{ee222}-(i) in the last inequality.
Therefore,} from \eqref{ee4.4} and \eqref{ee4.13} we get
$$
\begin{array}{llll}
 \displaystyle \mathbb{E}\Big[\int_t^{\tau ^{\delta}}\Big(|Y_r^{1,i,u,v}|+\lambda
\sum\limits_{l=1}^{m-1}|H_r^{1,i,u,v}(l)| +|Z_r^{1,i,u,v}|+
\int_E|K_r^{1,i,u,v}(e)|(1\wedge|e|)\nu(de)\Big)dr\mid \mathcal{F}_t\Big]    \\
  \leq  \displaystyle C\delta^{\frac{1}{2}}\mathbb{E}\Big[\int_t^{\tau ^{\delta}}\big(1+|X_{r}^{t,x,i;u,v}|\big)dr\mid
\mathcal{F}_t\Big]+C\delta^{\frac{1}{2}}\Big(\mathbb{E}\Big[
\int_t^{\tau ^{\delta}}|Z_r^{1,i,u,v}|^2dr\mid \mathcal{F}_t\Big]\Big)^{\frac{1}{2}}   \\
\ \   \displaystyle +C\delta^{\frac{1}{2}}\Big(\mathbb{E}\Big[\int_t^{\tau ^{\delta}}
\sum\limits_{l=1}^{m-1}|H_r^{1,i,u,v}(l)|^2dr\mid \mathcal{F}_t\Big]\Big)^{\frac{1}{2}}+C\delta^{\frac{1}{2}}\Big(\mathbb{E}\Big[\int_t^{\tau ^{\delta}}
\int_E |K_r^{1,i,u,v}(e)|^2\nu(de)dr\mid \mathcal{F}_t\Big]\Big)^{\frac{1}{2}}
  \\
 \leq  \displaystyle  C\delta^{\frac{5}{4}},\ \ \mbox{P-a.s.}
\end{array}%
$$
\end{proof}

Replacing $X_{s}^{t,x,i;u,v}$ in equation \eqref{equ4.8} by $x$, we get the following new equation
\begin{equation}
\left\{
\begin{aligned}
&\!dY_{s}^{2,i,u,v}   =   -F_i\big(s,x,Y_{s}^{2,i,u,v} ,H_{s}^{2,i,u,v},Z_{s}^{2,i,u,v},K_{s}^{2,i,u,v},u_{s},v_{s}\big)ds+\lambda
\sum\limits_{l=1}^{m-1}H_{s}^{2,i,u,v}(l)ds &  \\
&\!\qquad\qquad\quad   +Z_{s}^{2,i,u,v}dB_{s}+\sum\limits_{l=1}^{m-1}H_{s}^{2,i,u,v}(l)d\tilde{%
N}_{s}(l)+\int_EK_{s}^{2,i,u,v}(e)\tilde\mu(ds,de),\ \ s\in [ t,\tau ^{\delta }],   \\
&\!Y_{\tau ^{\delta }}^{2,i,u,v}   =     0.
\end{aligned}%
\right.  \label{equu2}
\end{equation}%
\smallskip
According to Theorem \ref{l1},  \eqref{equu2} admits a unique solution $(Y^{2,i,u,v},H^{2,i,u,v},Z^{2,i,u,v},K^{2,i,u,v}) \in \mathcal{S}^2[t,\tau^\delta]$.

\begin{lemma}
\label{l4.3}\sl  For each $i\in \mathbf{M}$,  and all $u\in \mathcal{U}_{t,\tau ^{\delta }},\ v \in
\mathcal{V}_{t,\tau ^{\delta }}$,   it holds, $P$-a.s.,
\begin{equation}\label{042301}
\begin{aligned}
 &{\rm(i)} \  |Y_{t}^{1,i,u,v}-Y_{t}^{2,i,u,v}|\leq C\delta ^{\frac{5}{4}},\\
&{\rm(ii)} \ \mathbb{E}\Big[\int_t^{\tau^\delta}\Big(|Y_s^{1,i,u,v}-Y_s^{2,i,u,v}|^2+
\lambda\sum_{l=1}^{m-1}|H_s^{1,i,u,v}(l)-H_s^{2,i,u,v}(l)|^2+|Z_s^{1,i,u,v}-Z_s^{2,i,u,v}|^2\\
&\hskip2cm +\int_E|K_s^{1,i,u,v}(e)-K_s^{2,i,u,v}(e)|^2\nu(de)\Big)\mid  \mathcal {F}_t \Big]\leq C\delta^2.
\end{aligned}
\end{equation}%
\end{lemma}

\begin{proof}
Firstly, from \eqref{ee4.5}, \eqref{ee222}-(i) and Lemma \ref{l00} , we have
\begin{equation}\label{ee00-1}\begin{aligned}
&\mathbb{E}\Big[\Big(\int_t^{\tau^\delta} \int_E|K_{s}^{1,i,u,v} (e)|^2\nu(de)  ds\Big)^2\mid  \mathcal {F}_t \Big]\\
%
\leq& C\mathbb{E}\Big[\Big(\int_t^{\tau^\delta}  \big(1+|X_{s}^{t,x,i;u,v}|^2\big)  ds\Big)^2+\Big(\int_t^{\tau^\delta} \int_E|K_s^{u,v}(e)|^2\nu(de) ds\Big)^2  \mid  \mathcal {F}_t \Big]\leq  C\delta^2(1+|x|^4).
\end{aligned}
\end{equation}
For the equations \eqref{equ4.8} and \eqref{equu2}, by setting $\xi_1=\xi_2=0$  and $\varphi_2(s)=0,$
$\varphi_1(s)=F_i(s,X_s^{t,x,i;u,v},Y_{s}^{1,i,u,v}\textcolor[rgb]{1.00,0.00,0.00}{+\Lambda_s} ,$ $H_{s}^{1,i,u,v},Z_{s}^{1,i,u,v},K_{s}^{1,i,u,v},u_{s},v_{s})-F_i(s,x,Y_{s}^{1,i,u,v} ,H_{s}^{1,i,u,v},Z_{s}^{1,i,u,v},K_{s}^{1,i,u,v},u_{s},v_{s})$
in Lemma \ref{l2}, from Remark \ref{Re000}, \eqref{ee00-1}  and  \eqref{ee222}, we get
\begin{equation}\label{2021011701}
\begin{aligned}
&\mathbb{E}\Big[\int_t^{\tau^\delta}\Big(|Y_s^{1,i,u,v}-Y_s^{2,i,u,v}|^2
+\lambda\sum_{l=1}^{m-1}|H_s^{1,i,u,v}(l)-H_s^{2,i,u,v}(l)|^2+|Z_s^{1,i,u,v}-Z_s^{2,i,u,v}|^2\\
&\hskip1.3cm +\int_E|K_s^{1,i,u,v}(e)-K_s^{2,i,u,v}(e)|^2\nu(de)\Big)ds\mid  \mathcal {F}_t \Big]\\
&\leq C\mathbb{E}\Big[\int_t^{\tau^\delta} \big|F_i\big(s,X_s^{t,x,i;u,v}, {Y_{s}^{1,i,u,v}},
H_{s}^{1,i,u,v},Z_{s}^{1,i,u,v},K_{s}^{1,i,u,v},u_{s},v_{s}\big)\\
&\hskip2cm-F_i\big(s,x,Y_{s}^{1,i,u,v} ,H_{s}^{1,i,u,v},Z_{s}^{1,i,u,v},K_{s}^{1,i,u,v},u_{s},v_{s}\big)\big|^2ds    \mid   \mathcal {F}_t \Big]\\
&\leq C\mathbb{E}\Big[\int_t^{\tau^\delta} \Big(1+|x|+|x|^2+|X_s^{t,x,i;u,v}|+\int_E|K_{s}^{1,i,u,v} (e)|(1\wedge|e|) \nu(de) \Big)^2|X_s^{t,x,i;u,v}-x|^2  ds\mid  \mathcal {F}_t \Big]\\
&\leq C\delta \mathbb{E}\Big[\sup\limits_{s\in[t,\tau^\delta]}\big(|X_s^{t,x,i;u,v}-x|^2+|X_s^{t,x,i;u,v}|^2\cdot|X_s^{t,x,i;u,v}-x|^2\big)\mid  {\mathcal {F}_t \Big]}\\
&\quad +C \Bigg(\mathbb{E}\Big[\Big(\int_t^{\tau^\delta} \int_E|K_{s}^{1,i,u,v} (e)|^2\nu(de)ds \Big)^2\mid  \mathcal {F}_t \Big]\Bigg)^\frac12\cdot \Bigg(\mathbb{E}\Big[\sup\limits_{s\in[t,\tau^\delta]} |X_s^{t,x,i;u,v}-x|^4 \mid  \mathcal {F}_t \Big]\Bigg)^\frac12 \\
&\leq C \delta^\frac{3}{2} .
\end{aligned}
\end{equation}
Therefore, using Remark \ref{Re000} again and from \eqref{ee222}, \eqref{ee4.13} and \eqref{2021011701}, we have
$$\begin{aligned}
&\big|Y_t^{1,i,u,v}-Y_t^{2,i,u,v}\big|=\big|\mathbb{E}[Y_t^{1,i,u,v}-Y_t^{2,i,u,v}\mid \mathcal {F}_t]\big|\\
&\leq \mathbb{E}\Big[\int_t^{\tau^\delta}\Big(\big|F_i\big(s,X_s^{t,x,i;u,v},Y_{s}^{1,i,u,v} +\Lambda_s ,H_{s}^{1,i,u,v},Z_{s}^{1,i,u,v},K_{s}^{1,i,u,v},u_{s},v_{s}\big)\\
&\hskip1.3cm-F_i\big(s,x,Y_{s}^{2,i,u,v} ,H_{s}^{2,i,u,v},Z_{s}^{2,i,u,v},K_{s}^{2,i,u,v},u_{s},v_{s}\big)\big|+\lambda\sum\limits_{l=1}^{m-1}|H_{s}^{1,i,u,v}(l)-H_{s}^{2,i,u,v}(l)| \Big)ds\mid  \mathcal {F}_t \Big]\\
&\leq C\mathbb{E}\Big[\int_t^{\tau^\delta}\Big(\big(1+|X_s^{t,x,i;u,v}|+|X_s^{t,x,i;u,v}|^2+|x| +\int_E\big|K_{s}^{1,i,u,v}(e)\big|(1\wedge|e|)\nu(de)\big )\cdot|X_s^{t,x,i;u,v}-x| \\
&\hskip1.5cm+{\big|Y_{s}^{1,i,u,v}-Y_s^{2,i,u,v}\big|+}\sum\limits_{l=1}^{m-1}\big|H_{s}^{1,i,u,v}(l)-H_s^{2,i,u,v}(l)\big|+\big|Z_{s}^{1,i,u,v}-Z_s^{2,i,u,v}\big|\\
&\hskip1.5cm
 +\int_E\big|K_{s}^{1,i,u,v}(e)-K_s^{2,i,u,v}(e)\big|(1\wedge|e|)\nu(de)\Big)ds\mid \mathcal {F}_t \Big]\\
&\leq C\delta\mathbb{E}\Big[\sum_{k=1}^3\sup_{s\in[t,\tau^\delta]}|X_s^{t,x,i;u,v}-x|^k\mid {\mathcal {F}_t \Big]}\\
& \hskip0.3cm +  C\delta^\frac12\Big(\mathbb{E}\Big[\int_t^{\tau^\delta}\int_E|K_{s}^{1,i,u,v}(e)|^2\nu(de) ds\mid \mathcal {F}_t \Big] \Big)^\frac12\Big(\mathbb{E}\Big[\sup_{s\in[t,\tau^\delta]}|X_s^{t,x,i;u,v}-x|^2\mid \mathcal {F}_t \Big]\Big)^\frac12\\
&\hskip0.3cm   +C\delta^\frac12\Big\{\mathbb{E}\Big[\int_t^{\tau^\delta}\Big(|Y_s^{1,i,u,v}-Y_s^{2,i,u,v}|^2+
 \sum_{l=1}^{m-1}|H_s^{1,i,u,v}(l)-H_s^{2,i,u,v}(l)|^2+|Z_s^{1,i,u,v}-Z_s^{2,i,u,v}|^2\\
&\hskip2.5cm +\int_E|K_s^{1,i,u,v}(e)-K_s^{2,i,u,v}(e)|^2\nu(de)\Big)ds\mid \mathcal {F}_t \Big]\Big\}^\frac12\leq C\delta^\frac54.
\end{aligned}$$%
\end{proof}

For the solution  $(Y^{2,i,u,v},H^{2,i,u,v},Z^{2,i,u,v},K^{2,i,u,v})\in \mathcal{S}^2[t,\tau^\delta]$ of BSDE \eqref{equu2},
  we also have the following estimate.

\begin{lemma}
\label{l4.0000}\sl For each $i\in \mathbf{M}$,
for any  $u\in \mathcal{U}_{t,\tau ^{\delta }},\ v \in
\mathcal{V}_{t,\tau ^{\delta }}$,
$$\begin{array}{lll}
&\displaystyle \mathbb{E}\Big[\int_{t}^{\tau ^{\delta
}}\Big( |Y_s^{2,i,u,v}|+\lambda\sum_{l=1}^{m-1}|H_s^{2,i,u,v}(l)|+|Z_s^{2,i,u,v}|+\int_E| K_{s}^{2,i,u,v}(e)|(1\wedge|e|)\nu(de) \Big)ds \mid
\mathcal{F}_{t}\Big] \leq C\delta^\frac 54,\ P\text{-a.s.}
\end{array}$$

\end{lemma}

\begin{proof}
From \eqref{ee444} and \eqref{042301}-(ii) we have
$$\begin{array}{lll}
&\displaystyle \mathbb{E}\Big[\int_{t}^{\tau ^{\delta
}}\Big( |Y_s^{2,i,u,v}|+\lambda\sum_{l=1}^{m-1}|H_s^{2,i,u,v}(l)|+|Z_s^{2,i,u,v}|+\int_E| K_{s}^{2,i,u,v}|(1\wedge|e|)\nu(de) \Big)ds\mid
\mathcal{F}_{t}\Big]\\
& \leq \displaystyle\mathbb{E}\Big[\int_{t}^{\tau ^{\delta
}}\Big(|Y_{s}^{1,i,u,v}|+\lambda \sum%
\limits_{l=1}^{m-1}|H_{s}^{1,i,u,v}(l)| +|Z_{s}^{1,i,u,v}|+ \int_E|K_{s}^{1,i,u,v}|(1\wedge|e|)\nu(de)\Big)ds\mid
\mathcal{F}_{t}\Big]\\
&\ \ \displaystyle+\mathbb{E}\Big[\int_{t}^{\tau ^{\delta
}}\Big(|Y_{s}^{1,i,u,v}-Y_{s}^{2,i,u,v}|+\lambda \sum%
\limits_{l=1}^{m-1}|H_{s}^{1,i,u,v}(l)-H_{s}^{2,i,u,v}(l)|\\
&\displaystyle \hskip1.8cm
+|Z_{s}^{1,i,u,v}-Z_{s}^{2,i,u,v}|
+\int_E|K_{s}^{1,i,u,v}-K_{s}^{2,i,u,v}|(1\wedge|e|)\nu(de)\Big)ds\mid
\mathcal{F}_{t}\Big]\\
 & \leq  \displaystyle C\delta^\frac 54+C\delta^\frac12 \Big(\mathbb{E}\Big[\int_{t}^{\tau ^{\delta
}}\Big(|Y_{s}^{1,i,u,v}-Y_{s}^{2,i,u,v}|^{2} +\lambda\sum%
\limits_{l=1}^{m-1}|H_{s}^{1,i,u,v}(l)-H_{s}^{2,i,u,v}(l)|^{2}\\
&\displaystyle\hskip3cm
+|Z_{s}^{1,i,u,v}-Z_{s}^{2,i,u,v}|^{2}+\int_E|K_{s}^{1,i,u,v}(e)-K_{s}^{2,i,u,v}(e)|^{2}\nu(de)\Big)ds\mid
\mathcal{F}_{t}\Big]\Big)^\frac12\\
 & \leq C\delta^\frac 54.
\end{array}$$
\end{proof}

\begin{lemma}
\label{l4.6} \sl 
Let $(Y^{3,i},H^{3,i})$ be the unique solution of the following BSDE with jumps:
\begin{equation}  \label{ee4.19}
\left\{
\begin{aligned}
&\! dY_{s}^{3,i}  =  -G_i\big(s,x,Y_s^{3,i},H_{s}^{3,i},0,0\big)ds+\sum%
\limits_{l=1}^{m-1}H_{s}^{3,i}(l)\big[d\tilde{N}_{s}(l)+\lambda ds\big],\ \ s\in [
t,\tau ^{\delta }],    \\
&\! Y_{\tau ^{\delta }}^{3,i}   =   0,
\end{aligned}
\right.
\end{equation}%
where the generator $G_i:[  0,T]\times \mathbb{R}^n\times \mathbb{R} \times [L^2(\mathbf{L};\mathbb{R})]^{m-1}\times \mathbb{R}^d\times L^2_\nu\big(E,\mathcal {B}(E);\mathbb{R}\big)\rightarrow \mathbb{R} $ is defined as follows
\begin{equation}
G_i(s,x,y,h,z,k):=\underset{u\in U}{\sup }\underset{v\in V}{\mathop{\rm inf}}%
F_i(s,x,y,h,z,k,u,v).
\end{equation}%
Then, we have $Y_{t}^{3,i } =
\mathop{\rm esssup}\limits_{u\in \mathcal{U}_{t,\tau ^{\delta }}}\mathop{\rm
essinf}\limits_{v\in \mathcal{V}_{t,\tau ^{\delta
}}}Y_{t}^{2,i,u,v},
$  P-a.s.

\end{lemma}

\begin{proof} 
We denote
$
\tilde F_i(s,x,y,h,z,k,u):=\mathop{\rm inf}\limits_{v\in V}F_i (s,x,y,h,z,k,u,v),$ for $(s,x,y,h,z,k,u)\in [  0,T]\times \mathbb{R}^n\times \mathbb{R} \times [L^2(\mathbf{L};\mathbb{R})]^{m-1}\times \mathbb{R}^d\times L^2_\nu\big(E,\mathcal {B}(E);\mathbb{R}\big)\times U.
$
 For $u\in \mathcal{U}_{t,\tau ^{\delta }}$, it follows from Theorem \ref{l1} that the following BSDE with jumps
\begin{equation}\label{ee4.22}
\left\{
\begin{aligned}
&\!dY_{s}^{3,i,u} =   -\tilde F_i\big(s,x,Y_s^{3,i,u},H_{s}^{3,i,u},Z_{s}^{3,i,u},K_s^{3,i,u},u_{s}\big)ds+ \lambda\sum%
\limits_{l=1}^{m-1}H_{s}^{3,i,u}(l)
ds \\
&\!\qquad\qquad +\sum%
\limits_{l=1}^{m-1}H_{s}^{3,i,u}(l) d\tilde{N}_{s}(l) +Z_{s}^{3,i,u}dB_{s}+\int_EK_s^{3,i,u}(e)\tilde{\mu}(ds,de),\ s\in [  t,\tau ^{\delta }],   \\
&\!Y_{\tau ^{\delta }}^{3,i,u}  =  0,
\end{aligned}
\right.
\end{equation}%
has a
unique solution $(Y^{3,i,u},H^{3,i,u},Z^{3,i,u},K^{3,i,u})\in \mathcal{S}^2[t,\tau^\delta]$.
We then claim that
\begin{equation*}
Y_{t}^{3,i,u}=\essinf\limits_{v\in \mathcal{V}_{t,\tau ^{\delta }}}{Y}_{t}^{2,i,u,v},\ \mbox{P-a.s.,\ for all
}u\in \mathcal{U}_{t,\tau ^{\delta }}.
\end{equation*}%
Indeed, from the definition of $\tilde F_i$ and the comparison theorem for  BSDEs with jumps (Theorem \ref{Com-Th}),
we have
\begin{equation*}
Y_{t}^{3,i,u}\leq \essinf\limits_{v\in \mathcal{V}_{t,\tau ^{\delta }}}{Y}_{t}^{2,i,u,v},\ \mbox{P-a.s.,\ for all
}u\in \mathcal{U}_{t,\tau ^{\delta }}.
\end{equation*}%
On the other hand, there exists a Borel measurable function $\bar v:[0,T]\times
\mathbb{R}^n\times
\mathbb{R}\times
[L^2(\mathbf{L};\mathbb{R})]^{m-1}\times
\mathbb{R}^d\times L^2_\nu(E,\mathcal {B}(E);\mathbb{R})\times U\rightarrow V$ such that
\begin{equation*}
\tilde F_i(s,x,y,h,z,k,u)=F_i\big(s,x,y,h,z,k,u,\bar v(s,x,y,h,z,k,u)\big).
\end{equation*}%
Let $u\in \mathcal U_{t,\tau^\delta}$. By setting  $\widetilde{v}_{s} := \bar v(s,x,Y_s^{3,i,u},H_s^{3,i,u},Z_s^{3,i,u},K_s^{3,i,u},u_{s}),\ s\in [  t,\tau
^{\delta }],$ we see $\widetilde{v}\in \mathcal{V}_{t,\tau ^{\delta }}$, and
\begin{equation*}
\tilde F_i(s,x,Y_s^{3,i,u},H_s^{3,i,u},Z_s^{3,i,u},K_s^{3,i,u},u_{s})
=F_i(s,x,Y_s^{3,i,u},H_s^{3,i,u},Z_s^{3,i,u},K_s^{3,i,u},u_{s},\widetilde{v}_{s}),\ s\in
[  t,\tau ^{\delta }].
\end{equation*}%
The uniqueness of the solution of the BSDE with jumps implies that $$(Y^{3,i,u},{H}%
^{3,i,u},{Z}^{3,i,u},K^{3,i,u})=({Y}^{2,i,u,\widetilde{v}},{H}^{2,i,u,\widetilde{v}%
},{Z}^{2,i,u,\widetilde{v}},K^{2,i,u,\widetilde{v}}).$$
Particularly, $Y_{t}^{3,i,u}={Y}%
_{t}^{2,i,u,\widetilde{v}},\ \mbox{P-a.s.}$
Consequently, from the arbitrariness of the choice of $u\in \mathcal U_{t,\tau^\delta}$,  $Y_{t}^{3,i,u}=\essinf\limits_{v\in \mathcal{V}_{t,\tau ^{\delta }}}{Y}_{t}^{2,i,u, {v} },\ \mbox{P-a.s.}$, for
all $u\in \mathcal{U}_{t,\tau ^{\delta }}.$

Finally, as $G_i(s,x,y,h,z,k)=\sup\limits_{u\in {U}}\tilde F_i(s,x,y,h,z,k,u)$,
by a similar argument  we get the desired result
\begin{equation*}
Y_{t}^{3,i}=\esssup\limits_{u\in \mathcal{U}_{t,\tau ^{\delta }}}Y_{t}^{3,i,u}=\esssup\limits_{u\in
\mathcal{U}_{t,\tau ^{\delta }}}\essinf\limits_{v\in \mathcal{%
V}_{t,\tau ^{\delta }}}{Y}_{t}^{2,i,u,{v}%
 },\ \mbox{P-a.s.}
\end{equation*}%

\end{proof}

Note that equation \eqref{ee4.19} is independent of
the Brownian motion $B$ and the Poisson random measure $ \mu$, but it still depends on the Poisson random measure $N$.
Now we extend equation \eqref{ee4.19} to the time interval $[t,t+\delta ] $ as
follows:
\begin{equation}\label{equ4.111}
\left\{
\begin{aligned}
&\!d{Y}_{s}^{0,\delta }=-\mathbf{1}_{[t,\tau ^{\delta
}]}(s)G_i \big(s,x,Y_s^{0,\delta},H_{s}^{0,\delta },0,0\big)ds+\sum\limits_{l=1}^{m-1}H_{s}^{0,\delta }(l)\big[d\tilde{%
N}_{s}(l)+\lambda ds\big],\ s\in [ t,t+\delta ], \\
&\!Y_{t+\delta}^{0,\delta }=0.
\end{aligned}
\right.
\end{equation}
  Obviously, we have
\begin{equation}\label{re4.4}
Y_{s}^{0,\delta }=\left\{
\begin{array}{l}
\!\!\!Y_{s}^{3,i },\ \ \ \ \ \ \ s\in [ t,\tau ^{\delta }], \\
\!\!\!0,\ \ \ \ \ \ \ \ \ \ \  s\in ( \tau ^{\delta },t+\delta ],%
\end{array}%
\right.\ \ \ H_{s}^{0,\delta }=\left\{
\begin{array}{l}
\!\!\!H_{s}^{3,i },\ \ \ \ \ \ \ s\in [ t,\tau ^{\delta }], \\
\!\!\!0,\ \ \ \ \ \ \ \ \ \ \   s\in( \tau ^{\delta },t+\delta ],%
\end{array}%
\right.
\end{equation}%
and $Y_{t}^{0,\delta }=Y_{t}^{3,i}=\mathop{\rm esssup}%
\limits\limits_{u\in \mathcal{U}_{t,\tau ^{\delta }}}\mathop{\rm essinf}%
\limits\limits_{v\in \mathcal{V}_{t,\tau ^{\delta }}}Y_{t}^{2,i,u,v}. $

\smallskip
We consider a last auxiliary differential equation
\begin{equation}  \label{ee4.24}
\left\{
\begin{aligned}
&\!d\bar{Y}_{s}^{0,\delta }=-G_i \big(s,x,\bar{Y}_{s}^{0,\delta },0,0,0\big)ds,\ s\in [ t,t+\delta ], \\
&\!\bar{Y}_{t+\delta }^{0,\delta }=0.
\end{aligned}%
\right.
\end{equation}
The difference between ${Y}^{0,\delta}_t$ and $\bar{Y}^{0,\delta}_t$ is estimated
as follows.

\begin{lemma}
\label{l4.7}\sl  There exists a constant $C>0$ such that,
\begin{equation*}
|{Y}^{0,\delta}_t-\bar{Y}^{0,\delta}_t|\leq C\delta^\frac{3}{2},\ P\text{-a.s.}
\end{equation*}

\end{lemma}

\begin{proof}
As $|G_i(r,x,y,0,0,0)|\leq C(1+|x|+|x|^2+|y|)$, from \eqref{ee4.24}  we have,
 $$|\bar{Y}_{s}^{0,\delta }|\leq C\int_{s}^{t+\delta}\big(1+|x|+|x|^2+|\bar{Y}_{r}^{0,\delta }|\big) dr.$$
The Gronwall inequality implies  that
$$|\bar{Y}_{s}^{0,\delta }|\leq (1+|x|+|x|^2)(e^{C(t+\delta-s)}-1)\leq C(1+|x|+|x|^2)e^{C\delta}(t+\delta-s),\ s\in[t,t+\delta]. $$
Thus,
$$|\bar{Y}_{\tau^\delta}^{0,\delta }| \leq C(1+|x|+|x|^2)e^{C\delta}(t+\delta-\tau^\delta) \leq C(1+|x|+|x|^2)e^{C\delta}\sum\limits_{l=1}^{m-1}\delta\textbf{1}_{\{N((t,t+\delta]\times\{l\})\geq 1\}}, $$
and
$$\mathbb{E}[|\bar{Y}_{\tau^\delta}^{0,\delta }|^2\mid\mathcal {F}_t]\leq  Cm^2\delta^2\Big(1-P\big(N((t,t+\delta]\times \{l\})=0\big)\Big)=Cm^2\delta^2(1-e^{-\lambda \delta})\leq C\lambda m^2\delta^3.$$
Note that as equation \eqref{ee4.24}  is deterministic, it  
 can  be rewritten as follows,
\begin{equation}\label{equ5555}
\left\{
\begin{aligned}
&\!d\bar{Y}_{s}^{0,\delta }=-G_i \big(s,x,\bar{Y}_{s}^{0,\delta },\bar H_s^{0,\delta},0,0\big)ds+\sum\limits_{l=1}^{m-1}\bar H_{s}^{0,\delta }(l)[d\tilde{N%
}_{s}(l)+\lambda ds],\ s\in [ t,\tau^\delta ], \\
 &\!\bar{Y}_{\tau^\delta}^{0,\delta }=\int_{\tau^\delta}^{t+\delta}G_i(s,x,\bar{Y}_{s}^{0,\delta },0,0)ds,
\end{aligned}%
\right.
\end{equation}
with $\bar H_s^{0,\delta}(l)\equiv0,\ 1\leq l\leq m-1,$ on $ [ t,t+\delta ]$.
For the equations \eqref{equ5555} and  \eqref{equ4.111}, by using Lemma \ref{l2}  we get
\begin{equation}\label{1133}\begin{aligned}
& |\bar{Y}_{t}^{0,\delta }-%
Y_{t}^{0,\delta }|^{2}+\lambda\mathbb{E}\Big[\int_{t}^{\tau^\delta}\sum\limits_{l=1}^{m-1}|\bar H_{s}^{0,\delta }(l)-H_{s}^{0,\delta }(l)|^{2}ds\mid\mathcal {F}_t\Big]\\
 &\leq
C\mathbb{E}\big[|\bar{Y}_{\tau^\delta}^{0,\delta }-%
Y_{\tau^\delta}^{0,\delta }|^{2}\mid\mathcal {F}_t\big]=C\mathbb{E}\big[|\bar{Y}_{\tau^\delta}^{0,\delta }|^{2}\mid\mathcal {F}_t\big]\leq   C\lambda m^2\delta^3,\ P\text{-a.s.,}
\end{aligned}
\end{equation}
from which  we can conclude the desired result.

\end{proof}

We now give the proof of Theorem \ref{th4.1}.

\noindent \textbf{Proof of Theorem \ref{th4.1}:}

\noindent\emph{Step 1: viscosity supersolution case.} Let $i\in \mathbf{M}$. Given  any $\phi \in C_{l,b}^{3}([0,T]\times \mathbb{R}^n;\mathbb{R})$ and $(t,x)\in [ 0,T)\times \mathbb{R}^n$, we assume that  for all $(s,y)\in[0,T]\times \mathbb{R}^n$,  $W_i(s,y)-\phi(s,y)\geq W_i(t,x)-\phi(t,x)$. Without loss of generality we
suppose that $\phi (t,x)=W_{i}(t,x)$. Then for any $u\in \mathcal {U}_{t,\tau^\delta}$, $v\in \mathcal {V}_{t,\tau^\delta}$,
\begin{equation*}
\begin{aligned}
&    Y_{\tau^\delta}^{u,v}=\phi \big(\tau ^{\delta },X_{\tau ^{\delta }}^{t,x,i;u,v}\big)\mathbf{1}_{\{N_{\tau
^{\delta }}^{t,i}=i\}}+\sum\limits_{l=1}^{m-1}W_{(l+i)\mbox{mod} (m)}\big(\tau ^{\delta },X_{\tau ^{\delta }}^{t,x,i;u,v}\big)\mathbf{1}_{\{N_{\tau ^{\delta
}}^{t,i}=(l+i)\mbox{mod}(m)\}}    \\
& \leq  W_{i}\big(\tau ^{\delta },X_{\tau ^{\delta }}^{t,x,i;u,v}\big)\mathbf{1}_{\{N_{\tau
^{\delta }}^{t,i}=i\}}+\sum\limits_{l=1}^{m-1}W_{(l+i)\mbox{mod} (m)}\big(\tau ^{\delta },X_{\tau ^{\delta }}^{t,x,i;u,v}\big)\mathbf{1}_{\{N_{\tau ^{\delta
}}^{t,i}=(l+i)\mbox{mod}(m)\}}   \\
& =   W_{N_{\tau ^{\delta }}^{t,i}}\big(\tau ^{\delta },X_{\tau ^{\delta
}}^{t,x,i;u,v}\big).
\end{aligned}%
\end{equation*}
Combining $\tilde Y_t^{1,i,u,v}=\phi(t,x)$ and Lemma \ref{l4.1} with the comparison theorem for   BSDEs with jumps (Theorem \ref{Com-Th}), we obtain,
for all $u\in \mathcal {U}_{t,\tau^\delta}$, $\beta\in \mathcal {B}_{t,\tau^\delta}$,
\begin{equation*}
\begin{aligned}
Y_{t}^{1,i,u,\beta (u)}&=G_{t, \tau^\delta}^{t,x,i;u,\beta (u)}\big[Y_{\tau^\delta}^{u,\beta (u)}\big]- \tilde Y_{t}^{1,i,u,\beta (u)}
\leq  G_{t,\tau ^{\delta }}^{t,x,i;u,\beta (u)}\Big[W_{N_{\tau ^{\delta
}}^{t,i}}\big(\tau ^{\delta },X_{\tau ^{\delta }}^{t,x,i;u,\beta(u)}\big)\Big] -\phi (t,x) \\
&=   G_{t,\tau ^{\delta }}^{t,x,i;u,\beta (u)}\Big[W_{N_{\tau ^{\delta
}}^{t,i}}\big(\tau ^{\delta },X_{\tau ^{\delta }}^{t,x,i;u,\beta(u)}\big)\Big]-W_{i}(t,x).%
\end{aligned}%
\end{equation*}%
Using the strong DPP (Theorem \ref{SDPP}), we have
\begin{equation*}
\mathop{\rm essinf}\limits_{\beta \in \mathcal{B}_{t,\tau ^{\delta }}} %
\mathop{\rm esssup}\limits_{u\in \mathcal{U}_{t,\tau ^{\delta
}}}Y_{t}^{1,i,u,\beta (u)}\leq \mathop{\rm essinf}\limits_{\beta \in
\mathcal{B}_{t,\tau ^{\delta }}} \mathop{\rm esssup}\limits_{u\in \mathcal{U}%
_{t,\tau ^{\delta }}} G_{t,\tau ^{\delta }}^{t,x,i;u,\beta (u)}\Big[W_{N_{\tau
^{\delta }}^{t,i}}\big(\tau ^{\delta }, {X}_{\tau ^{\delta }}^{{%
t,x,i;u,\beta (u)}}\big)\Big]-W_{i}(t,x) = 0,\ \mbox{P-a.s.}
\end{equation*}%
Furthermore, Lemma \ref{l4.3} implies that
$$\mathop{\rm essinf}%
\limits_{\beta \in \mathcal{B}_{t,\tau ^{\delta }}} \mathop{\rm esssup}%
\limits_{u\in \mathcal{U}_{t,\tau ^{\delta }}}{Y}_{t}^{2,i,u,\beta (u)}\leq
\mathop{\rm essinf}%
\limits_{\beta \in \mathcal{B}_{t,\tau ^{\delta }}} \mathop{\rm esssup}%
\limits_{u\in \mathcal{U}_{t,\tau ^{\delta }}}{Y}_{t}^{1,i,u,\beta (u)}+C\delta^\frac54
\leq
C\delta
^{\frac54},\ \mbox{P-a.s.} $$
Consequently, from $\mathop{\rm essinf}%
\limits_{v\in \mathcal{V}_{t,\tau ^{\delta }}}{Y}_{t}^{2,i,u,v}\leq {Y}%
_{t}^{2,i,u,\beta (u)},\ \beta \in \mathcal{B}_{t,\tau ^{\delta }},$ we get
\begin{equation*}
\mathop{\rm esssup}\limits_{u\in \mathcal{U}_{t,\tau ^{\delta }}}\mathop{\rm
essinf}\limits\limits _{v\in \mathcal{V}_{t,\tau ^{\delta }}}{Y}%
_{t}^{2,i,u,v}\leq \mathop{\rm essinf}\limits_{\beta \in \mathcal{B}_{t,\tau
^{\delta }}}\mathop{\rm esssup}_{u\in \mathcal{U}_{t,\tau ^{\delta }}}{Y}%
_{t}^{2,i,u,\beta (u)}\leq C\delta ^{\frac54},\ \mbox{P-a.s.}
\end{equation*}%
It follows from Lemma \ref{l4.6} that
 ${Y}_{t}^{0,\delta}={Y}_{t}^{3,i}=\mathop{\rm esssup}\limits_{u\in \mathcal{U}_{t,\tau ^{\delta }}}\mathop{\rm
essinf}\limits\limits _{v\in \mathcal{V}_{t,\tau ^{\delta }}}{Y}%
_{t}^{2,i,u,v}\leq C\delta ^{\frac54},\ %
\mbox{P-a.s.}$
Furthermore, from   \eqref{re4.4} and Lemma \ref{l4.7}, it holds
$
 \bar{Y}_{t}^{0,\delta} \leq {Y}_{t}^{0,\delta} +C\delta^\frac32\leq C\delta ^{\frac54}.
$
Therefore, from equation \eqref{ee4.24}, we finally obtain
$$
\sup\limits_{u\in \mathcal{U}}\mathop{\rm inf}\limits_{v\in
V}F_i(t,x,0,0,0,0,u,v)=G_i (t,x,0,0,0,0)\leq 0,
$$
that is,
\begin{equation*}\begin{aligned}
&0\geq\frac{\partial \phi}{\partial t}(t,x)+\sup\limits_{u\in U}\mathop{\rm inf}\limits_{v\in
V}\Big\{A_{u,v}^i\phi(t,x)+B_{u,v}^i\phi(t,x)+f_i\big(t,x,\mathbf{W}(t,x),D
\phi(t,x)\sigma_i(t,x,u,v),  C_{u,v}^i\phi(t,x),u,v\big)\Big\},
\end{aligned}\end{equation*}
which shows that $\mathbf{W}$ is a viscosity
supersolution of the system \eqref{equ4.2}.

\vskip2mm
\noindent \emph{Step 2: viscosity subsolution case.}  For $i\in \mathbf{M},$ let $\phi \in C_{l,b}^{3}\big([0,T]\times \mathbb{R}^n;\mathbb{R}\big)$ and $(t,x)\in [ 0,T)\times \mathbb{R}^n$ be  such that, $W_i(s,y)-\phi(s,y)\leq W_i(t,x)-\phi(t,x)=0$, for any $  (s,y)\in[0,T]\times \mathbb{R}^n$.
Due to  Definition \ref{Def111}-(i), we only need to prove
\begin{equation}\label{ee000}
\sup\limits_{u\in {U}}\mathop{\rm inf}\limits_{v\in
V}F_i\big(t,x,0,0,0,0,u,v\big)=G_i \big(t,x,0,0,0,0\big)\geq 0.
\end{equation}
For this, we suppose that \eqref{ee000} does not hold. Then there exists some $\theta >0$
such that
\begin{equation}\label{2021011901}
\sup_{u\in {U}}\mathop{\rm inf}_{v\in
V}F_i\big(t,x,0,0,0,0,u,v\big)=G_i\big (t,x,0,0,0,0\big)\leq -\theta <0,
\end{equation}%
and we can find a measurable function $\chi: U\rightarrow V$ such that
\begin{equation*}
F_i\big(t,x,0,0,0,0,u,\chi(u)\big)\leq -\frac{3}{4}\theta ,\ \mbox{for all }u\in U.
\end{equation*}%
Since $F_i(\cdot ,x,0,0,0,0,\cdot ,\cdot )$ is uniformly continuous
on $[0,T]\times U\times V,$ there exists some   $R\in(0,T-t] $ such that
\begin{equation}  \label{4.27}
F_i\big(s,x,0,0,0,0,u,\chi(u)\big)\leq -\frac{1}{2}\theta, \ \mbox{for all }u\in U%
\mbox{
and }|s-t|\leq R.
\end{equation}%
On the other hand, notice that, for $u\in\mathcal{U}_{t,\tau^\delta}$ and $\beta\in\mathcal{B}_{t,\tau^\delta}$ we have
\begin{equation*}
\begin{array}{llll}
&   Y_{\tau^\delta}^{u,\beta(u)}= \phi \big(\tau ^{\delta },X_{\tau ^{\delta }}^{t,x,i;u,\beta (u)}\big)\mathbf{1}%
_{\{N_{\tau ^{\delta }}^{t,i}=i\}}+\sum\limits_{l=1}^{m-1}W_{(l+i)\mbox{mod}%
(m)}\big(\tau ^{\delta },X_{\tau ^{\delta }}^{t,x,i;u,\beta (u)}\big)\textbf{1}%
_{\{N_{\tau ^{\delta }}^{t,i}=(l+i)\mbox{mod}(m)\}} &  \\
& \geq   W_{i}\big(\tau ^{\delta },X_{\tau ^{\delta }}^{t,x,i;u,\beta (u)}\big)\mathbf{1}%
_{\{N_{\tau ^{\delta }}^{t,i}=i\}}+\sum\limits_{l=1}^{m-1}W_{(l+i)\mbox{mod}%
(m)}\big(\tau ^{\delta },X_{\tau ^{\delta }}^{t,x,i;u,\beta (u)}\big)\mathbf{1}%
_{\{N_{\tau ^{\delta }}^{t,i}=(l+i)\mbox{mod}(m)\}} &  \\
& =   W_{N_{\tau ^{\delta }}^{t,i}}\big(\tau ^{\delta },X_{\tau ^{\delta
}}^{t,x,i;u,\beta (u)}\big). &
\end{array}%
\end{equation*}
The comparison theorem for  BSDEs with jumps (Theorem \ref{Com-Th}), Lemma \ref{l4.1} as well as $\tilde Y_{t}^{1,i,u,\beta (u)}=\phi (t,x)$  yield
\begin{equation*}
\begin{aligned}
Y_{t}^{1,i,u,\beta (u)}& = G_{t, \tau^\delta}^{t,x,i;u,\beta (u)}[Y_{\tau^\delta}^{u,\beta (u)}]- \tilde Y_{t}^{1,i,u,\beta (u)}
\geq  G_{t,\tau ^{\delta }}^{t,x,i;u,\beta (u)}[W_{N_{\tau ^{\delta
}}^{t,i}}(\tau ^{\delta },X_{\tau ^{\delta }}^{t,x,i;u,\beta (u)})]  -\phi (t,x)   \\
& = G_{t,\tau ^{\delta }}^{t,x,i;u,\beta (u)}[W_{N_{\tau ^{\delta
}}^{t,i}}(\tau ^{\delta },X_{\tau ^{\delta }}^{t,x,i;u,\beta (u)})]-W_{i}(t,x).
\end{aligned}%
\end{equation*}%
\newline
Then, the strong DPP implies that
$$
\mathop{\rm essinf}\limits_{\beta \in \mathcal{B}_{t,\tau ^{\delta }}} %
\mathop{\rm esssup}\limits_{u\in \mathcal{U}_{t,\tau ^{\delta }}}Y_{t}^{1,i,u,\beta
(u)}\geq 0,\ \mbox{P-a.s.},
$$
and in particular,
$
\mathop{\rm esssup}\limits_{u\in \mathcal{U}_{t,\tau ^{\delta
}}}Y_{t}^{1,i,u,\chi(u)}\geq 0,\ \mbox{P-a.s.}
$
Here, by putting $v_{s}(\omega )=\chi(u_{s}(\omega )),\ (s,\omega )\in [
t,T]\times \Omega ,$ it is easy to check that $v\in\mathcal{V}_{t,\tau_{\delta}}$ and we identify $\chi$ as an element of $\mathcal{B}_{t,\tau
^{\delta }}.$ Given an arbitrarily $\varepsilon >0$ we  choose $%
u^{\varepsilon }\in \mathcal{U}_{t,\tau ^{\delta }}$ (depending on $\delta>0$) such that $%
Y_{t}^{1,i,u^{\varepsilon },\chi(u^{\varepsilon })}\geq -\varepsilon \delta .$
From Lemma \ref{l4.3}, we have
\begin{equation}  \label{4.29}
{Y}_{t}^{2,i,u^{\varepsilon },\chi(u^{\varepsilon })}\geq -C\delta ^{\frac54
}-\varepsilon \delta ,\ \mbox{P-a.s.}
\end{equation}%
From    equation \eqref{equu2}, and using the fact that $(y,h,z,k)\rightarrow F_i(s,x,y,h,z,k,u,v)$ is Lipschitz, uniformly in $(s,x,u,v)$, we get
\begin{equation*}\begin{aligned}
& \! {Y}_{t}^{2,i,u^{\varepsilon },\chi(u^{\varepsilon })}
 \leq\mathbb{E}\Big[\int_t^{\tau^\delta}F_i\big(s,x,0,0,0,0,u_s^\varepsilon,\chi(u_s^\varepsilon)\big)ds \mid\mathcal {F}_t \Big]\\
& \! +C\mathbb{E}\Big[\int_t^{\tau^\delta}\Big(|Y_s^{2,i,u^\varepsilon,\chi(u^\varepsilon)}|+\sum_{l=1}^{m-1}|H_s^{2,i,u^\varepsilon,\chi(u^\varepsilon)}(l)|
+|Z_s^{2,i,u^\varepsilon,\chi(u^\varepsilon)}|+\int_E|K_s^{2,i,u^\varepsilon,\chi(u^\varepsilon)}(e)|(1\wedge|e|)\nu(de)\Big)ds \mid\mathcal {F}_t \Big]\\
&\!\leq \mathbb{E}\Big[\int_t^{\tau^\delta}F_i(s,x,0,0,0,0,u_s^\varepsilon,\chi(u_s^\varepsilon))ds \mid\mathcal {F}_t \Big]\!
+C\delta^\frac 54,
\end{aligned}\end{equation*}
 where the latter inequality comes from  Lemma \ref{l4.0000}.
Consequently, combining this  with \eqref{4.27} we obtain, for all $\delta\in(0,  R)$,
 \begin{equation}\label{ee66}
Y_t^{2,i,u^\varepsilon,\chi(u^\varepsilon)}
 \leq  -\frac12\theta\mathbb{E}\big[ \tau^\delta-t \mid\mathcal {F}_t \big]
+C\delta^\frac 54. \end{equation}
 From the definition of $\tau^\delta$, we have  $\tau^\delta= t+\delta$ on $\{N\big((t,t+\delta]\times\{l\}\big)=0,\ 1\leq l\leq m-1\}$.
Therefore,
$$\begin{aligned}
&\mathbb{E}\big[ \tau^\delta  \mid\mathcal {F}_t \big]
 \geq (t+\delta)P\big\{N\big((t,t+\delta]\times\{l\}\big)=0,\ 1\leq l\leq m-1\big\}=(t+\delta)e^{-\lambda(m-1)\delta}.
\end{aligned}$$
Setting $ \psi(\delta):=  (t+\delta)e^{-\lambda(m-1)\delta}$, we know there exists some $\hat \delta\in(0,\delta)$ such that
%
$$ \mathbb{E}\big[ \tau^\delta-t \mid\mathcal {F}_t \big]\geq \psi(\delta)-\psi(0) =\psi'(
\hat\delta)\delta=e^{-\lambda(m-1)\hat\delta}\big(1-\lambda(m-1)(t+\hat\delta)\big)\delta \geq e^{-\lambda(m-1)\delta}\big(1-\lambda(m-1)T\big)\delta(>0),$$
with $\lambda\in\Big(0,\frac1{(m-1)T} \Big)$.
By \eqref{ee66}, we get
 \begin{equation}\label{ee77}
Y_t^{2,i,u^\varepsilon,\chi(u^\varepsilon)}
 \leq   -\frac12\theta\mathbb{E}\big[ \tau^\delta-t \mid\mathcal {F}_t \big]+C\delta^\frac 54\leq -\frac12\theta e^{-\lambda(m-1)\delta}\big(1-\lambda(m-1)T\big)\delta
+C\delta^\frac 54. \end{equation}
Furthermore,  \eqref{4.29} and \eqref{ee77} imply that,
 $$-C\delta ^{\frac14%
}-\varepsilon \leq-\frac14\theta e^{-\lambda(m-1) \delta}\big(1-\lambda(m-1)T\big)+C\delta ^{\frac{1}{4}},\ \mbox{P-a.s.}$$
Letting $\delta \downarrow 0,$ and then $\varepsilon \downarrow 0,$ we
obtain  that $\theta \leq 0,$ which forms  a contradiction with $\theta>0$ in \eqref{2021011901}.

  Finally, from the above two steps we get that the lower value function $\mathbf{W}$ given in \eqref{equ3.08} is a
viscosity solution of  system \eqref{equ4.2}.

\endpf

\subsection{Uniqueness Theorem  }
In this subsection, we shall study the uniqueness of the viscosity solution  of the system of coupled Isaacs' type integral-partial differential equations \eqref{equ4.2}.
For this, we introduce the following space of continuous functions:
$$
\begin{array}
[c]{llll}%
&&\Theta=\Big\{ \varphi \in C\big([0,T]\times \mathbb{R}^n;\mathbb{R}^m\big)\mid \lim \limits_{|x|\rightarrow
\infty}|\varphi(t,x)|\exp
\big\{-\tilde{A}\big[\log\big((|x|^{2}+1)^{\frac{1}{2}}\big)\big]^{2}%
\big\}=0,\\
&&\hskip6.4cm \mbox{ uniformly in }t\in  [0,T], \mbox{ for some } \tilde{A}%
>0 \Big\}.
\end{array}
$$
Note that the growth condition of the continuous functions in $\Theta$ is weaker than the polynomial growth, but is more restrictive than the exponential growth. The main reason for introducing this condition  in $\Theta$ comes from the quadratic growth of $\sigma\sigma^*$. Moreover, it will be the  optimal condition to ensure the uniqueness of the viscosity solution of our system \eqref{equ4.2}. We refer the reader to the  corresponding  arguments in Section 3 of \cite{BBP} for more details.

\begin{lemma} \label{lemma-un}\sl  Let $\underline{\mathbf{W}}=(\underline{W}_1,\cdots,\underline{W}_m)$ be a viscosity subsolution and $\overline{\mathbf{W}}=(\overline{W}_1,\cdots,\overline{W}_m)$ be a viscosity supersolution of
\eqref{equ4.2}.
%
Then the function
$\widehat{\mathbf{W}}=(\widehat{W} _1,\cdots,\widehat{W} _m):=  \underline{\mathbf{W}}-\overline{\mathbf{W}}=(\underline{W}_1-\overline{W}_1,\cdots,\underline{W}_m-\overline{W}_m)$ is  a viscosity subsolution of the following coupled equations
\begin{equation}\label{eee999}\left\{\begin{aligned}
&\! \frac{\partial} {\partial t}\widehat{W} _i(t,x)+\sup\limits_{u\in U, v\in V}\Big\{A_{u,v}^i\widehat{W} _i(t,x) +B_{u,v}^i\widehat{W} _i(t,x)+\tilde L_i|\widehat{\mathbf{W}}(t,x)|+\tilde L_i|D\widehat{W} _i(t,x)\sigma_i(t,x,u,v)|\\
& \!\hskip 3.5cm +\tilde L_i\big(C_{u,v}^i\widehat{W} _i(t,x)\big)^+  \Big\}=0,\  (t,x,i)\in [0,T]\times \mathbb{R}^n\times \mathbf{M},\\
&\!\widehat{W} _i(T,x)=0,\ \ (x,i)\in \mathbb{R}^n\times \mathbf{M},
\end{aligned}\right.\end{equation}
where  $\tilde{L}_i$ represents the Lipschitz constant of the function $f_i(t,\cdot,\cdot,\cdot,\cdot,u,v)$ with respect to $(x,a,z,k)$.

\end{lemma}

\begin{proof}
For any $ i\in \mathbf{M}$, let $\varphi\in C_{l,b}^3\big([0,T]\times \mathbb{R}^n;\mathbb{R}\big)$, and let $(t_0,x_0)\in [0,T)\times \mathbb{R}^n$ be a maximum point of $\widehat{W} _i-\varphi$ such that $\widehat{W} _i(t_0,x_0)-\varphi(t_0,x_0)=0$.
  Without loss of generality, we assume  the maximum point $(t_0,x_0)$ is strict and global; otherwise, we can modify $\varphi$ outside a small
  neighborhood of  $\widehat{W} _i-\varphi$  if necessary.


  For each fixed $i\in \mathbf{M}$ and any given $\varepsilon>0$, define $\psi_\varepsilon(t,x,y):=  \underline{W}_i(t,x)-\overline{W}_i(t,y)-\frac{|x-y|^2}{\varepsilon^2}-\varphi(t,x)$.
We denote by $(t_\varepsilon,x_\varepsilon,y_\varepsilon)$ the global maximum point of $\psi_\varepsilon$ in $[0,T]\times \bar B_R^2$, where $\bar B_R$ is the  closed centered ball with  radius $R$ sufficiently large. Then using standard arguments (see, e.g., Proposition 3.7 in \cite{CIL}),  it follows that

\noindent (i) $(t_\varepsilon,x_\varepsilon,y_\varepsilon) \rightarrow (t_0,x_0,x_0)$, as $\varepsilon\rightarrow 0$;

\noindent (ii)  $\frac{|x_\varepsilon-y_\varepsilon|^2}{\varepsilon^2}$ is bounded and tends to $0$, when $\varepsilon\rightarrow 0$.

%
\smallskip
By Theorem 8.3 in \cite{CIL}, for any $\delta>0$, there exists $(X^\delta,Y^\delta)\in \mathbb{S}^n\times \mathbb{S}^n$, $c^\delta\in \mathbb{R}$ such that
$$\begin{aligned}
&\Big(c^\delta+\frac{\partial \varphi(t_\varepsilon,x_\varepsilon)}{\partial t},\frac{2(x_\varepsilon-y_\varepsilon)}{\varepsilon^2}+D\varphi(t_\varepsilon,x_\varepsilon),X^\delta\Big)\in \bar {\mathcal {P}}^{2,+}\underline{W}_i(t_\varepsilon,x_\varepsilon),\\
 &\Big(c^\delta ,\frac{2(x_\varepsilon-y_\varepsilon)}{\varepsilon^2} ,Y^\delta\Big)\in \bar {\mathcal {P}}^{2,-}\overline{W}_i(t_\varepsilon,y_\varepsilon),\\
\end{aligned}$$
and $$\left(
        \begin{array}{cc}
          X^\delta & 0\\
          0& -Y^\delta \\
        \end{array}
      \right) \leq A+\delta A^2,
$$
where $A=\left(
           \begin{array}{cc}
             D^2\varphi(t_\varepsilon,x_\varepsilon)+\frac2{\varepsilon^2} I& -\frac2{\varepsilon^2}I \\
             -\frac2{\varepsilon^2}I & \frac2{\varepsilon^2}I \\
           \end{array}
         \right)$,  $\bar {\mathcal {P}}^{2,+}\underline{W}_i(t_\varepsilon,x_\varepsilon)$ (resp. $\bar {\mathcal {P}}^{2,+}\overline{W}_i(t_\varepsilon,y_\varepsilon)$) is the second-order
parabolic superdifferential of $\underline{W}_i$ (resp. $\overline{W}_i$) at $(t_\varepsilon,x_\varepsilon)$ (resp. $(t_\varepsilon,y_\varepsilon)$),  $\mathbb{S}^n$ is the set of all $n\times n$ symmetric matrices, and $I$ is the unit matrix in $ \mathbb{R}^{n\times n}$.

In particular, by taking $\delta=\frac{\varepsilon^2}{4}$, we get
\begin{equation}\label{X-Y}\left(
        \begin{array}{cc}
          X^\delta & 0\\
          0& -Y^\delta \\
        \end{array}
      \right) \leq \frac4{\varepsilon^2}  \left(
           \begin{array}{cc}
             I & -I\\
             -I & I\\
           \end{array}
         \right)+ \left(
           \begin{array}{cc}
              2D^2\varphi(t_\varepsilon,x_\varepsilon)  & -\frac1{2}D^2\varphi(t_\varepsilon,x_\varepsilon)  \\
             -\frac1{2}D^2\varphi(t_\varepsilon,x_\varepsilon)  & 0 \\
           \end{array}
         \right)+ \frac{\varepsilon^2}{4}\left(
           \begin{array}{cc}
            \big( D^2\varphi(t_\varepsilon,x_\varepsilon) \big)^2 &0 \\
           0 &0\\
           \end{array}
         \right).
\end{equation}
In the following let $\delta=\frac{\varepsilon^2}{4}$, and  {put}
 $$\begin{aligned}
C_{1,\varepsilon,\delta}^{i,u,v}:=   &\int_{E_\delta} \Big(\varphi\big(t_\varepsilon,x_\varepsilon+\gamma_i(t_\varepsilon,x_\varepsilon,u,v,e)\big)-\varphi(t_\varepsilon,x_\varepsilon)\Big)\rho(x_\varepsilon,e)\nu(de)\\
 &
+\int_{E_\delta} \Big(\frac{|\gamma_i(t_\varepsilon,x_\varepsilon,u,v,e)|^2} {\varepsilon^2}+\frac{2(x_\varepsilon-y_\varepsilon)\gamma_i(t_\varepsilon,x_\varepsilon,u,v,e)} {\varepsilon^2}\Big)\rho(x_\varepsilon,e)\nu(de)\\
& +\int_{E_\delta^c}\Big( \underline{W}_i\big(t_\varepsilon,x_\varepsilon+\gamma_i(t_\varepsilon,x_\varepsilon,u,v,e)\big)-\underline{W}_i(t_\varepsilon,x_\varepsilon)\Big)\rho(x_\varepsilon,e)\nu(de),
 \end{aligned}$$
 and  $$\begin{aligned}
 C_{2,\varepsilon,\delta}^{i,u,v}:=  &\int_{E_\delta} \Big(-\frac{|\gamma_i(t_\varepsilon,y_\varepsilon,u,v,e)|^2} {\varepsilon^2}+\frac{2(x_\varepsilon-y_\varepsilon)\gamma_i(t_\varepsilon,y_\varepsilon,u,v,e)} {\varepsilon^2}\Big)\rho(y_\varepsilon,e)\nu(de)\\
& +\int_{E_\delta^c}\Big( \overline{W}_i\big(t_\varepsilon,y_\varepsilon+\gamma_i(t_\varepsilon,y_\varepsilon,u,v,e)\big)-\overline{W}_i(t_\varepsilon,y_\varepsilon)\Big)\rho(y_\varepsilon,e)\nu(de).
 \end{aligned}$$
Since $ \underline{\mathbf{W}}=(\underline{W}_1,\cdots,\underline{W}_m)$ is the  viscosity subsolution of \eqref{equ4.1}, for any $i\in\mathbf{M}$, we get
\begin{equation}\label{eee111}\begin{aligned}
&c^\delta+\frac{\partial \varphi(t_\varepsilon,x_\varepsilon)}{\partial t}
+\sup\limits_{u\in U}\inf\limits_{v\in V}\Big\{\frac 12 \tr\big(\sigma_i\sigma_i^*(t_\varepsilon,x_\varepsilon,u,v)X^\delta\big)+\Big\langle b_i(t_\varepsilon,x_\varepsilon,u,v), \frac{2(x_\varepsilon-y_\varepsilon)}{\varepsilon^2}+D\varphi(t_\varepsilon,x_\varepsilon)\Big\rangle \\
&+\int_{E_\delta} \Big(\varphi\big(t_\varepsilon,x_\varepsilon+\gamma_i(t_\varepsilon,x_\varepsilon,u,v,e)\big)-\varphi(t_\varepsilon,x_\varepsilon)
+\frac{|\gamma_i(t_\varepsilon,x_\varepsilon,u,v,e)|^2} {\varepsilon^2} - D\varphi(t_\varepsilon,x_\varepsilon) \gamma_i(t_\varepsilon,x_\varepsilon,u,v,e)\Big)\nu(de)\\
& +\int_{E_\delta^c}\Big( \underline{W}_i(t_\varepsilon,x_\varepsilon+\gamma_i(t_\varepsilon,x_\varepsilon,u,v,e))-\underline{W}_i(t_\varepsilon,x_\varepsilon)
-\Big(\frac{2(x_\varepsilon-y_\varepsilon)}{\varepsilon^2}+D\varphi(t_\varepsilon,x_\varepsilon)\Big)\gamma_i(t_\varepsilon,x_\varepsilon,u,v,e)\Big)\nu(de)\\
&+f_i\Big(t_\varepsilon,x_\varepsilon,\underline{\mathbf{W}}(t_\varepsilon,x_\varepsilon),\Big(\frac{2(x_\varepsilon-y_\varepsilon)}{\varepsilon^2}+D\varphi(t_\varepsilon,x_\varepsilon)\Big)\sigma_i(t_\varepsilon,x_\varepsilon,u,v),
C_{1,\varepsilon,\delta}^{i,u,v}, u,v\Big) \Big\}\geq 0.
\end{aligned}\end{equation}
On the other hand, as    $ \overline{\mathbf{W}}=(\overline{W}_1,\cdots,\overline{W}_m)$ is the viscosity  supersolution of \eqref{equ4.1}, from   the analogy of  Definition \ref{Def111}-(ii) but formulated for second-order super-differentials (see \cite{CIL})we have
\begin{equation}\label{eee222}\begin{aligned}
&c^\delta
+\sup\limits_{u\in U}\inf\limits_{v\in V}\Big\{\frac 12 \tr\big(\sigma_i\sigma_i^*(t_\varepsilon,y_\varepsilon,u,v)Y^\delta\big)+\Big\langle b_i(t_\varepsilon,y_\varepsilon,u,v), \frac{2(x_\varepsilon-y_\varepsilon)}{\varepsilon^2}\Big\rangle-\int_{E_\delta} \frac{|\gamma_i(t_\varepsilon,y_\varepsilon,u,v,e)|^2} {\varepsilon^2}  \nu(de)\\
& +\int_{E_\delta^c}\Big( \overline{W}_i\big(t_\varepsilon,y_\varepsilon+\gamma_i(t_\varepsilon,y_\varepsilon,u,v,e)\big)-\overline{W}_i(t_\varepsilon,y_\varepsilon)
- \frac{2(x_\varepsilon-y_\varepsilon)}{\varepsilon^2} \gamma_i(t_\varepsilon,y_\varepsilon,u,v,e)\Big)\nu(de)\\
&+f_i\Big(t_\varepsilon,y_\varepsilon,\overline{\mathbf{W}}(t_\varepsilon,y_\varepsilon), \frac{2(x_\varepsilon-y_\varepsilon)}{\varepsilon^2}
\sigma_i(t_\varepsilon,y_\varepsilon,u,v),
C_{2,\varepsilon,\delta}^{i,u,v}, u,v\Big) \Big\}\leq 0.
\end{aligned}\end{equation}
From $\eqref{eee111}$ and $\eqref{eee222}$, we have
\begin{equation}\label{eee333}\begin{aligned}
 \frac{\partial \varphi(t_\varepsilon,x_\varepsilon)}{\partial t}
+\sup\limits_{u\in U,v\in V} \Big\{A_{\varepsilon,\delta}^{i,u,v}+B_{\varepsilon,\delta}^{i,u,v}+f_{\varepsilon,\delta}^{i,u,v} \Big\}\geq 0,
\end{aligned}\end{equation}
where
$$\begin{aligned}
&  A_{\varepsilon,\delta}^{i,u,v}:=  \frac 12 \tr\big(\sigma_i\sigma_i^*(t_\varepsilon,x_\varepsilon,u,v)X^\delta\big)-\frac 12 \tr\big(\sigma_i\sigma_i^*(t_\varepsilon,y_\varepsilon,u,v)Y^\delta\big)\\
&\qquad\quad+\Big\langle b_i\big(t_\varepsilon,x_\varepsilon,u,v\big), \frac{2(x_\varepsilon-y_\varepsilon)}{\varepsilon^2}+D\varphi(t_\varepsilon,x_\varepsilon)\Big\rangle-\Big\langle b_i(t_\varepsilon,y_\varepsilon,u,v), \frac{2(x_\varepsilon-y_\varepsilon)}{\varepsilon^2}\Big\rangle\\
& \qquad  \leq  C \frac{|x_\varepsilon-y_\varepsilon|^2} {\varepsilon^2} +\frac 12 \tr\big(\sigma_i\sigma_i^*(t_\varepsilon,x_\varepsilon,u,v)D^2\varphi(t_\varepsilon,x_\varepsilon)\big)+\frac {\varepsilon^2}8\tr\big(\sigma_i\sigma_i^*(t_\varepsilon,x_\varepsilon,u,v)(D^2\varphi(t_\varepsilon,x_\varepsilon))^2\big) \\
&\qquad\quad  + \frac 12 \tr\big( \sigma_i^*(t_\varepsilon,x_\varepsilon,u,v)D^2\varphi(t_\varepsilon,x_\varepsilon)( \sigma_i(t_\varepsilon,x_\varepsilon,u,v)-\sigma_i(t_\varepsilon,y_\varepsilon,u,v))\big)\\
&\qquad\quad+\big\langle b_i(t_\varepsilon,x_\varepsilon,u,v) , D\varphi(t_\varepsilon,x_\varepsilon)\big \rangle  \quad ({\mbox{using }\eqref{X-Y}});\\
&B_{\varepsilon,\delta}^{i,u,v}:=   \int_{E_\delta} \Big(\varphi\big(t_\varepsilon,x_\varepsilon+\gamma_i(t_\varepsilon,x_\varepsilon,u,v,e)\big)-\varphi(t_\varepsilon,x_\varepsilon)- D\varphi(t_\varepsilon,x_\varepsilon) \gamma_i(t_\varepsilon,x_\varepsilon,u,v,e) \\
&\qquad\qquad\qquad
+\frac{|\gamma_i(t_\varepsilon,x_\varepsilon,u,v,e)|^2+|\gamma_i(t_\varepsilon,y_\varepsilon,u,v,e)|^2} {\varepsilon^2} \Big)\nu(de)\\
&\quad +\int_{E_\delta^c}\Big( \underline{W}_i\big(t_\varepsilon,x_\varepsilon+\gamma_i(t_\varepsilon,x_\varepsilon,u,v,e)\big)-\underline{W}_i(t_\varepsilon,x_\varepsilon)
-\Big(\frac{2(x_\varepsilon-y_\varepsilon)}{\varepsilon^2}+D\varphi(t_\varepsilon,x_\varepsilon)\Big)\gamma_i(t_\varepsilon,x_\varepsilon,u,v,e)\\
&\qquad\qquad-\overline{W}_i\big(t_\varepsilon,y_\varepsilon+\gamma_i(t_\varepsilon,y_\varepsilon,u,v,e)\big)+\overline{W}_i(t_\varepsilon,y_\varepsilon)
+ \frac{2(x_\varepsilon-y_\varepsilon)}{\varepsilon^2} \gamma_i(t_\varepsilon,y_\varepsilon,u,v,e)\Big)\nu(de),\\
&f_{\varepsilon,\delta}^{i,u,v}:= f_i\Big(t_\varepsilon,x_\varepsilon,\underline{\mathbf{W}}(t_\varepsilon,x_\varepsilon),\Big(\frac{2(x_\varepsilon-y_\varepsilon)}{\varepsilon^2}+D\varphi(t_\varepsilon,x_\varepsilon)\Big)\sigma_i(t_\varepsilon,x_\varepsilon,u,v),
 C_{1,\varepsilon,\delta}^{i,u,v}, u,v\Big)\\
&\qquad\qquad-f_i\Big(t_\varepsilon,y_\varepsilon,\overline{\mathbf{W}}(t_\varepsilon,y_\varepsilon), \frac{2(x_\varepsilon-y_\varepsilon)}{\varepsilon^2}
\sigma_i(t_\varepsilon,y_\varepsilon,u,v),  C_{2,\varepsilon,\delta}^{i,u,v}, u,v\Big).
\end{aligned}$$
From $(t_\varepsilon,x_\varepsilon,y_\varepsilon)$ being a global maximum point of $\psi_\varepsilon$, we have, for any $u\in U$, $v\in V$, $e\in E$,
  $$\psi_\varepsilon\big(t_\varepsilon,x_\varepsilon+\gamma_i(t_\varepsilon,x_\varepsilon,u,v,e),y_\varepsilon+\gamma_i(t_\varepsilon,y_\varepsilon,u,v,e)\big)\leq \psi_\varepsilon(t_\varepsilon,x_\varepsilon,y_\varepsilon),$$
which implies that
$$\begin{aligned}
 & \underline{W}_i\big(t_\varepsilon,x_\varepsilon+\gamma_i(t_\varepsilon,x_\varepsilon,u,v,e)\big)-\underline{W}_i\big(t_\varepsilon,x_\varepsilon\big)
 -\Big(\overline{W}_i\big(t_\varepsilon,y_\varepsilon+\gamma_i(t_\varepsilon,y_\varepsilon,u,v,e)\big)-\overline{W}_i\big(t_\varepsilon,y_\varepsilon)\Big)\\
&
-\frac{2(x_\varepsilon-y_\varepsilon)\big(\gamma_i(t_\varepsilon,x_\varepsilon,u,v,e)-\gamma_i(t_\varepsilon,y_\varepsilon,u,v,e)\big)}{\varepsilon^2}
-\frac{| \gamma_i(t_\varepsilon,x_\varepsilon,u,v,e)-\gamma_i(t_\varepsilon,y_\varepsilon,u,v,e)|^2}{\varepsilon^2}\\
&
-\varphi(t_\varepsilon,x_\varepsilon+\gamma_i(t_\varepsilon,x_\varepsilon,u,v,e))
+\varphi(t_\varepsilon,x_\varepsilon)\leq 0.\end{aligned}$$
Therefore,
\begin{equation}\label{2021012501}
\begin{aligned}
B_{\varepsilon,\delta}^{i,u,v}
\leq &\int_{E_\delta} \Big(\varphi\big(t_\varepsilon,x_\varepsilon+\gamma_i(t_\varepsilon,x_\varepsilon,u,v,e)\big)-\varphi(t_\varepsilon,x_\varepsilon)- D\varphi(t_\varepsilon,x_\varepsilon) \gamma_i(t_\varepsilon,x_\varepsilon,u,v,e)\\
&
\hskip1cm+\frac{|\gamma_i(t_\varepsilon,x_\varepsilon,u,v,e)|^2} {\varepsilon^2} +\frac{|\gamma_i(t_\varepsilon,y_\varepsilon,u,v,e)|^2} {\varepsilon^2}\Big)\nu(de)\\
& +\int_{E_\delta^c}\Big(\varphi\big(t_\varepsilon,x_\varepsilon+\gamma_i(t_\varepsilon,x_\varepsilon,u,v,e)\big)
-\varphi(t_\varepsilon,x_\varepsilon)-D\varphi(t_\varepsilon,x_\varepsilon)\gamma_i(t_\varepsilon,x_\varepsilon,u,v,e)\\
&\hskip1cm+\frac{| \gamma_i(t_\varepsilon,x_\varepsilon,u,v,e)-\gamma_i(t_\varepsilon,y_\varepsilon,u,v,e)|^2}{\varepsilon^2}
\Big)\nu(de)\\
 \leq&\int_{E} \Big(\varphi\big(t_\varepsilon,x_\varepsilon+\gamma_i(t_\varepsilon,x_\varepsilon,u,v,e)\big)-\varphi(t_\varepsilon,x_\varepsilon)- D\varphi(t_\varepsilon,x_\varepsilon) \gamma_i(t_\varepsilon,x_\varepsilon,u,v,e)\Big)\nu(de)\\
&
+\int_{E_\delta}\Big(\frac{|\gamma_i(t_\varepsilon,x_\varepsilon,u,v,e)|^2} {\varepsilon^2}+\frac{|\gamma_i(t_\varepsilon,y_\varepsilon,u,v,e)|^2} {\varepsilon^2}\Big)\nu(de)+\int_{E_\delta^c}\kappa(e)^2\frac{|x_\varepsilon - y_\varepsilon |^2}{\varepsilon^2}\nu(de),
\end{aligned}
\end{equation}
and
\begin{equation}\label{2021012502}
\begin{aligned}
& C_{1,\varepsilon,\delta}^{i,u,v}- C_{2,\varepsilon,\delta}^{i,u,v}  \\
\leq &\int_{E_\delta} \Big(\varphi\big(t_\varepsilon,x_\varepsilon+\gamma_i(t_\varepsilon,x_\varepsilon,u,v,e)\big)-\varphi(t_\varepsilon,x_\varepsilon)\Big)\rho(x_\varepsilon,e)\nu(de)\\
 &
+\int_{E_\delta} \Big(\frac{|\gamma_i(t_\varepsilon,x_\varepsilon,u,v,e)|^2} {\varepsilon^2}\rho(x_\varepsilon,e) +\frac{|\gamma_i(t_\varepsilon,y_\varepsilon,u,v,e)|^2} {\varepsilon^2}\rho(y_\varepsilon,e) \Big)\nu(de)\\
&
+\int_{E_\delta} \Big( \frac{2(x_\varepsilon-y_\varepsilon)\big(\gamma_i(t_\varepsilon,x_\varepsilon,u,v,e)-\gamma_i(t_\varepsilon,y_\varepsilon,u,v,e)\big)} {\varepsilon^2}\Big)\rho(x_\varepsilon,e)\nu(de)\\
&+\int_{E_\delta}\frac{2(x_\varepsilon-y_\varepsilon)\gamma_i(t_\varepsilon,y_\varepsilon,u,v,e)} {\varepsilon^2}\big(\rho(x_\varepsilon,e)-\rho(y_\varepsilon,e)\big)\nu(de)\\
 &+\int_{E_\delta^c}\Big( \frac{2(x_\varepsilon-y_\varepsilon)\big(\gamma_i(t_\varepsilon,x_\varepsilon,u,v,e)-\gamma_i(t_\varepsilon,y_\varepsilon,u,v,e)\big)}{\varepsilon^2}
+\frac{| \gamma_i(t_\varepsilon,x_\varepsilon,u,v,e)-\gamma_i(t_\varepsilon,y_\varepsilon,u,v,e)|^2}{\varepsilon^2}\\
&\hskip1cm
+\varphi\big(t_\varepsilon,x_\varepsilon+\gamma_i(t_\varepsilon,x_\varepsilon,u,v,e)\big)
-\varphi(t_\varepsilon,x_\varepsilon)\Big)\rho(x_\varepsilon,e)\nu(de)\\
&+\int_{E_\delta^c}\Big( \overline{W}_i\big(t_\varepsilon,y_\varepsilon+\gamma_i(t_\varepsilon,y_\varepsilon,u,v,e)\big)-\overline{W}_i(t_\varepsilon,y_\varepsilon)\Big)
\big(\rho(x_\varepsilon,e)-\rho(y_\varepsilon,e)\big)\nu(de)\\
\leq &\int_{E} \Big(\varphi\big(t_\varepsilon,x_\varepsilon+\gamma_i(t_\varepsilon,x_\varepsilon,u,v,e)\big)-\varphi(t_\varepsilon,x_\varepsilon)\Big)\rho(x_\varepsilon,e)\nu(de)\\
 &
+\int_{E_\delta} \Big(\frac{|\gamma_i(t_\varepsilon,x_\varepsilon,u,v,e)|^2} {\varepsilon^2}\rho(x_\varepsilon,e) +\frac{|\gamma_i(t_\varepsilon,y_\varepsilon,u,v,e)|^2} {\varepsilon^2}\rho(y_\varepsilon,e) \Big)\nu(de)\\
&
+\int_{E} \big( 2\kappa(e)+\kappa(e)^2\big) \frac{|x_\varepsilon-y_\varepsilon|^2} {\varepsilon^2} \rho(x_\varepsilon,e)\nu(de)+\int_{E_\delta}\frac{2C|x_\varepsilon-y_\varepsilon|^2\cdot|\gamma_i(t_\varepsilon,y_\varepsilon,u,v,e)|} {\varepsilon^2} \nu(de)\\
&+\int_{E_\delta^c}\Big( \overline{W}_i\big(t_\varepsilon,y_\varepsilon+\gamma_i(t_\varepsilon,y_\varepsilon,u,v,e)\big)-\overline{W}_i(t_\varepsilon,y_\varepsilon)\Big)
\big(\rho(x_\varepsilon,e)-\rho(y_\varepsilon,e)\big)\nu(de).
 \end{aligned}
 \end{equation}
 {Noting that the term
$$\int_{E_\delta} \Big(\frac{|\gamma_i(t_\varepsilon,x_\varepsilon,u,v,e)|^2} {\varepsilon^2}\rho(x_\varepsilon,e) +\frac{|\gamma_i(t_\varepsilon,y_\varepsilon,u,v,e)|^2} {\varepsilon^2}\rho(y_\varepsilon,e) \Big)\nu(de)\rightarrow 0,\ \text{as}\ \delta=\frac{\varepsilon^2}{4}\downarrow 0.$$
Indeed, from the assumption $(\mathbf{H1})$ it holds
$$\begin{aligned}
&\int_{E_\delta} \Big(\frac{|\gamma_i(t_\varepsilon,x_\varepsilon,u,v,e)|^2} {\varepsilon^2}\rho(x_\varepsilon,e) +\frac{|\gamma_i(t_\varepsilon,y_\varepsilon,u,v,e)|^2} {\varepsilon^2}\rho(y_\varepsilon,e) \Big)\nu(de)\\
& \leq    \frac{C (1+| x_\varepsilon|^2+|y_\varepsilon|^2)} {\varepsilon^2}   (1\wedge\delta) \int_{E_\delta}  \kappa(e)^2 \nu(de)\leq C (1+| x_\varepsilon|^2+|y_\varepsilon|^2)    \int_{E_\delta}  \kappa(e)^2 \nu(de).
\end{aligned}$$}
Thus, from the Lipschitz property  of $f_i$ with respect to $(x,a,z,k)$ and the fact that $f_i$ is nondecreasing   with respect to $k$,  we get
\begin{equation}\label{2021012503}
\begin{aligned}
&f_{\varepsilon,\delta}^{i,u,v}= f_i\Big(t_\varepsilon,x_\varepsilon,\underline{\mathbf{W}}(t_\varepsilon,x_\varepsilon),\big(\frac{2(x_\varepsilon-y_\varepsilon)}{\varepsilon^2}
+D\varphi(t_\varepsilon,x_\varepsilon)\big)\sigma_i(t_\varepsilon,x_\varepsilon,u,v),
C_{1,\varepsilon,\delta}^{i,u,v}, u,v\Big)\\
&\qquad\quad\ -f_i\Big(t_\varepsilon,y_\varepsilon,\overline{\mathbf{W}}(t_\varepsilon,y_\varepsilon), \frac{2(x_\varepsilon-y_\varepsilon)}{\varepsilon^2}
\sigma_i(t_\varepsilon,y_\varepsilon,u,v), C_{2,\varepsilon,\delta}^{i,u,v}, u,v\Big)\\
&\qquad \leq   \tilde L_i\Big(|x_\varepsilon-y_\varepsilon|+ |\underline{\mathbf{W}}(t_\varepsilon,x_\varepsilon)-\overline{\mathbf{W}}(t_\varepsilon,y_\varepsilon)|+2L_{\sigma_i} \frac{|x_\varepsilon-y_\varepsilon|^2}{\varepsilon^2}\\
&\qquad\qquad\
+ |D\varphi(t_\varepsilon,x_\varepsilon)|\cdot|\sigma_i(t_\varepsilon,x_\varepsilon,u,v)|+(C_{1,\varepsilon,\delta}^{i,u,v}
-C_{2,\varepsilon,\delta}^{i,u,v})^+\Big).
\end{aligned}
\end{equation}
Finally, from \eqref{eee333} and \eqref{2021012501}-\eqref{2021012503}  we
 obtain
$$\begin{aligned}
&0\leq \lim\limits_{\varepsilon\rightarrow 0} \Big\{ \frac{\partial \varphi(t_\varepsilon,x_\varepsilon)}{\partial t}
+\sup\limits_{u\in U,v\in V} \Big(A_{\varepsilon,\delta}^{i,u,v}+B_{\varepsilon,\delta}^{i,u,v}+f_{\varepsilon,\delta}^{i,u,v} \Big)\Big\}\ \ \  (\text{with}\ \delta=\frac{\varepsilon^2}{4})\\
&\leq  \frac{\partial \varphi(t_0,x_0)}{\partial t}
+\sup\limits_{u\in U,v\in V} \Big\{\frac 12 \tr\big(\sigma_i\sigma_i^*(t_0,x_0,u,v)D^2\varphi(t_0,x_0)\big) + \langle b_i(t_0,x_0,u,v), D\varphi(t_0,x_0)\rangle \\
 &\ \ \  + \int_{E} \Big(\varphi(t_0,x_0+\gamma_i(t_0,x_0,u,v,e))-\varphi(t_0,x_0)- D\varphi(t_0,x_0) \gamma_i(t_0,x_0,u,v,e)\Big)\nu(de)\\
&\ \ \  +\tilde L_i |\underline{\mathbf{W}}(t_0,x_0)-\overline{\mathbf{W}}(t_0,x_0)|
+\tilde L_i|D\varphi(t_0,x_0) \sigma_i(t_0,x_0,u,v)| \\
&\ \ \ + \tilde L_i  \Big(\int_{E } \big(\varphi(t_0,x_0+\gamma_i(t_0,x_0,u,v,e))-\varphi(t_0,x_0)\big)\rho(x_0,e)\nu(de)\Big)^+ \Big\}\\
&= \frac{\partial} {\partial t}\varphi(t_0,x_0)+\sup\limits_{u\in U, v\in V}\Big\{A_{u,v}^i\varphi(t_0,x_0) +B_{u,v}^i\varphi(t_0,x_0)+\tilde L_i|\widehat{\mathbf{W}}(t_0,x_0)|\\
& \hskip 5cm +\tilde L_i|D\varphi(t_0,x_0)\sigma_i(t_0,x_0,u,v)|+\tilde L_i \big(C_{u,v}^i\varphi(t_0,x_0)\big)^+  \Big\}.
\end{aligned}$$
Therefore, $\widehat{\mathbf{W}}$ is a viscosity subsolution of \eqref{eee999}.

\end{proof}

It is easy to check from the above proof that Lemma \ref{lemma-un} still holds if the constant $\tilde{L}_i$ appearing in equation \eqref{eee999} is substituted by the  constant $\tilde L:=\max\limits_{1\leq i\leq m}\tilde L_{i}$, which will be applied in the proof of the following Theorem \ref{th-uni}.

\begin{theorem}\label{th-uni}\sl  Assume that the assumptions $(\mathbf{H1})$-$(\mathbf{H3})$ hold. Then the system \eqref{equ4.2} has at most one viscosity solution in $\Theta$.
\end{theorem}

In order to prove this theorem, we first present an auxiliary lemma which is borrowed from Lemma 5.2 in \cite{BHL}.
\begin{lemma}\label{lemma-uni11} \sl
For any $\tilde A>0$, there exists a constant $C_1>0$ such that the function $\vartheta(t,x)=\exp\big[(C_1(T-t)+\tilde A)\Psi(x)\big]$ with $\Psi(x)=\big[\log(|x|^2+1)^\frac12+1\big]^2$,
$x\in\mathbb{R}^n$, satisfies
\begin{equation}\label{eee888} \begin{aligned}
& \frac{\partial} {\partial t}\vartheta(t,x)+\sup\limits_{u\in U, v\in V}\Big\{A_{u,v}^i\vartheta(t,x) +B_{u,v}^i\vartheta(t,x)+\tilde L \vartheta(t,x) +\tilde L\big|D \vartheta(t,x)\sigma_i(t,x,u,v)\big|\\
& \hskip 3.5cm +\tilde L\big(C_{u,v}^i\vartheta(t,x)\big)^+  \Big\}<0,\ \mbox{in } [t_1,T]\times \mathbb{R}^n,\ i\in \mathbf{M},
\end{aligned} \end{equation}
where $t_1=T-\frac{\tilde A}{C_1}.$

\end{lemma}

Notice that the function $\tilde\vartheta:=(\vartheta,\cdots,\vartheta)$ can be seen as a classical sub-solution of equation \eqref{eee999} with the sup norm $|\cdot|$ in $\mathbb{R}^m$. We now give the detailed proof of  Theorem \ref{th-uni} with this sup norm.

\begin{proof}
%
%
%
%
%

Let $ \mathbf{W}^1 =({W}_1^1,\cdots,{W}_m^1)\in \Theta$, $ {\mathbf{W}}^2=({W}_1^2,\cdots, {W}_m^2)\in \Theta$  be two viscosity solutions of \eqref{equ4.2}. Setting $\widetilde{\mathbf{W}}:= \mathbf{W}^1-\mathbf{W}^2$, we shall prove that, for any $\epsilon>0$, $$|\widetilde{{W}_i}(t,x)|\leq \epsilon\vartheta(t,x),\ (t,x)\in [0,T]\times \mathbb{R}^n,\ i\in\mathbf{M}.$$
Obviously, the function $\widetilde{\mathbf{W}}$ is continuous and belongs to the space $\Theta$, namely,  there exists some $ \tilde{A}>0$   such that
$$\lim \limits_{|x|\rightarrow
\infty} |\widetilde{\mathbf{W}}(t,x)|\exp
\Big\{-\tilde{A}\big[\log\big((|x|^{2}+1)^{\frac{1}{2}}\big)\big]^{2}\Big\}
 =0,$$
  uniformly in $t\in [0,T]$.
 As a result,  for $(t,x)\in [0,T]\times \mathbb{R}^n$, we know that $|\widetilde{\mathbf{W}}(t,x)|-\epsilon\vartheta(t,x) $ is bounded from above.
  Denote $$\Upsilon:=   \max\limits_{i\in\mathbf{M}}\sup\limits_{[t_1,T]\times \mathbb{R}^n}(|\widetilde{{W}}_i(t,x)|-\epsilon\vartheta(t,x) )e^{-\tilde L   (T-t)},$$ which can be achieved at some point $(t_0,x_0)\in  [t_1,T]\times \mathbb{R}^n$ and some $i_0\in \mathbf{M}$.
Notice that we use the sup norm $|\cdot|$ in $\mathbb{R}^m$, then
$$\Upsilon=  \max\limits_{[t_1,T]\times \mathbb{R}^n}\Big(|\widetilde{\mathbf{W}}(t,x)|-\epsilon\vartheta(t,x) \Big)e^{-\tilde L   (T-t)}, \ \mbox{and} \  |\widetilde{\mathbf{W}}(t_0,x_0)|=|\widetilde{W}_{i_0}(t_0,x_0)| \geq 0.$$

We claim that $\widetilde{W}_{i_0}(t_0,x_0)=0$. In fact, if  $ \widetilde{W}_{i_0}(t_0,x_0)  >0$ (the situation $ \widetilde{W}_{i_0}(t_0,x_0)  <0$ can be similarly considered by putting $ \widetilde{\mathbf{W}}=\mathbf{W}^2-\mathbf{W}^1$ ),  due to $\Upsilon$ is achieved at some point $(t_0,x_0)$ and some $i_0\in \mathbf{M}$, we have
$$ |\widetilde{W}_{i_0}(t,x)| -\epsilon\vartheta(t,x) \leq\big ( \widetilde{W}_{i_0}(t_0,x_0) -\epsilon\vartheta(t_0,x_0)\big )e^{-\tilde L   (t-t_0)},\ (t,x)\in [t_1,T]\times \mathbb{R}^n.$$
Let $\phi(t,x):= \epsilon\vartheta(t,x) +\big(\widetilde{W}_{i_0}(t_0,x_0) -\epsilon\vartheta(t_0,x_0) \big)e^{-\tilde L   (t-t_0)}$, $(t,x)\in [t_1,T]\times \mathbb{R}^n$, then  $$\widetilde{W}_{i_0}(t,x)-\phi(t,x)\leq 0,\  \widetilde{W}_{i_0}(t_0,x_0)=\phi(t_0,x_0),
$$
which implies that the function $\widetilde{W}_{i_0}(t,x)-\phi(t,x)$ has a maximum point at $(t_0,x_0)$. Moreover,
Lemma \ref{lemma-un} tells   that $ \widetilde{\mathbf{W}}$ is a viscosity subsolution of \eqref{eee999}, and, thus,
$$
\begin{aligned}
&0\leq \frac{\partial} {\partial t}\phi(t_0,x_0)+\sup\limits_{u\in U, v\in V}\Big\{A_{u,v}^{i_0}\phi(t_0,x_0) +B_{u,v}^{i_0}\phi(t_0,x_0)+\tilde L |\widetilde{\mathbf{W}}(t_0,x_0)|+\tilde L  |D\phi(t_0,x_0)\sigma_{i_0}(t_0,x_0,u,v)|\\
& \hskip 3.5cm +\tilde L   \big(C_{u,v}^{i_0}\phi(t_0,x_0) \big)^+  \Big\} \\
&=  \frac{\partial} {\partial t}\phi(t_0,x_0)+\sup\limits_{u\in U, v\in V}\Big\{A_{u,v}^{i_0}\phi(t_0,x_0) +B_{u,v}^{i_0}\phi(t_0,x_0)+\tilde L  \phi(t_0,x_0)+\tilde L  |D\phi(t_0,x_0)\sigma_{i_0}(t_0,x_0,u,v)|\\
& \hskip 3.5cm +\tilde L    \big(C_{u,v}^{i_0}\phi(t_0,x_0) \big)^+  \Big\}.
\end{aligned}
$$
 {From the  definition of $\phi$, for any $\epsilon>0$, we get
$$
\begin{aligned}
0\leq&  \epsilon\frac{\partial} {\partial t}\vartheta(t_0,x_0)-\tilde L \big(\widetilde{W}_{i_0}(t_0,x_0) -\epsilon\vartheta(t_0,x_0) \big)+\tilde L   \widetilde{W}_{i_0}(t_0,x_0)\\
  &+ \epsilon \sup\limits_{u\in U, v\in V}\Big\{ A_{u,v}^{i_0}\vartheta(t_0,x_0) +B_{u,v}^{i_0}\vartheta(t_0,x_0)+\tilde L  |D\vartheta(t_0,x_0)\sigma_{i_0}(t_0,x_0,u,v)|  +\tilde L   \big (C_{u,v}^{i_0}\vartheta(t_0,x_0) \big)^+  \Big\}  \\
=&  \epsilon\frac{\partial} {\partial t}\vartheta(t_0,x_0)    + \epsilon \sup\limits_{u\in U, v\in V}\Big\{ A_{u,v}^{i_0}\vartheta(t_0,x_0) +B_{u,v}^{i_0}\vartheta(t_0,x_0)+\tilde L     \vartheta(t_0,x_0)    +\tilde L  |D\vartheta(t_0,x_0)\sigma_{i_0}(t_0,x_0,u,v)|\\
& \hskip 3.5cm +\tilde L   \big (C_{u,v}^{i_0}\vartheta(t_0,x_0) \big)^+  \Big\} ,
\end{aligned} $$
which
 leads to  a contradiction with Lemma \ref{lemma-uni11}.}
%

Therefore, we get $\widetilde{W}_{i_0}(t_0,x_0)=0$, i.e., $ |\widetilde{\mathbf{W}}(t_0,x_0)|\equiv 0$. Furthermore, from the definition of $\Upsilon$,
 for any $\epsilon>0$, it is easy to see that
  $$ |\widetilde{\mathbf{W}}(t,x)|< \epsilon\vartheta(t,x), \ (t,x)\in  [ t_1 ,T]\times \mathbb{R}^n\ {(t_1=T-\frac{\tilde A}{C_1}  \mbox{as given in Lemma 4.9})}.$$
Letting $\epsilon\rightarrow0$, we get $  \widetilde{\mathbf{W}}(t,x)\equiv0 $, for any $(t,x)\in [t_1,T]\times \mathbb{R}^n $. That is,  $ {\mathbf{W}}^1(t,x)=  {\mathbf{W}}^2 (t,x),\ (t,x)\in [t_1,T]\times \mathbb{R}^n.$  By using the same argument on $[t_2,t_1]$  successively with $t_2=\big(t_1-\frac{\tilde A}{C_1}\big)^+$, and then, if $t_2>0$, on $[t_3,t_2]$ with $t_3=(t_2-\frac{\tilde A}{C_1})^+$,
 etc., we finally conclude that
  $$ {\mathbf{W}}^1(t,x)=  {\mathbf{W}}^2 (t,x),\ (t,x)\in [0,T]\times \mathbb{R}^n.$$
%


\end{proof}

We remark that the lower value function $\mathbf{W}(t,x)=\big(W_1(t,x),W_2(t,x),\cdots,W_m(t,x)\big)$ belongs to the space $\Theta$. Then, the following result is derived directly from Theorem \ref{th4.1} and Theorem \ref{th-uni}.
\begin{theorem}\sl
Under the assumptions $(\mathbf{H1})$-$(\mathbf{H3})$, when  $\lambda\in\big(0,\frac1{(m-1)T} \big)$,  the lower value function $\mathbf{W}(t,x)=\big(W_1(t,x),W_2(t,x),\cdots,W_m(t,x)\big),$ $  (t,x)\in [0,T]\times
\mathbb{R}^n$,  defined through \eqref{equ3.08} is the unique viscosity solution of the coupled system  \eqref{equ4.2} in  $\Theta$.

\end{theorem}

\subsection{The Upper Value Function Case}
We have  clarified   the relationship between the lower value function $\mathbf{W}(\cdot,\cdot)$  and the coupled HJBI system \eqref{equ4.2}. 
To make this work   complete, we remark that the upper value function $\mathbf{U}(t,x)=(U_1(t,x),U_2(t,x),\cdots,U_m(t,x))$, $(t,x)\in [0,T]\times \mathbb{R}^n$ can be also associated with {a} coupled HJBI system. Indeed, we have:

  \begin{theorem} \sl Under our assumptions $(\mathbf{H1})$-$(\mathbf{H3})$. Assuming   $\lambda\in\big(0,\frac1{(m-1)T} \big)$, the upper value function
  $\mathbf{U}(t,x)=\big(U_{1}(t,x),U_{2}(t,x),\cdots ,U_m(t,x)\big)$, $(t,x)\in [0,T]\times \mathbb{R}^n$, is a viscosity solution of the following system of HJBI equations:
\begin{equation}\label{equ5.1}
\left\{
\begin{aligned}
&\!\frac{\partial U_i}{\partial t}(t,x)+\underset{v\in V}{\mathop{\rm inf}}%
\sup\limits_{u\in U}\Big\{L_{u,v}^iU_i(t,x)
+f_i\big(t,x,\mathbf{U}(t,x), D
U_i(t,x)\sigma_i(t,x,u,v), C_{u,v}^iU_i(t,x),u,v\big)\Big\}=0, \\
&\!U_{i}(T,x)=\Phi_i (x), \ (t,x,i)\in [0,T)\times \mathbb{R}^m\times \mathbf{M}.
\end{aligned}
\right.
\end{equation}%
In addition, the function $\mathbf{U}(\cdot,\cdot)$ is the unique viscosity solution  in  $\Theta$ for the system  \eqref{equ5.1}.
   \end{theorem}
As a byproduct, we obtain that the value of the related game problem exists under the Isaacs' condition.
\begin{theorem} \sl (\textbf{The value of the game})
 If for all $(t,x,i)\in [0,T]\times \mathbb{R}^n\times  \mathbf{M}$, the following Isaacs' condition
 $$\begin{aligned}& \sup\limits_{u\in U}\mathop{\rm inf}%
\limits_{v\in V}\Big\{L_{u,v}^iU_i(t,x)
+f_i\big(t,x,\mathbf{U}(t,x), D
U_i(t,x)\sigma_i(t,x,u,v), C_{u,v}^iU_i(t,x),u,v\big)\Big\}&  &  &  \\
& =\mathop{\rm inf}%
\limits_{v\in V}\sup\limits_{u\in U} \Big\{L_{u,v}^iU_i(t,x)
+f_i\big(t,x,\mathbf{U}(t,x), D
U_i(t,x)\sigma_i(t,x,u,v), C_{u,v}^iU_i(t,x),u,v\big)\Big\}
\end{aligned} $$
 holds true, then we have $\mathbf{W}(t,x)=\mathbf{U}(t,x)$.
That is to say, there exists a value for our stochastic differential games under the Isaacs' condition.

  \end{theorem}

\begin{proof}

 When the Isaacs' condition holds true,   equations \eqref{equ4.2} and \eqref{equ5.1} coincide. The uniqueness of viscosity solution  implies  that $\mathbf{W}(t,x)=\mathbf{U}(t,x)$, $(t,x)\in[0,T]\times \mathbb{R}^n$.
\end{proof}

\section{Appendix: Proof of Theorem \ref{SDPP}}

To simplify the notations, we denote by $\mathscr{W}_i(t,x)$ the right-hand side of equality \eqref{equ-SDPP}, that is,
 $$\mathscr{W}_i(t,x):= \mathop{\rm essinf}\limits_{\beta
\in \mathcal{B}_{t,\tau}}\mathop{\rm esssup}\limits_{u\in \mathcal{U}%
_{t,\tau}}G_{t,{\tau}}^{t,x,i;u,\beta(u)} \big[W_{N_\tau^{t,i}}\big(\tau,
{X}_{\tau}^{t,x,i;u,\beta(u)}\big)\big],\ i\in\mathbf{M},\ (t,x)\in[0,T]\times\mathbb{R}^n. $$
Similar to the fact that  $W_i(t,x)$  is deterministic (see, Proposition \ref{Pro---3.1}), we can check that also  $\mathscr{W}_i(t,x)$ is  deterministic, for all $i\in \mathbf{M}$, $(t,x)\in[0,T]\times\mathbb{R}^n$. In order to show the equality between  $\mathscr{W}_i (\cdot,\cdot)$
and $W_i(\cdot,\cdot)$, we split the proof into the following two steps.
%

\medskip

\noindent \emph { Step 1.}
For all $i\in \mathbf{M}$, $(t,x)\in[0,T]\times\mathbb{R}^n$,  $\mathscr{W}_i (t,x)\leq W_i(t,x),$ $P$-a.s.\\
For any fixed $\beta \in \mathcal{B}_{t,T}$,  any given
$u_{2}\in \mathcal{U}_{\tau,T}$, we define, for $u_{1}
\in \mathcal{U}_{t,\tau}$,
$$
\beta_{1}(u_{1}):=   \beta(u_{1}\oplus u_{2})|_{[t,\tau]},\ \mbox{ with\  } u_{1}\oplus u_{2}:=   u_{1}\mathbf{1}_{[t,\tau]}+u_{2}\mathbf{1}%
_{(\tau,T]}\in\mathcal{U}_{t,T}.
$$
Clearly,  the nonanticipativity    of $\beta$ implies that  $\beta_{1}\in \mathcal{B}_{t,\tau}$ and so $\beta_{1}$ is
independent of the special choice of $u_{2}\in \mathcal{U}_{\tau,T}%
$.
Using  the definition of $\mathscr{W}_i (t,x)$,   similar to the proof of (A.20) in \cite{BHL}, we know that, for any $\varepsilon>0$ and $\beta_1\in \mathcal{B}_{t,\tau}$, there exists $u_1^\varepsilon\in \mathcal {U}_{t,\tau}$, such that
\begin{equation}%
\begin{array}
[c]{l}%
\mathscr{W}_i (t,x)\leq  G_{t,\tau}^{t,x,i;u_1^\varepsilon,\beta_1(u_1^\varepsilon)}\big[W_{N_{\tau}^{t,i}}\big(\tau, {X}_{\tau}^{t,x,i;u_1^\varepsilon,\beta_1(u_1^\varepsilon)}\big)\big] +\varepsilon,\ P\text{-a.s.}
\end{array}
\label{ee7.1}%
\end{equation}
Now we focus on $W_{N_{\tau}^{t,i}}\big(\tau,{X}_{\tau}^{t,x,i;u_{1}^\varepsilon,\beta_{1}(u_{1}^\varepsilon)}\big)$. Define
$$\beta_2^\varepsilon(u_2):=   \beta(u_1^\varepsilon\oplus
u_{2})|_{[\tau,T]},\ u_{2}\in \mathcal{U}_{\tau,T}.$$
Since $\beta\in\mathcal{B}_{t,T}$, we have $\beta_2^\varepsilon\in \mathcal{B}_{\tau,T}$.
 Taking  $\eta=  X_\tau^{t,x,i;u_1^\varepsilon,\beta_1(u_1^\varepsilon)}$ in  Proposition \ref{prop3.3}, we have
\begin{equation*}
\begin{aligned}
W_{N_\tau^{t,i}}\big(\tau,  X_\tau^{t,x,i;u_1^\varepsilon,\beta_1(u_1^\varepsilon)}\big)=&\mathop{\rm essinf}\limits_{\beta_2\in\mathcal{B}%
_{\tau,T}} \mathop{\rm esssup}\limits_{u_2\in\mathcal{U}_{\tau,T}}J\big(
\tau,  X_\tau^{t,x,i;u_1^\varepsilon,\beta_1(u_1^\varepsilon)},N_\tau^{t,i};u_2,\beta_2(u_2)\big)\\
\leq& \mathop{\rm esssup}\limits_{u_2\in\mathcal{U}_{\tau,T}}J\big(
\tau,  X_\tau^{t,x,i;u_1^\varepsilon,\beta_1(u_1^\varepsilon)},N_\tau^{t,i};u_2,\beta_2^\varepsilon (u_2)\big),\  P\text{-a.s.}
\end{aligned}
\end{equation*}
Again, a standard argument implies that there exists $u_2^\varepsilon\in\mathcal{U}_{\tau,T}$, such that
\begin{equation}\label{equ031402}
\begin{aligned}
W_{N_\tau^{t,i}}\big(\tau,  X_\tau^{t,x,i;u_1^\varepsilon,\beta_1(u_1^\varepsilon)}\big)
 \leq& J\big(
\tau, X_\tau^{t,x,i;u_1^\varepsilon,\beta_1(u_1^\varepsilon)},N_\tau^{t,i};u_2^\varepsilon,\beta_2^\varepsilon(u_2^\varepsilon)\big)+\varepsilon\\
=& J\big(
\tau,  X_\tau^{t,x,i;u^\varepsilon,\beta(u^\varepsilon)},N_\tau^{t,i};u^\varepsilon,\beta(u^\varepsilon)\big)+\varepsilon,\ P\text{-a.s.},
\end{aligned}
\end{equation}
with $u^\varepsilon:=   u_1^\varepsilon\oplus u_2^\varepsilon\in\mathcal{U}_{t,T}$.
Therefore, from \eqref{ee7.1}, \eqref{equ031402}, Lemma \ref{l2} and the comparison theorem for  BSDE with jumps  (Theorem \ref{Com-Th}) we obtain
\begin{equation}\label{equ031403}
\begin{aligned}
\mathscr{W}_i (t,x)\leq&  G_{t,{\tau}}^{t,x,i;u_{1}^\varepsilon,\beta_{1}(u_{1}^\varepsilon)}\big[J\big(
\tau,  X_\tau^{t,x,i;u^\varepsilon,\beta(u^\varepsilon)},N_\tau^{t,i};u^\varepsilon,\beta (u^\varepsilon)\big)+\varepsilon\big]+\varepsilon\\
\leq& G_{t,{\tau}}^{t,x,i;u^\varepsilon,\beta(u^\varepsilon)}\big[J\big(
\tau,X_\tau^{t,x,i;u^\varepsilon,\beta(u^\varepsilon)},N_\tau^{t,i};u^\varepsilon,\beta (u^\varepsilon)\big)\big]+C\varepsilon,\
 P\text{-a.s.}
\end{aligned}
\end{equation}
 Finally, from \eqref{equ031403} and the definition of our backward semigroup,  we get
\begin{equation}\nonumber
\begin{aligned}
\mathscr{W}_i (t,x)
\leq& G_{t,{\tau}}^{t,x,i;u^\varepsilon,\beta(u^\varepsilon)}\big[J\big(
\tau,X_\tau^{t,x,i;u^\varepsilon,\beta(u^\varepsilon)},N_\tau^{t,i};u^\varepsilon,\beta (u^\varepsilon)\big)\big]+C\varepsilon \\
=& G_{t,{\tau}}^{t,x,i;u^\varepsilon,\beta(u^\varepsilon)}\big[Y_\tau^{t,x,i;u^\varepsilon,\beta(u^\varepsilon)}\big]+C\varepsilon\\
=& Y_t^{t,x,i;u^\varepsilon,\beta(u^\varepsilon)}+C\varepsilon=J(t,x,i;u^\varepsilon,\beta(u^\varepsilon))+C\varepsilon\\
\leq &  \mathop{\rm esssup}_{u\in \mathcal{U}_{t,T}}J(t,x,i;u,\beta(u))+C\varepsilon,\
 P\text{-a.s.}
\end{aligned}
\end{equation}
From the arbitrary choice of    $\beta \in \mathcal{B}_{t,T}$, we conclude that
\begin{equation}
\mathscr{W}_i (t,x)\leq \mathop{\rm essinf}_{\beta \in \mathcal{B}_{t,T}%
} \mathop{\rm esssup}_{u\in \mathcal{U}_{t,T}}J(t,x,i;u,\beta(u))+C\varepsilon=W_i(t,x)+C\varepsilon.
\end{equation}
Letting $\varepsilon \downarrow 0$, we get
$\mathscr{W}_i (t,x)\leq W_i(t,x)$, $i\in \mathbf{M}$, $(t,x)\in[0,T]\times\mathbb{R}^n$.

 \medskip

\noindent \emph{ Step 2.}
For all $i\in \mathbf{M}$, $(t,x)\in[0,T]\times\mathbb{R}^n$,  $\mathscr{W}_i (t,x)\geq W_i(t,x),$ $P$-a.s.\\
 Given any $\varepsilon>0$, similar  to the proof of (A.21) in \cite{BHL},  there exists $\beta_1^\varepsilon\in\mathcal{B}_{t,\tau}$ such that for all $u_{1}\in \mathcal{U}_{t,\tau}$,
\begin{equation}\label{ee6.3.1}
\begin{aligned}
\mathscr{W}_i (t,x)  \geq G_{t,{\tau}}^{t,x,i;u_{1},\beta_{1}^\varepsilon(u_{1})}\big[W_{N_{\tau}^{t,i}}\big(\tau,{X}_{\tau}^{t,x,i;u_{1},\beta_{1}^\varepsilon(u_{1})}\big)\big]-\varepsilon,\ P\text{-a.s.}
\end{aligned}
\end{equation}
Now we estimate $W_{N_{\tau}^{t,i}}\big(\tau, {X}_{\tau}^{t,x,i;u_{1},\beta_{1}^\varepsilon(u_{1})}\big)$.  Choosing $\eta=  X_\tau^{t,x,i;u_1,\beta_1^\varepsilon(u_1)}$ in Proposition \ref{prop3.3}, we have
\begin{equation*}
\begin{aligned}
W_{N_\tau^{t,i}}\big(\tau, X_\tau^{t,x,i;u_1,\beta_1^\varepsilon(u_1)}\big)=&\mathop{\rm essinf}\limits_{\beta_2\in\mathcal{B}%
_{\tau,T}} \mathop{\rm esssup}\limits_{u_2\in\mathcal{U}_{\tau,T}}J\big(
\tau, X_\tau^{t,x,i;u_1,\beta_1^\varepsilon(u_1)},N_\tau^{t,i};u_2,\beta_2(u_2)\big),\  P\text{-a.s.}
\end{aligned}
\end{equation*}
Using a similar argument, we deduce that there exists $\beta_2^\varepsilon\in\mathcal{B}_{\tau,T}$ (depending on $u_1$) such that for all $u_2\in\mathcal{U}_{\tau,T}$,
\begin{equation}\label{equ031406}
\begin{aligned}
W_{N_\tau^{t,i}}\big(\tau,  X_\tau^{t,x,i;u_1,\beta_1^\varepsilon(u_1)}\big)
\geq J\big(
\tau, X_\tau^{t,x,i;u_1,\beta_1^\varepsilon(u_1)},N_\tau^{t,i};u_2,\beta_2^\varepsilon(u_2)\big)-\varepsilon,\  P\text{-a.s.}
\end{aligned}
\end{equation}
For  $u\in\mathcal{U}_{t,T}$, we define
$ \beta^\varepsilon(u):=   \beta^\varepsilon_1(u_1)\oplus \beta^\varepsilon_2(u_2), \mbox{ with }u_1=u|_{[t,\tau]},\ u_2=u|_{(\tau,T]}.$
 Then, by verifying Definition \ref{Def-2}, we get   $\beta^\varepsilon\in\mathcal{B}_{t,T}$.
Let $u\in\mathcal{U}_{t,T}$ be arbitrarily given and decomposed into $u_1=u|_{[t,\tau]}\in\mathcal{U}_{t,\tau}$ and $u_2=u|_{(\tau,T]}\in\mathcal{U}_{\tau,T}$. Then from \eqref{ee6.3.1}, \eqref{equ031406}, Lemma \ref{l2}, we get
 \begin{equation}\nonumber
\begin{aligned}
\mathscr{W}_i (t,x)  \geq& G_{t,{\tau}}^{t,x,i;u_{1},\beta_{1}^\varepsilon(u_{1})}\big[J\big(
\tau,  X_\tau^{t,x,i;u_1,\beta_1^\varepsilon(u_1)},N_\tau^{t,i};u_2,\beta_2^\varepsilon(u_2)\big)-\varepsilon\big]-\varepsilon\\
\geq& G_{t,{\tau}}^{t,x,i;u_{1},\beta_{1}^\varepsilon(u_{1})}\big[J\big(
\tau, X_\tau^{t,x,i;u_1,\beta_1^\varepsilon(u_1)},N_\tau^{t,i};u_2,\beta_2^\varepsilon(u_2)\big)\big]-C\varepsilon\\
=& G_{t,{\tau}}^{t,x,i;u,\beta^\varepsilon(u)}\big[Y_\tau^{t,x,i;u,\beta^\varepsilon(u)}\big]-C\varepsilon\\
=& Y_t^{t,x,i;u,\beta^\varepsilon(u)}-C\varepsilon=J(t,x,i;u,\beta^\varepsilon(u))-C\varepsilon,\
P\text{-a.s.}
\end{aligned}
\end{equation}
Consequently,
\begin{equation*}
\mathscr{W}_i (t,x)\geq \mathop{\rm esssup}\limits_{u\in \mathcal{U}_{t,T}%
}J(t,x,i;u,\beta^{\varepsilon}(u))-C\varepsilon \geq
\mathop{\rm essinf}\limits_{\beta \in \mathcal{B}_{t,T}}%
\mathop{\rm esssup}\limits_{u\in \mathcal{U}_{t,T}}J(t,x,i;u,\beta
(u))-C\varepsilon=W_i(t,x)-C\varepsilon,\  \  \mbox{P-a.s.}\label{ee7.12}%
\end{equation*}
Finally, letting $\varepsilon \downarrow0$, we get $\mathscr{W}_i (t,x)\geq W_i(t,x).$


%

\end{document}